\pdfoutput=1
\documentclass[11pt,letterpaper]{amsart}
\usepackage[english]{babel}
\usepackage{amsfonts}
\usepackage{amsthm}
\usepackage{amsmath}
\usepackage{amssymb}
\usepackage{mathrsfs} 
\usepackage{enumerate}
\usepackage{mathtools}
\usepackage{tikz}
\usepackage{tikz-cd}
\usepackage[all]{xy}
\usepackage{graphicx}
\usepackage{amsmath}
\usepackage{adjustbox}
\usepackage{graphicx}
\usepackage{extpfeil}
\usepackage{comment}
\usepackage{mathabx}
\usepackage{float}
\usepackage[labelformat=empty]{caption}
\usepackage[toc]{appendix}
\usepackage{hyperref}
\usepackage{resizegather}
\usepackage[activate={true, nocompatibility}, final, tracking=true, kerning=true, spacing=true, factor=1100, stretch=10, shrink=10]{microtype}
\usepackage{dsfont}
\usetikzlibrary{matrix}
\usepackage{cleveref}
\usepackage{ulem}

\theoremstyle{plain}
\newtheorem{thm}{Theorem}[section]

\newtheorem{defn}[thm]{Definition}
\newtheorem{lemma}[thm]{Lemma}

\newtheorem{prop}[thm]{Proposition}
\newtheorem{remark}[thm]{Remark}

\newtheorem*{thm*}{Main Theorem}
\newtheorem{notn}[thm]{Notation}
\newtheorem{context}[thm]{Context}

\newtheorem{question}[thm]{Question}

\usepackage{color} 
\definecolor{darkgreen}{rgb}{0.0, 0.7, 0.0}
\newenvironment{??}{\noindent \color{darkgreen}{\bf ???:} \footnotesize}{}
\definecolor{cyan}{cmyk}{1,0,0,0}

\newcommand{\bdg}{\begin{dg}}

\tikzset{
  symbol/.style={
    draw=none,
    every to/.append style={
      edge node={node [sloped, allow upside down, auto=false]{$#1$}}}
  }
}

\microtypecontext{spacing=nonfrench}

\newcommand{\cC}{\mathcal{C}}

\newcommand{\cE}{\mathcal{E}}

\newcommand{\cG}{\mathcal{G}}
\newcommand{\cH}{\mathcal{H}}
\newcommand{\cI}{\mathcal{I}}
\newcommand{\cJ}{\mathcal{J}}

\newcommand{\cM}{\mathcal{M}}

\newcommand{\cO}{\mathcal{O}}

\newcommand{\cS}{\mathcal{S}}

\newcommand{\cZ}{\mathcal{Z}}

\newcommand{\sA}{\mathscr{A}}

\newcommand{\sD}{\mathscr{D}}

\newcommand{\sG}{\mathscr{G}}
\newcommand{\sH}{\mathscr{H}}

\newcommand{\sL}{\mathscr{L}}
\newcommand{\sM}{\mathscr{M}}

\newcommand{\sP}{\mathscr{P}}

\newcommand{\sS}{\mathscr{S}}
\newcommand{\sT}{\mathscr{T}}

\newcommand{\sX}{\mathscr{X}}

\newcommand{\Fb}{\mathfrak{b}}
\newcommand{\Fc}{\mathfrak{c}}

\newcommand{\Ff}{\mathfrak{f}}
\newcommand{\Fg}{\mathfrak{g}}

\newcommand{\Fl}{\mathfrak{l}}

\newcommand{\Fo}{\mathfrak{o}}
\newcommand{\Fp}{\mathfrak{p}}

\newcommand{\Fs}{\mathfrak{s}}
\newcommand{\Ft}{\mathfrak{t}}

\newcommand{\Fas}{\mathfrak{a}\mathfrak{s}}
\newcommand{\Fgl}{\Fg\Fl}
\newcommand{\Fsl}{\Fs\Fl}

\newcommand{\bA}{\mathbb{A}}
\newcommand{\bC}{\mathbb{C}}

\newcommand{\bG}{\mathbb{G}}

\newcommand{\bF}{\mathbb{F}}
\newcommand{\bQ}{\mathbb{Q}}

\newcommand{\bN}{\mathbb{N}}

\newcommand{\bZ}{\mathbb{Z}}

\newcommand{\oql}{\overline{\bQ_{\ell}}}

\usetikzlibrary{calc,decorations.pathmorphing,shapes}

\newcounter{sarrow}

\newcommand{\ov}{\overline}

\usepackage{color} 

\newcommand{\wt}{\widetilde}

\newtheorem{cor}[thm]{Corollary}

      \newcommand{\git}{/\! /}

\usepackage[toc]{appendix}

\newcommand{\IC}{\cI\cC}

\usepackage{bbm}
\usepackage{amssymb}

\hypersetup{
    colorlinks=true,
    linkcolor=blue,
    filecolor=magenta,      
    urlcolor=cyan,
    pdftitle={Overleaf Example},
    pdfpagemode=FullScreen,
    }

\begin{document}

\title{Logarithmic Non-Abelian Hodge Theory \\ for curves  in prime characteristic}
\author{Mark Andrea de Cataldo and Siqing Zhang}
\date{}

\maketitle

\begin{abstract}
For a curve $C$ and a reductive group $G$ in prime characteristic, we relate the de Rham moduli of logarithmic $G$-connections on $C$ to the Dolbeault moduli of 
logarithmic $G$-Higgs bundles on the Frobenius twist of $C$. 
We name this result the Log-$p$-NAHT.
It is a logarithmic version of Chen-Zhu's characteristic $p$ Non Abelian Hodge Theorem ($p$-NAHT) in \cite{chen-zhu}. 
In contrast to the no pole case, the two moduli stacks in the log case are not isomorphic \'etale locally over the Hitchin base.
Instead, they differ by an  Artin-Schreier type Galois cover of the base.
In contrast to the case over the complex numbers, where some parabolic/parahoric data are needed to specify the boundary behavior of the tame harmonic metrics, no parabolic/parahoric data are needed in Log-$p$-NAHT.
We also establish a semistable version of the Log-$p$-NAHT, and deduce several geometric and cohomological consequences.
In particular, when $G=GL_r$, the Log-$p$-NAHT induces an embedding of the intersection cohomology of the degree $d$ Dolbeault moduli to that of the degree $pd$ de Rham moduli, and the embedding is an isomorphism when $r$ is coprime to $d$ and $p>r$.
\end{abstract}

\tableofcontents

\section{Introduction}

Let $X$ be a compact Riemann surface.
The Non Abelian Hodge Theorem (NAHT) of C. Simpson et. al. establishes a
canonical homeomorphism between the moduli spaces of semistable
Higgs bundles of degree zero  and   algebraic flat connections on $X$. 
If we allow log-poles at some fixed punctures on $X$,
Simpson has established  an equivalence of categories relating  the {\it parabolic} variants of such objects on $X$; this result, and its vast generalization due to T. Mochizuki is usually referred to as the Log-NAHT. However, it
is well-known that there is no natural version of the Log-NAHT relating semistable  Higgs bundles
with log poles to flat connections with log poles in the form of a natural bijection: one needs to fix the behaviour at the punctures in terms of parabolic data.

There is a natural analogue
of the NAHT result for a  nonsingular projective irreducible curve $C$ over an algebraically closed field of positive characteristic: this analogue is  called the $p$-NAHT and it relates 
Higgs bundles on the Frobenius twist $C'$ of $C$ and flat connections
on the curve $C$: the two moduli stacks are \'etale
locally equivalent over the Hitchin base, target of the Hitchin morphism from the Higgs moduli stack and of the de Rham-Hitchin morphism from the moduli stack of conenctions.

In this paper, we establish a non parabolic/parahoric Log-$p$-NAHT for curves $C$ over algebraically closed fields of characteristic $p$.
The first novelty is that the Log-$p$-NAHT is about Higgs bundles/flat connections without the need of additional structure of parabolic/parahoric data.  Unlike the case where there are no poles, the fibers of the  de Rham-Hitchin morphism are disconnected, which prohibits an \'etale local equivalence over the Hitchin base as in the case with no poles.  The second novelty is that 
then
the Log-$p$-NAHT takes place  in the form of an \'etale local equivalence over a Galois branched cover
of the Hitchin base, obtained by means of an Artin-Schreier-type morphism arising from the relation between the residues of
the log-connections and the residues of their $p$-curvatures.

In the remainder of this introduction, we provide context and summarize the results of this paper.

NAHT.
For a complex projective smooth connected curve $C$ and a reductive group $G$, the Non Abelian Hodge Theory (NAHT) as in \cite{hitchin1987self, Simpson-repnI, simpson-repnII} shows that the Dolbeault moduli space $M_{Dol}^{ss}(C)$ of semistable degree zero $G$-Higgs bundles, the de Rham moduli space $M_{dR}(C)$ of $G$-connections, and the character variety $M_B(C)$, are canonically diffeomorphic.
NAHT has wide applications in areas such as mirror symmetry \cite{hausel-thaddeus-04, hausel2005mirror}, geometric Langlands program \cite{donagi2006langlands}, and the $P=W$ phenomenon   \cite{de2012topology, de2022abeliansurface, maulik2024p, hausel2022p}; see the Seminaire Bourbaki \cite{hoskins2023two}.

$p$-NAHT.
The characteristic $p>0$ analogue of NAHT, $p$-NAHT, states that the Dolbeault moduli stack $\sM_{Dol}(C')$ of Higgs bundles on the Frobenius twist $C'$ of $C$, and the de Rham moduli stack $\sM_{dR}(C)$,   sit over the same Hitchin base $A$, admit an action of a natural Picard stack $\sP$ over $A$, and they differ by a twist of a $\sP$-torsor $\sH$, i.e. there is a canonical $A$-isomorphism $\sH\times^{\sP}\sM_{Dol}(C')\xrightarrow{\sim}\sM_{dR}(C)$, \cite{ov07, bezrukavnikov2007geometric, groechenig-moduli-flat-connections, chen-zhu}.
The $p$-NAHT is related to prismatization \cite{bhatt2022prismatization, ogus2022crystalline}, and has applications to the geometric Langlands program \cite{chzh17}, Bogomolov-Miyaoka type inequalities \cite{langer2015bogomolov}, Simpson motivicity conjecture \cite{esnault2020rigid}, and the $P=W$ phenomenon \cite{dCMSZ}.

Log-NAHT.
The quest  for a logarithmic version of NAHT (Log-NAHT) started with \cite{simpson1990harmonic} and  has been carried forward by many  authors, see \cite{mochizuki2004kobayashi, boalch2018wild, huang2022tame} and references therein. 
 As it turns out,
the correct object of study are the filtered, or rather the parahoric, analogues
of  logarithmic Higgs bundles and connections, i.e. the boundary behavior near the log-poles should be fixed.
To the best of our knowledge, if the boundary behavior is not fixed, a Log-NAHT seems unlikely, see \Cref{rmk: parabolic}.
The Log-NAHT has been used crucially in the monograph on Hecke eigensheaves \cite{donagi2024twistor}.

Parahoric Log-$p$-NAHT.
The  parahoric Log-$p$-NAHT has been pursued  \cite{li2024tame}, where a kind of Log-$p$-NAHT relating certain parahoric connections and certain parahoric Higgs bundles has been established.

The search for a Log-$p$-NAHT.
 On the other hand, as it is suggested by some results in \cite{schepler2005logarithmic, shen2018tamely}, a  non-parahoric version of Log-$p$-NAHT, connecting  logarithmic Higgs bundles and  logarithmic connections, might be possible. 
Instead of  keeping track of  the  parahoric data at  the punctures,  one may merely keep track of  the eigenvalues/adjoint orbits.
In the case of fixed residues, this non-parahoric  Log-$p$-NAHT in the $GL_r$-case has  been established and applied to the geometric Langlands program in \cite{shen2018tamely}.

As mentioned above, in this paper we establish:
(the non-parahoric version of) the Log-$p$-NAHT; a refinement of 
the Log-$p$-NAHT  in the context of semistable objects;  geometric and cohomological consequences for the corresponding moduli spaces.

Let $C$ be a nonsingular integral and projective curve over an algebraically closed field of positive characteristic $p>0$. Let $D>0$ be a reduced strictly effective divisor on $C$ (the divisor of poles of the log-connections). Let $C'$ be the Frobenius twist of $C$ and let $D'>0$ be divisor on $C'$ pull-back of $D$ via the natural
scheme isomorphism $C'\stackrel{\sim}\to C$. Recall that the fundamental group $\pi_1(G)$ is in natural bijection with the connected components of $Bun_G(C)$, of the Higgs moduli stack, and of the  Higgs moduli space of semistable objects.
 For $G=GL_r$, the element $d \in \pi_1(G)$ is naturally identified with the degree of the underlying vector bundle.

The de Rham-Hitchin morphism $h_{dR}: \sM_{dR}(C,D)\to A(\omega_{C'}(D'))$, factors as $\sM_{dR}(C,D)\to \Fc A\to A(\omega_{C'}(D'))$,
where $\Fc A \to A(\omega_{C'}(D'))$ is the branched Galois cover defined in  \Cref{defn: base cA} by means of the constructions in
\Cref{subs: residues}, culminating in diagram \eqref{diag: hdr fiber not conn}.
The morphism $\Fc A\to A(\omega_{C'}(D'))$ is the Artin-Schreier-type morphism mentioned above as a main novelty introduced in this document.

The main result of this paper is the following theorem that summarizes 
 some of the results proved here.

\begin{thm}[{\bf (Semistable) Log-$p$-NAHT}]\label{thm: 1st main thm}
The following statements hold:

 {\rm {\bf Log-$p$-NAHT} (\Cref{thm: main thm general g}, \Cref{prop: H pstorsor}, \Cref{prop: pstorsor ho} and \Cref{cor: not surj})}.
 Assume that $(p,G)$ satisfies the assumption in \Cref{context: G}.
Let $\wt{\sP}$ (resp. $\wt{\sM}_{Dol}(C',D')$) be the base change of $\sP$ (resp. $\sM_{Dol}(C',D')$) to $\Fc A$. 
    
    Then, there is a non-empty $\wt{\sP}$-pseudo-torsor  $\sH$ smooth over $\Fc A$, and a morphism of $\Fc A$-stacks
\begin{equation}\label{eq: thm 1.1}
            \mathbf{n}: \mathscr{H}\times^{\widetilde{\mathscr{P}}} \widetilde{\mathscr M}_{Dol}(C',D')\to \mathscr{M}_{dR}(C,D), 
        \end{equation}
which is an isomorphism over the nonempty open image of $\sH\to \Fc A$; this open image surjects onto $A(\omega_{C'}(D'))$.

{\rm {\bf Semistable Log-$p$-NAHT} (\Cref{thm: main sst})}. Assume that $(p,G)$ is good enough as in \Cref{defn: good enough}. 
Let $\wt{\sP^o}$ be the base change of the neutral component $\sP^o$ of $\sP$ to $\Fc A$. 
There is a nonempty substack $\sH^o\hookrightarrow\sH$ which is a $\wt{\sP^o}$-pseudo-torsor over $\Fc A$. 
For each degree $d\in \pi_1(G)$, the morphism $\mathbf{n}$ restricts to a morphism of $\Fc A$-stacks
 \begin{equation}\label{eq: sst nss in main thm1}
        \mathbf{n}^{ss}: \sH^o\times^{\wt{\sP}^{o}} \wt{\sM}_{Dol}^{ss}(C',D',d) \to \sM_{dR}^{ss}(C, D,pd),
    \end{equation}
    which is an isomorphism over the open image of $\sH^o\to \Fc A$. 
    If $G=GL_r$ or $G$ satisfies (LH) as in \Cref{def: low height}, then each stack in the display above admits a quasi-projective adequate moduli space, and $\mathbf{n}^{ss}$ descends to a morphism between the moduli spaces.

\end{thm}

We show in \Cref{lemma: constr m} that, when $\deg(D)$ is even, the $\Fc A$-stack $\sH$ is isomorphic to the substack $\sM_{dR}^{reg}(C,D)\subset \sM_{dR}(C,D)$ given by log-connections whose $p$-curvature is regular in the Lie-theoretic sense.

 In the case $G=GL_r$, we also calculate how the degrees of the objects change under
the morphism $\mathbf{n}$, giving rise to a more refined version of \eqref{eq: thm 1.1}; see \Cref{thm: main thm glr}.

In the $G=GL_{r<p}$ case, although the base $\Fc A$ did not appear in \cite{shen2018tamely}, the morphism $\mathbf{n}$ is constructed over various closed subschemes of $\Fc A$ in \cite[Thm. 4.3]{shen2018tamely}.
Furthermore, our morphism $\mathbf{n}$ is compatible with Schepler's Log-$p$-NAHT as in \cite{schepler2005logarithmic} in the case of nilpotent $p$-curvature and nilpotent residue.
For $G$ reductive, a different version of Log-$p$-NAHT is established in \cite{li2024tame}, which firstly is not over $\Fc A$, but rather over $A(\omega_{C'}(D'))$
and secondly, in place of the Dolbeault moduli stack as in \eqref{eq: thm 1.1}, it has a stack of certain triples.
 Moreover, the results in \cite{schepler2005logarithmic}, \cite{shen2018tamely}, and \cite{li2024tame},
have no analogue for semistable objects; see \eqref{eq: sst nss in main thm1}.
We refer the reader to \Cref{subs: comparison} for a more detailed comparison between \Cref{thm: 1st main thm} and the aforementioned works.

We now state the two main geometric applications that
we derive from \Cref{thm: 1st main thm}.

In the $G=GL_r$ case, using the Log-$p$-NAHT \Cref{thm: 1st main thm} above, we can transfer many geometric properties of $\sM_{Dol}(C',D')$ to $\sM_{dR}(C,D)$, and
deduce the following theorem. 
Note that, when $G=GL_r$, there is a decomposition of $\Fc A$ into connected components $\Fc A=\coprod_{\delta\in \bF_p}\Fc A^{\delta}$.
For each $d\in \bZ$, we denote by $\ov{d}\in\bF_p$ its image in $\bF_p$. The morphism $h_{dR}^{\Fc A}: \sM_{dR}(C,D)\to \Fc A$ restricts to $\sM_{dR}(C,D,d)\to \Fc A^{-\ov{d}}$.
\begin{thm}[{\bf Geometric properties of the de Rham-Hitchin morphism}: \Cref{prop: irreducible}, \Cref{prop: surj of drhr}, \Cref{prop: surj and conn fib}, \Cref{cor: fiber dim of hdrc}, \Cref{prop: weak ab fib}]
Assume that $G=GL_r$  and let $d \in \bZ$.
    The stack $\sM_{dR}(C,D,d)$ is integral and normal.
    The morphism $h_{dR}^{\Fc A}:\sM_{dR}(C,D,d)\to \Fc A^{-\ov{d}}$ is surjective.
   The morphism $h_{dR}^{\Fc A}: M_{dR}^{ss}(C,D,pd)\to \Fc A^{\ov{0}}$
    is proper and surjective with connected equidimensional fibers of the same dimension, and  it gives rise to a weak abelian fibration in the sense of Ng\^o \cite[\S7.1.1]{ngo-lemme-fondamental}.
    \end{thm}

Using the connectedness of fibers of $\sP^o$ over $\Fc A$ and the Decomposition Theorem \cite{bbdg}, we obtain the following cohomological consequence:
\begin{thm}[{\bf The splitting injection for $GL_r$}: \Cref{thm: injection} and \Cref{thm: coh pnaht2}]
    Assume that $G=GL_{r<p}$.
    The semistable Log-$p$NAHT morphism \eqref{eq: sst nss in main thm1} induces a split monomorphism  in $D^b_c(A(\omega_{C'}(D')),\oql)$  of derived direct image complexes
    \begin{equation*}
    h_{Dol,*}\IC_{M_{Dol}^{ss}(C',D',d)}\hookrightarrow h_{dR,*}\IC_{M_{dR}^{ss}(C,D,pd),}
    \end{equation*}
     hence a split  filtered injection of intersection cohomology groups
    \begin{equation*}
    I\!H^* (M_{Dol}^{ss}(C',D',d))
    \hookrightarrow 
    I\!H^* (M_{dR}^{ss}(C,D,pd)).
    \end{equation*}

    Furthermore, if $(pd,r)=1$, then the embedding above induces an isomorphism of filtered cohomology groups
    \[ \Big(H^*(M_{Dol}^{ss}(C',D',d)),P^h\Big)\cong 
        \Big(H^*(M_{dR}^{ss}(C,D,pd)), P^{h_{dR}}\Big).\]
\end{thm}

\Cref{section: two views on residues} is devoted to developing the theory of log-connections on torsors on affine group schemes over a base.
This theory is rooted in \cite[\S2.6]{beilinson-drinfeldquantization}. In the case with no poles and thus no residues,
this theory is developed in \cite[\S A]{chen-zhu}.
In the case with poles, this theory is used in \cite{li2024tame}.
We need this theory in the course of the proof of \Cref{thm: 1st main thm}, as we need to keep track of various types of residues under the Log-$p$-NAHT morphism $\mathbf n$ \eqref{eq: thm 1.1}.

{\bf Acknowledgments.}
We thank Bhargav Bhatt, Michael Groechenig,
Andres Fernandez Herrero, Pengfei Huang, Jie Liu, Davesh Maulik, Weite Pi, Junliang Shen, Shiyu Shen, Hao Sun, Longke Tang, Qixiang Wang, and Daxin Xu for useful conversations.
The first-named author has been supported by a grant from the Simons Foundation (672936, de Cataldo)
 and by NSF Grants DMS 1901975 and DMS 2200492. 
This material is based
upon work supported by the National Science Foundation under Grant No. DMS-1926686.

\subsection{Notation}

\subsubsection{Algebraic Geometry} $\;$
We work over an algebraically closed field $k$ of positive characteristic $p>0$.
Let ${\rm fr}_k: \mathrm{Spec}(k)\to \mathrm{Spec}(k)$ denote the absolute Frobenius.

Let $C$ be an integral nonsingular projective curve over $k$ with genus $g(C)\geq 2$.
Let $\omega_C$ be the cotangent sheaf of $C$.
Let $D = \sum_i q_i >0$ be a reduced strictly effective divisor on $C$. 
We denote by $\omega_C(D)$ the
rank one locally free sheaf of rational $1$-forms on $C$ with at worst log poles a $D$.

Let $C':= C\times_{\mathrm{Spec}(k), {\rm fr}_k} \mathrm{Spec}(k)$ be the Frobenius twist of $C$.
Let $Fr_C: C\to C'$ be the relative Frobenius $k$-morphism.
We often denote by $(-)'$ the pull-back to $C'$ of an object $(-)$ on $C$
via the natural projection  $C'\to C$.
For example, $D'$ is $D\times_C C'$.
Sometimes we use $Fr$ to denote $Fr_C$ when it is clear from the context.

 Unless otherwise stated, by stack we mean an algebraic stack over $k$. If $G$ is a group scheme over $k$ and $X$ is a $G$-variety, then  the quotient stack is denoted by $X/G$ and the adequate moduli space, if it exists, is denoted by $X/G \to X\git G$.

Given two morphisms of scheme $X\to Y\leftarrow Z$, we sometimes use $X_Z$ to denote the base change $X\times_Y Z$ when it is clear from the context.

\subsubsection{Group Theory} $\;$

Let $G$ be a connected smooth reductive group over $k$.
Let $r(G)$ be the reductive rank of $G$.
Let $B\subset G$ be a Borel subgroup.
Let $T\subset B$ be a maximal torus.
Let $\Ft\subset \Fb\subset \Fg$ be the Lie algebras of $T\subset B\subset G$.
Let $W$ be the Weyl group for $(G,T)$.
Let $B^+\subset G$ be the Borel subgroup opposite to $B$ with respect to $T$. 
Let $\Fb^+$ be the Lie algebra of $B^+$.
Let $X^*(T)$ and $X_*(T)$ be the character and cocharacter lattices.
Let $\Phi$ be the root system for $(G,T)$.
Let $\Phi^+\subset \Phi$ be the positive roots, i.e., by definition, the roots whose root spaces are in $\Fb^+$.
Let $\Delta\subset \Phi^+$ be the corresponding simple roots.
Let $\check{\Phi}$ be the coroot lattice in $X_*(T)$.
Let $\pi_1(G):=X_*(T)/\check{\Phi}$ be the fundamental group of $G$ as in \cite[Def. 5.4]{hoffmann-connected-components}.

\begin{context}\label{context: G}
    We assume that either $p\nmid |W|$, or $G=\prod_{i=1}^m GL_{r_i}, r_i\in\bN_{>0}$.
\end{context}

For the convenience of the reader, we record that for simple $G$, the order $|W|$ according to the type of $G$ is: $A_n: |W|=(n+1)!$; $B_n, C_n: |W|=2^nn!$; $D_n: |W|=2^{n-1}n!$; $G_2: |W|=12$; $F_4: |W|=2^7\times 3^2$; $E_6: |W|=72\times 6!$; $E_7: |W|=72\times 8!$; $E_8: |W|=192\times 10!$.

Let $\Fc:= \Fg\git G$ be the GIT quotient of the adjoint action of $G$ on $\Fg$.
There is the canonical Chevalley isomorphism $\Ft\git W\cong \Fc$.

Given a group scheme $\cG$ over a $k$-scheme $S$, we denote by $\mathrm{Lie}(\cG)$ the sheaf of $p$-Lie algebras over $S$.
We use $x\mapsto x^{[p]}$ to denote the $p$-operation in the $p$-Lie algebras.
When $\cG$ is smooth, we do not distinguish between the locally free sheaf $\mathrm{Lie}(\cG)$ and its associated vector bundle scheme over $S$.

\begin{defn}\label{defn: cox num}
    Let $h(G)$ be the number defined in \cite[\S5.1]{serre2005complete}, 
i.e. if $G$ is a torus, then $h(G)=1$, otherwise $h(G)$ the maximum of the Coxeter numbers of the simple quotients of $G$.
\end{defn}
For simple groups, the number $h(G)$ according to the type of $G$ is: $A_n, h=n+1; B_n, C_n: h=2n; D_n: h=2n-2; G_2: h=6; F_4, E_6: h=12; E_7: h=18; E_8: h=30$.

Sometimes we also require that $p\geq h(G)$. 
This condition implies that $p\nmid |W|$, except for the cases $A_{p-1}$ and $D_{2=p}$.

\begin{defn}
\label{defn: good enough}
    We say that $(p,G)$ is good enough if $p\geq h(G)$, or if $G=\prod_{i=1}^m GL_{r_i}, r_i\in \bN_{>0}$.
\end{defn}

By \cite[(5.2.6)]{serre2005complete}, we have that $2h(G)-2=\mathrm{ht}(\Fg)$, the height of the adjoint representation of $G$ on $\Fg$, see \cite[\S5.2]{serre2005complete} and \cite[Def. 4.4]{balaji2017complete} for the notion of the height of a $G$-representation.
\begin{defn}\label{def: low height}
    We say that $G$ satisfies Low Height (LH) if $p> \mathrm{ht}(\Fg)$.
\end{defn}
We see that if $G$ satisfies (LH), then $p\geq h(G)$ and $p\nmid |W|$.

\section{Hitchin morphisms and residues}

\subsection{Hitchin-type morphisms}\label{subs: moduli}$\;$

Let $L$ be an invertible sheaf on $C$.
Let $L^{\times}:=\mathrm{Spec}_{\cO_C}(\bigoplus_{n\in \bZ}L^{\otimes n})$ be the associated $\bG_m$-torsor on $C$.
The natural $\bG_m$-action on $\Fg$ gives rise to natural $\bG_m$-actions on $\Ft, \Fg/G, \Fc$.
We denote the corresponding $\bG_m$-twists $(-)\times^{\bG_m} L^{\times}$ by $\Fg_L, \Fc_L, \Ft_L, (\Fg/G)_L$.
The first three are schemes over $C$ and $(\Fg/G)_L$ is a stack over $C$.

\subsubsection{Higgs bundles and connections with log-poles} $\;$

An $L$-Higgs bundle on $C$ is a section of the natural morphism $(\Fg/G)_L\to C$.
Equivalently, it is a pair $(E, \phi)$, where $E$ is a $G$-bundle and $\phi\in H^0(\mathrm{ad}(E)\otimes L)$, where $\mathrm{ad}(E)$ is the adjoint bundle. We call $\phi$ the $L$-twisted Higgs field.
A Higgs (resp. log-Higgs) field/bundle corresponds to the special case $L=\omega_C$ (resp. $L=\omega_C(D)$).
Given a $G$-bundle $E$ on $C$, the notion of a (log-)connection $\nabla$ on $E$ is reviewed in \Cref{section: two views on residues}. 
By abuse of notation, we also call the pair $(E,\nabla)$ a (log-)connection.

\subsubsection{Degrees of $G$-bundles} $\;$

Let $Bun_C(G)$ be the moduli stack of $G$-bundles on $C$.
By \cite[Thm. 5.8]{hoffmann-connected-components}, the set of connected components $\pi_0(Bun_C(G))$ is in canonically isomorphic to $\pi_1(G)$.
Given a $G$-bundle $E$ on $C$ and $d\in \pi_1(G)$, we say that $E$ is of degree of $d$ if the point in $Bun_G(C)$ defined by $E$ lies in the connected component in  $\pi_0(Bun_C(G)) = \pi_1(G)$ defined by $d$.
We say that a log-Higgs bundle $(E,\phi)$ or a log-connection $(E,\nabla)$ is of degree $d$ if $E$ has degree $d$.

\subsubsection{Moduli stacks and spaces} $\;$

In this paper, we consider the moduli of log-Higgs bundles and log-connections on $C$ of degree $d$, which are denoted by $\sM_{Dol}(C,G,D,d)$ and $\sM_{dR}(C,G, D,d)$, respectively.
We denote by  $\sM^{ss}_{(-)} (C,G, D,d)$  the corresponding open substacks of slope-semistable objects.
If $G$ satisfies the Low Height (LH) property as in \Cref{def: low height}, or if $G=GL_r$, then $\sM^{ss}_{(-)} (C,G, D,d)$ admits a quasi-projective adequate moduli space 
$M_{(-)}^{ss}(C,G, D,d)$, see \cite[Thm. 1.1]{langer2014semistable} for the existence of these moduli spaces in the vector bundle case, and \cite[Thm. 2.26]{herrero-zhang-mero} for the general $G$-bundle case.

For any invertible sheaf $L$, we use $\sM_{Dol}(C,G, L,d)$ to denote the moduli stack of $L$-twisted Higgs bundles of degree $d$ on $C$.
In particular, we have that $\sM_{Dol}(C,G, \omega_C(D),d)$ is the same as $\sM_{Dol}(C,G, D,d)$.

Unless otherwise stated, the group $G$ is fixed and we omit
it from the notation.
At times, we omit the decorations $C$ and $D$ when they are clear from the context.
At times, we omit the decoration $d$ to denote the whole moduli stack of objects of all possible degrees. We have $\sM_{(-)}(C,D) = \coprod_{d \in \pi_1(G)} \sM_{(-)}(C,D,d)$, etc.
We refer to the  case  when $D=\emptyset$ as to the ``case with no poles."

\subsubsection{The Hitchin morphism} $\;$

Let $A(\omega_C(D))$ be the functor that sends a $k$-scheme $R$ to the set of sections of the natural morphism $\Fc_{\omega_{C}(D)}\times_{\mathrm{Spec}(k)} R\to C\times_{\mathrm{Spec}(k)} R$.
The functor $A(\omega_C(D))$ is represented by an affine space over $k$, which is called the Hitchin base.
The quotient morphism $\chi: \Fg/G\to \Fc$ induces the Hitchin morphism $h_{Dol}: \sM_{Dol}(C,D)\to A(\omega_C(D))$.
Similarly, we also have the Hitchin morphisms 
 $h_{Dol}:\sM_{Dol}(C,D)\to A(D)$ and $h_{Dol}:\sM_{Dol}(C',D')\to A(\omega_{C'}(D'))$, etc.

\subsubsection{The de Rham-Hitchin morphism} $\;$

As recalled in \Cref{section: two views on residues}, especially \Cref{sssec: def p curv}, given a log-connection $(E,\nabla)$, the associated $p$-curavture $(E,\Psi(\nabla))$ is an $\omega_{C}^p(pD)$-twisted Higgs bundle, i.e., a section of $(\Fg/G)_{\omega_C^p(pD)}$ over $C$.
The Hitchin image of $(E,\Psi(\nabla))$, which a priori is a section of $\Fc_{\omega_{C}^p(pD)}$ over $C$, is indeed the $Fr_C$-pullback of a section $(\chi(E,\Psi(\nabla)))^{1/p}$ of $\Fc_{\omega_{C'}(D')}$ over $C'$, see \cite[\S5.2]{li2024tame} and \cite[Prop. 5.7]{herrero-zhang-mero}.
The  resulting de Rham-Hitchin morphism $h_{dR}:\sM_{dR}(C,D)\to A(\omega_{C'}(D'))$ is the morphism that sends $(E,\nabla)$ to $(\chi(E,\Psi(\nabla)))^{1/p}$.
 
For some basic geometric properties of these moduli stacks/spaces and of these Hitchin-type morphisms, see 
$\S$\ref{subs: geom conseq},
\ref{subs: fltns} and
\ref{subs: waf}.

\subsection{The Picard stack}\label{subs: act}\;

Let $L$ be an invertible sheaf on $C$.
Let $R$ be a $k$-scheme.

\subsubsection{Regular centralizers} $\;$

Let $J$ be the smooth commutative group scheme of regular centralizers over $\Fc$ as in \cite[Lem. 2.1.1]{ngo-lemme-fondamental}. 
By \cite[the paragraph above \S2.2]{ngo-lemme-fondamental}, the $\Fc$-group scheme $J$ is $\bG_m$-equivariant for the $\bG_m$-action on $\Fc$ induced from the scaling $\bG_m$-action on $\Fg$; therefore,  $J$ descends to a group scheme $J/\bG_m$ over the stack $\Fc/\bG_m$.
We can thus twist the group scheme $J$ by the $\bG_m$-torsor $L^{\times}$ and obtain the group scheme $J_{L}$ over $\Fc_{L}$.
For each section $a: C_R\to (\Fc_L)_R$, let $J_a:= a^* (J_L)_R$ be the pullback group scheme over $C_R$.
We emphasize that, to define $J_a$, we do not need to use the Kostant section, thus do not need to require that $L$ has a square root.

\subsubsection{Picard stacks and their torsors } $\;$

Let $\sS$ be a site and let $\sG$ be a stack on $\sS$.
Following the terminology in \cite[\S A.1]{chzh17}, we say that $\sG$ is a Picard stack if it is a strictly commutative Picard stack as in \cite[XVIII, Def. 1.4.5]{SGA4}. 
Let $\sT$ be a stack over $\sS$.
Let $a: \sG\times \sT\to \sT$ be a bi-functor.
We say that $a$ defines an action of $\sG$ on $\sT$ if the condition (i) in \cite[\S A.5]{chzh17} is satisfied.
We say that $a$ makes $\sT$ a $\sG$-pseudo-torsor  if the conditions (i), (iii) in loc. cit. are satisfied.
We say that $a$ makes $\sT$ a
$\sG$-torsor
 if the conditions (i), (ii), (iii) in loc. cit. are satisfied.
In particular, a $\sG$-torsor is a $\sG$-pseudo-torsor that is locally nonempty.

 \subsubsection{The Picard stack $\sP(L)$} $\;$
 
The $A(L)$-stack $\sP(L)$ sends the point $a: R\to A(L)$ to the groupoid of $J_a$-torsors on $C_R$.
Since $J$ is commutative, we have that $\sP(L)$ is a Picard stack over $A(L)$.

 \subsubsection{The  action of  the Picard stack $\sP(L)$  on $\sM_{Dol}(L)$}\label{sssec: action of p on mdol} $\;$

 Let $(E,\phi)$ be a Higgs bundle over $C_R$ with Hitchin image $a\in A(L)(R)$.
 Let $Aut(E,\phi)$ be the group scheme over $C_R$ of automorphisms of $(E,\phi)$.
 There is a canonical morphism of group schemes $c_{E,\phi}: J_a\to Aut(E,\phi)$ on $C_R$, which enables us to twist $(E,\phi)$ by a $J_a$-torsor $F$ to obtain a new Higgs bundle $(E\times^{J_a}_{c_{E,\phi}} F, \phi_F)$ whose Hitchin image is still $a$, see, e.g. \cite[\S2.6]{chzh17}.
 In this way, we obtain an action of $\sP(L)$ on $\sM_{Dol}(L)$:
 \begin{equation}
 \label{eq: act p mdol}
     \sP(L)\times_{A(L)}\sM_{Dol}(L)\to \sM_{Dol}(L),\quad (F, (E,\phi))\mapsto (E\times_{c_{E,\phi}}^{J_a} F,\phi_F).
 \end{equation}

  \subsubsection{The Kostant section}$\;$
  
 Assume that the invertible sheaf has a square root $L^{1/2}$.
 Then the Hitchin morphism $\sM_{Dol}(L)\to A(L)$ has a section $\eta_{L^{1/2}}$, called the Kostant section, see \cite[\S2.3]{chzh17}.
 Via $\eta_{L^{1/2}}$, the action \eqref{eq: act p mdol} defines an open embedding 
 \begin{equation}
 \label{eq: kostant embed}
     \iota_{L^{1/2}}: \sP(L)\hookrightarrow\sM_{Dol}(L).
 \end{equation}

 \subsubsection{The Picard stack $\sP$}$\;$
 
 We denote by $\sP$ the Picard stack $\sP(\omega_{C'}(D'))$ over $A(\omega_{C'}(D'))$.

 \subsubsection{The action of $\sP$ on $\sM_{dR}(C,D)$}$\;$
 
 Let $a'\in A(\omega_{C'}(D'))(R)$.
 Let $F$ be a $J_{a'}$-torsor over $C'$.
 Then both the Frobenius pullbacks $Fr^*J_{a'}$ and $Fr^*F$ admit canonical connections $\nabla^{can}_J$ and $\nabla_F^{can}$, and $Fr^*J_{a'}$ acts on $Fr^*F$ in a way that is compatible with the connections \cite[Ex. A.5]{chen-zhu}. See \Cref{section: two views on residues} for a review of connections on affine (group) schemes/torsors.
 Let $(E,\nabla_E)$ be a log-connection on $C_R$ with de Rham-Hitchin image $a'$.
 Then the $p$-curvature $\Psi(\nabla_E)$ is an $\omega_C^p(pD)$-twisted Higgs bundle whose Hitchin image is the section $a^p:=Fr^*a': C_R\to (\Fc_{\omega_C^p(pD)})_R$.
 Note that the $C$-group schemes $J_{a^p}$ and $Fr^*J_{a'}$ are canonically identified.
 By the formalism developed in \Cref{section: two views on residues}, we can form the twist log-connection
 $(E_F,\nabla_{E_F}):=(E\times_{c_{E,\Psi(\nabla_E)}}^{J_{a^p}} Fr^*F, \nabla_E\otimes \nabla_F^{can})$.
 Since $\Psi(\nabla_F^{can})=0\in \Gamma(C, \mathrm{Lie}(J_{a^p})\otimes\omega_C^p(pD))$ \cite[Ex. A.5]{chen-zhu}, the same argument as in the penultimate paragraph of the proof of \cite[Thm. 3.12]{chen-zhu} shows that $(E_F,\nabla_{E_F})$ has $a'$ as its de Rham-Hitchin image.
 Therefore, we obtain an action of $\sP$ on $\sM_{dR}(C,D)$:
 \begin{equation}
     \sP\times_{A(\omega_{C'}(D'))}\sM_{dR}(C,D)\to \sM_{dR}(C,D),\quad (F,(E,\nabla_E))\mapsto (E_F,\nabla_{E_F}).
 \end{equation}

\subsubsection{The $GL_r$ case}$\;$

We consider the special case $G=GL_r$.
Let $X$ be a smooth projective curve over $k$ and 
let $L$ be a line bundle over $X$.
The Hitchin base $A(L)$ parametrizes the spectral curves in the total space of $L$, see \cite[Lem. 6.10]{simpson-repnII}.
Let $\cS(L)\to A(L)$ be the universal spectral curve.
We have that $\sP(L)$ is the $A(L)$-stack parametrizing line bundles on $\cS(L)$ relative to $A(L)$, see \cite[\S4.4.1]{ngo-lemme-fondamental}.
By the Beauville-Narasimhan-Ramanan (BNR) correspondence \cite[Prop. 3.6]{bnr-spectral-curves}, \cite[Prop. 2.1]{schaub-courbes-spectrales}, the $A(L)$-stack $\sM_{Dol}(L)$ of $L$-twisted Higgs bundles on the curve $X$ is canonically equivalent to the $A(L)$-stack of torsion free rank 1 sheaves on the spectral curves.
Therefore, we have that $\sP(L)$ is an open substack  of the stack of $\sM_{Dol}(L)$, and that $\sP(L)$ acts on $\sM_{Dol}(L)$ via the tensor products of sheaves on the spectral curves:
\begin{equation}
    \label{eq: pl act on mdoll}
    \sP(L)\times_{A(L)}\sM_{Dol}(L)\to \sM_{Dol}(L), 
    \quad  (\cM, \cE)
    \longmapsto \cE \otimes \cM.
\end{equation}

\begin{remark}\label{sssec: two embeddings}
    Note that the embedding $\sP(L)\hookrightarrow\sM_{Dol}(L)$ is induced by the section $s: A(L)\to \sM_{Dol}(L)$ sending $a\in A(L)$ to the structural sheaf of the spectral cover $X_a$, which is different from the Kostant section $\eta_{L^{1/2}}$. Indeed, when $r=2$, the vector bundle underlying the Higgs bundle $s(a)$ is $\cO_X\oplus L^{-1}$, while the one underlying $\eta_{L^{1/2}}(a)$ is $L^{-1/2}\oplus L^{1/2}$ \cite[\S7.1]{dalakov2017lectures}.
\end{remark}

When $X$ is the Frobenius twist $C'$ of $C$ and $L=\omega_{C'}(D')$, we denote by $\sP$ the Picard stack $\sP(\omega_{C'}(D'))$ over $A(\omega_{C'}(D'))$.
We have the canonical decomposition by degrees of line bundles $\sP = \coprod_{d \in \bZ} \sP(d)$.

Let $a'$ be a geometric point of $A(\omega_{C'}(D'))$.
Let $S_{a'}$ be the corresponding spectral curve over $C'$.
There is the central $\cO_{S_{a'}}$-algebra $\sD(-D)$ over $S_{a'}$   \eqref{diag: four corners on sp curves}.
By the Log-de Rham-BNR \Cref{lemma: log dr bnr}, the objects in $h_{dR}^{-1}(a')\subset \sM_{dR}(C,D)$ are given by torsion free rank $p$ sheaves on $S_{a'}$ with a $\sD(-D)$-module structure.
Let $\cE$ be such a $\sD(-D)$-module on $S_{a'}$. 
Given a line bundle $M$ on $S_{a'}$, by using that
$\sD(-D)$ over $S_{a'}$ is a central $\cO_{S_{a'}}$-algebra,
we have that $\cE\otimes M$ is  a $\sD(-D)$-module, hence an object in $h_{dR}^{-1}(a').$
We thus obtain the action of $\sP$ on $\sM_{dR}(C,D)$: 
\begin{equation}
\label{eq: p action mdr a}
\sP\times_{A(\omega_{C'}(D'))}\sM_{dR}(C,D)\to \sM_{dR}(C,D),
\quad 
(M, \cE) \longmapsto (\cE \otimes M).
\end{equation}

\subsection{Residues}\label{subs: residues}$\;$

Throughout this paper, given an integer $d \in \bZ$, we  denote by $\ov{d}$ the image of $d$ in $\bF_p$.

The goal of this section is to construct diagram \eqref{diag: res of p curv}, and its more precise version \eqref{diag: hdr fiber not conn} in the case of  $G=GL_r$, 
 relating the charactrisitc polynomials of the residues of log-connections to the ones of the residues of their $p$-curvature.

\subsubsection{Twists}$\;$

The invertible sheaf $\omega_{C}(D)|_D$ on $D$ has a canonical trivialization $\omega_{C}(D)|_D=\cO_D$, given by $\frac{dx}{x}$ in local coordinates.
Consider the twist $\Fg \times^{\mathbb G_m} (\omega_C(D)|_D)^\times$
of $\Fg$ by the line bundle $\omega_C(D)|_D = \cO_D$;
it is canonically isomorphic to $\Fg\times_k D$ and we denote it
by $\Fg_D$. Similarly, we obtain
  $\Fc_D =\Fc \times_k D$. We also need the twists of $\Fg$ and $\Fc$ with respect to the canonically trivialized line bundle 
$\omega^p_C(pD)|_D$ on $D$; 
these twists are canonically isomorphic to
the previous ones $\Fg_D$
and $\Fc_D$ via the 
$p$-th iteration of the residue isomorphism.
We set
$G_D:= G\times_k D$.

\subsubsection{Section stack over poles}$\;$

Given a stack $X$ over $D$, we denote by  $\Gamma(D,X)$ the $k$-stack of sections of $X$, i.e., for any $k$-scheme $S$, the groupoid $\Gamma(D,X)(S)$ is given by the sections $\Gamma(D\times_k S, X\times_k S)$.
If the stack $X$ is represented by a vector bundle of rank $r$ over $D$, which is necessarily trivial, then $\Gamma(D,X)$ is isomorphic to the $k$-affine space $\bA^{r\deg(D)}$.

\subsubsection{Residues of connections}$\;$

Given a connection $(E,\nabla)$ on $C$ with log-poles along $D$,
the residue $\mathrm{res}_D(\nabla)$ is a section in $H^0(D,\mathrm{ad}(E|_D))$.
We thus have the following two morphisms

\begin{equation}
\label{eq: resd chid}
\xymatrix{
\mathscr{M}_{dR}(C,D) \ar[r]^-{\mathrm{res}_D} &\Gamma(D, \Fg_{D}/G_{D}) \ar[r]^-{\chi_D} &
\Gamma(D, \Fc_D),
}    
\end{equation}
where $\mathrm{res}_D$ takes $(E,\nabla)$ to $(E|_D, \mathrm{res}_D(\nabla))$, and $\chi_D$ is induced by $\chi:\Fg/G\to \Fc$.

\subsubsection{Residues of Higgs bundles}$\;$

If we have a log-Higgs bundle
$(E, \phi\in H^0(\mathrm{ad}(E)\otimes\omega_C(D))$ on $C$, then we can form a residual Higgs bundle on $D$ by using 
the residue isomorphism and obtain $\mathrm{res}_D(\phi):=\phi|_D\in H^0(D,\mathrm{ad}(E)\otimes\omega_C(D)|_D)=H^0(D,\mathrm{ad}(E|_D))$.

By considering characteristic polynomials, we get the commutative diagram of $k$-stacks:
\begin{equation}\label{eq:resdol1}
\xymatrix{
\mathscr{M}_{Dol}(C,D)
\ar[r]^-{\mathrm{res}_D} \ar[d]^-{h_{Dol}}
&
\Gamma(D,\Fg_{D}/G_{D})
\ar[d]^-{\chi_D}
\\
A(\omega_C(D)) \ar[r]^-{ev_D = |_D}
&
\Gamma(D,\mathfrak{c}_D),
}
\end{equation}
where $\mathrm{res}_D((E,\phi)):=(E|_D,\mathrm{res}_D(\phi))$.

\subsubsection{The $GL_r$-case}$\;$

When $G=GL_r$, by \cite[Thm. 1.2]{manikandan2022criterion}, for each $d\in \bZ=\pi_1(GL_r)$, the morphism $\chi_D\circ Res(D)$ maps $\sM_{dR}(C,GL_r, D,d)$ onto the codimension one affine linear subspace $\Gamma(D,\Fc_D)^{-\overline{d}}\subset \Gamma(D,\mathfrak{c}_D)$ defined by the Fuchs relations, i.e. by the sum of traces $\sum_{q\in D}tr|_{q}$ being equal, in the field $k$, to $-\ov{d}$, i.e. minus the residual class modulo $p$  of the degree $d$. If we take the degree to be a multiple of $char(k)=p$, then this  value is zero in $k$ and the affine linear subspace is a linear subspace, which we denote by 
\begin{equation}\label{eq: gammazero}
\Gamma(D,\Fc_D)^{\ov 0}
\subset 
\Gamma(D,\Fc_D).
\end{equation}

\begin{remark}
    In the $GL_r$ case, since $H^1(C,\omega_C) =k$, the image of $ev_D$ coincides with the subscheme $\Gamma(D,\Fc_D)^{\ov 0}$.
\end{remark}

Using the canonical identification $\omega_C^p(pD)|_D=(\omega_C(D)|_D)^{\otimes p}=\cO_D$, we have a  commutative diagram of $k$-stacks for $\omega_C^p(pD)$-Higgs bundles
similar to \eqref{eq:resdol1}
\begin{equation}\label{eq:resdr1}
\xymatrix{
\mathscr{M}_{Dol}(\omega^p_C(pD))
\ar[r]^-{\mathrm{res}_D} \ar[d]^-{h_{Dol}^p}
&\ar[d]^-{\chi_D} \Gamma(D,\Fg_{D}/G_{D})
\\
A(\omega^p_C(pD)) \ar[r]^-{ev_D(-)}
&
\Gamma(D,\Fc_D).
}
\end{equation}

\subsubsection{Residues of $p$-curvature} $\;$

Consider the ring $k[x]$ and the derivation $x\partial_x$.
The $p$-th iterate derivation $(x\partial_x)^{[p]}$ equals $x\partial_x$ itself.
Therefore, given a log-connection $(E,\nabla)$ on $C$, we have the equality 
\begin{equation}\label{eq:respcurv}
\mathrm{res}_D(\Psi(\nabla))=\mathrm{res}_D(\nabla)^{[p]}-\mathrm{res}_D(\nabla)\in H^0(D, \mathrm{ad}(E)|_D).
\end{equation}

The assignment $A\mapsto A^{[p]}-A$, where $(-)^{[p]}$ is the $p$-operation on the $p$-Lie algebra $\Ft$, defines a surjective and \'etale $k$-morphism $AS_{\Ft}: \Ft \to \Ft$ which is equivariant with respect to the $W$-action \cite[p.279, Rmk. (3)]{mcninch2002abelian},
thus inducing a surjective  $k$-morphism $ AS_{\Fc}:\Fc \to \Fc$
of degree $p^{r(G)}$, where $r(G)$ is the reductive rank.

 \begin{lemma}
 \label{lemma: no conn fib}
    The morphism $AS_{\Fc}:\Fc\to \Fc$ is finite, surjective, Galois and  with disconnected geometric fibers. When $G=GL_{r\geq 2}$, the morphism $AS_{\Fc}$ is not \'etale. 
\end{lemma}
\begin{proof}
    The morphism $AS_{\Ft}: \Ft
    \to \Ft$  and the 
    characteristic polynomial morphism
    $\chi|\Ft: \Ft \to \Fc$ are both finite. The composition $\chi|\Ft \circ AS_{\Ft} = AS_{\Fc} \circ \chi|{\Ft}$ is thus finite. 
    It follows from \cite[\href{https://stacks.math.columbia.edu/tag/0AH6}{Tag 0AH6}]{stacks-project} that $AS_{\Fc}$ is also finite.

To show that the geometric fibers of $AS_{\Fc}$ are disconnected, it suffices to show that given any $t\in \Ft$, the points in the disconnected fiber $AS_{\Ft}^{-1}(t)$ do not all lie in the $W$-orbits.
Choose an embedding $(G,T)\hookrightarrow (GL_N,T_N)$.
The induced embedding $\Ft\hookrightarrow \Ft_N$ is compatible with both the Weyl group actions and the $A\mapsto A^{[p]}-A$ morphism \cite[Prop. 4.4.9]{springer1998linear}.
Identifying $T_N$ with $\bG_m^N$, we see that $AS_{\Ft}$ preserves each $\Fgl_1$ factor in $\mathrm{Lie}(\bG_m^N)$, and that the Weyl group of $GL_N$ permutes the factors. 
Since $AS_{\Fgl_1}$ has disconnected fibers, we have that $AS_{\Ft}^{-1}(t)$ lie in different $W$-orbits as desired.
    
In the $GL_r$-case, the $k$-points of $\Fc$ are in bijection with the set of unordered $r$-tuple of points  $\{e_1,...,e_r\}$, with $e_i\in k$, given by the eigenvalues of matrices.
The  general fiber (i.e. the $e_i$'s are pairwise distinct) is reduced with cardinality $p^r$. The set-theoretic fiber of any point of the form  $\{0,0, \ldots \}$  (at least two entries equal to zero) has cardinality strictly smaller than  $p^r$. It follows that   the finite and surjective morphism $AS_{\Fc}$
    onto the  integral $\Fc$ is not \'etale. 
\end{proof}

Taking $AS_{\Fc}$ on sections in $\Gamma(D,\Fc_D)$, we obtain a $k$-morphism $AS_{\Fc_{D}}: \Gamma(D,\Fc_D)\to \Gamma(D,\Fc_D)$.
By \Cref{lemma: no conn fib}), the  finite $k$-morphism $AS_{\Fc_{D}}$ is not \'etale when $G=GL_{r\geq 2}$, and it does not have connected geometric fibers.
In summary, we have the commutative diagram 
\begin{equation}
\label{diag: res of p curv}
    \xymatrix{
&
\mathscr{M}_{dR}(C,D)
\ar[r]^-{\Psi}
\ar[d]^-{h_{dR}}
\ar[ddl]_-{\chi_D\circ \mathrm{res}_D}
&
\mathscr{M}_{Dol}(C, \omega_C^{\otimes p}(pD))
\ar[d]^-{h_{Dol}^p}
\ar@/^3pc/[dd]^-{ \chi_D \circ \mathrm{res}_D}
\\
&
A(\omega_{C'}(D'))
\ar@{^{(}->}[r]^-{Fr_C^*}
\ar[d]^-{ev_{D'}}
& A(\omega_C^{\otimes p}(pD))
\ar[d]^-{ev_{D}}
\\
\Gamma(D,\Fc_D)
\ar@{->>}@/_2pc/[rr]^-{AS_{\Fc_{D}}} \ar@{.>>}[r]^-{\Fas}
&
\Gamma(D',\Fc_{D'})
\ar@{^{(}->>}[r]^-{Fr_D^*}
&
\Gamma(D,\Fc_D),
}
\end{equation}
where the surjectivity of $AS_{\Fc_D}$ follows from \Cref{lemma: no conn fib}, the definition of the morphism $\Fas$ is given in \Cref{defn: fas}, and the bottom square is given by pulling back sections along the commutative square: 
\begin{equation}\label{eq:cdfrob}
    \xymatrix{
    C' & C  \ar[l]_-{Fr_C} \\
    D' \ar[u] & D \ar[u] \ar[l]_-{ Fr_D}.
    }
\end{equation}

Since $D$ is reduced and zero-dimensional, we have that $Fr_D:D\to D'$ is an isomorphism.
Therefore the injective $k$-morphism $Fr_D^*: \Gamma(D',\Fc_{D'})\to \Gamma(D,\Fc_D)$ is an isomorphism.

\begin{defn}\label{defn: fas}
    We define the morphism $\Fas$ to be the composition
    \begin{equation}
        \label{eq: defining fas}
        \Fas: \Gamma(D,\Fc_D)\xrightarrow{AS_{\Fc_D}} \Gamma(D,\Fc_D)\xrightarrow{(Fr_D^*)^{-1}} \Gamma(D',\Fc_{D'}).
    \end{equation}
\end{defn}

\subsubsection{\underline{$GL_r$ case}}
When $G=GL_r$, we can make the diagram more precise by 
considering the degrees for the moduli stacks and the Fuchs relation.
The image of $\sM_{dR}(C,D,d)$ under $\chi_D(\mathrm{res}_D)$ lies in $\Gamma(D,\Fc_D)^{-\ov{d}}$.
Since $AS(d)=d^p-d=0 \in k$, we have that 
\begin{equation}
\label{eq: chi d to o}
    AS_{\Fc_{D}}(\Gamma(D,\Fc_D)^{-\overline{d}})=\Gamma(D,\Fc_D)^{\ov 0}.
\end{equation}

We thus have the commutative diagram of  $k$-morphisms, where the dotted arrow is uniquely determined by the fact that $Fr_D^*$ is an isomorphism  and where we have marked the surjective morphisms
\begin{equation}
    \label{diag: hdr fiber not conn}
    \xymatrix{
&
\mathscr{M}_{dR}(C,D,d)
\ar[r]^-{\Psi}
\ar@{->>}[d]^-{h_{dR}\,\, \mbox{not conn. fib.}}
\ar@{->>}[ddl]_-{\chi_D\circ \mathrm{res}_D}
&
\mathscr{M}_{Dol}(C, \omega_C^{\otimes p}(pD),d)
\ar@{->>}[d]^-{h_{Dol}^p}
\ar@{->>}[ddr]^-{ \chi_D \circ \mathrm{res}_D}
&
\\
&
A(\omega_{C'}(D'))
\ar@{^{(}->}[r]^-{Fr_C^*}
\ar@{->>}[d]^-{ev_{D'}}
& A(\omega_C^{\otimes p}(pD))
\ar@{->>}[dr]^-{ev_{D}}
&
\\
\Gamma(D,\Fc_D)^{-\overline{d}}
\ar@{.>>}[r]^-{\Fas}
\ar@{->>}@/_2pc/[rr]^-{AS_{\Fc_{D}}}_-{ \mbox{finite,  not conn. fib.,   not \'etale} \; \forall  r\geq 2}
&
\Gamma(D',\Fc_{D'})^{\ov{0}}
\ar@{^{(}->>}[r]^-{Fr_D^*}_-{\simeq_k}
&
\Gamma(D,\Fc_D)^{\ov{0}}\ar@{^{(}->}[r]
&
\Gamma(D,\Fc_D);
}
\end{equation}
symbolically, if $A$ is a matrix representing $\mathrm{res}_D(E,\nabla)$ and $\chi (A)$ is its characteristic polynomial, then we have $\chi_D (\mathrm{res}_D(E,\nabla))=\chi(A)$
and
$\chi_D(\mathrm{res}_D(\Psi (E,\nabla)))= \chi(A^p-A)$ and, moreover, the eigenvalues of $A^p-A$ are the eigenvalues of $A$ after application of the Artin-Schreier morphism.
The morphism $h_{Dol}^p$ is surjective because of the BNR correspondence.
The morphism $h_{dR}$ is surjective because of \Cref{prop: surj of drhr}.
The morphism $Fr_C^*$ is injective because $Fr_C:C\to C'$ is surjective. 
The morphism $ev_D$ is surjective because $H^1(C,\omega_{C}^{\otimes p}(pD))=0$ for $i\geq 1.$
The morphism $ev_{D'}$ is surjective because $H^1(C,\omega_{C})=k$ and $H^1(C,\omega_C(Q))=0$ for any  strictly  effective divisor $Q$.
The morphism $Fr_D^*$ is bijective because $Fr_D$ is an isomorphism.
The morphism $\chi_D\circ \mathrm{res}_D$ is surjective because of the sufficiency of the Fuchs relation \cite[Thm. 1.2]{manikandan2022criterion}.
By \Cref{lemma: no conn fib}, the morphism $AS_{\Fc_{D}}$ is finite with disconnected geometric fibers  and is not \'etale $\forall r \geq 2$.

\begin{remark}[\bf The log-de Rham-Hitchin morphism has disconnected fibers]\label{rmk:hdr not conn fib}
In the case with no poles, when the degree $d$ is a multiple of $p$ (a necessary  and sufficient condition for the non-emptiness of the moduli stack $\sM_{dR}(C,d)$), the de Rham-Hitchin morphism $h_{dR}:  \sM_{dR}(C,d) \to A(\omega_{C'})$
 has connected fibers, just like in the case of the Hitchin morphism
 $h_{Dol}: \sM_{Dol}(C,d) \to A (\omega_C)$.
 In the log-case, the morphism $h_{Dol}: \sM_{Dol}(C,D,d)\to  A(\omega_{C}(D))$, has connected fibers, but the morphism $h_{dR}: \sM_{dR}(C,D,d) \to  A(\omega_{C'}(D'))$ does \underline{not} have connected fibers, 
 not even after passing to the semistable good moduli space, not even in the case when the degree is a multiple of $p$ and coprime to the rank, and not even in the rank one case.
 In fact, since $ev_{D'}$ has connected fibers, if $h_{dR}$ had connected fibers, 
 then the surjective $Fr^*_D \circ ev_{D'} \circ h_{dR}$ would have connected fibers, contradicting the fact that the surjective $AS_{\Fc_{D}}$ does not have connected fibers, as explained in the paragraphs above diagram \eqref{diag: res of p curv}. 
 The non-connectedness
 is rooted in the fact that 
 once we prescribe the restriction
to $D'$  in $\Gamma(D', \Fc_{D'})^{\ov{0}}$ of the characteristic polynomial of
 the $p$-curvature of a log-connection,
 there are many possible distinct sets of eigenvalues 
 in $\Gamma(D, \Fc_{D})^{\ov{0}}$
 of the residue of a log-connection yielding the same set of values in $\Gamma(D', \Fc_{D'})^{\ov{0}}$ via the Artin-Schreier-type morphism $AS_{\Fc_{D}}$.
\end{remark}

\subsection{Hitchin-type residue morphisms}

\subsubsection{The variety of residues $\Fc A$}\label{subs: var rez} $\;$

Recall the notation as in \eqref{diag: res of p curv} and \eqref{eq: defining fas}.
\begin{defn}\label{defn: base cA}
    We define the $k$-scheme $\Fc A$ to be the fiber product
    \begin{equation}\label{diag: base cA}
    \xymatrix{
        \Fc A\ar[r]^-{pr_1} \ar[d]_-{pr_2} &
        A(\omega_{C'}(D')) \ar[d]^-{ev_{D'}}\\
        \Gamma(D,\Fc_D)\ar[r]_-{\Fas} & \Gamma(D',\Fc_{D'}).
        }
    \end{equation}
\end{defn}
One may view the $k$-points of $\Fc A$ as special pairs $(a',s)$, where $a' \in A(\omega_{C'}(D'))$ and $s$ is a collection of adjoint orbits in $\Fg^{reg}$, one for each puncture in $D$.

\subsubsection{The de Rham-Hitchin-residue morphism $h_{dR}^{\Fc A}$}$\;$

The commutativity of \eqref{diag: res of p curv} entails that the morphism $h_{dR}:\sM_{dR}(C,D,d)\to A(\omega_{C'}(D'))$ factors as
\begin{equation}
    \sM_{dR}(C,D)\xrightarrow{h_{dR}^{\Fc A}} \Fc A\xrightarrow{pr_1} A(\omega_{C'}(D')).
\end{equation}
We call the morphism $h_{dR}^{\Fc A}:\sM_{dR}(C,D)\to \Fc A$ the de Rham-Hitchin-residue morphism.

\subsubsection{$\wt{(-)}$ over $\Fc A$}\label{sssec: wt not}$\;$

Let $X$ be one of the stacks or schemes over $A(\omega_{C'}(D'))$ that we have introduced above. For example, $X$ can be $\sM_{Dol}(C',D'),\sP, M_{Dol}^{ss}(C,D)$, etc.
\begin{notn}
We denote by $\wt{X}$ the $\Fc A$-stack or scheme obtained by taking the base change $X\times_{A(\omega_{C'}(D')} \Fc A$.
We denote by $h_{Dol}^{\Fc A}:\wt{\sM_{Dol}}(C',D')\to \Fc A$ the morphism given by the base change of $h_{Dol}:\sM_{Dol}(C',D')\to A(\omega_{C'}(D'))$.
We call the morphism  $h^{\Fc A}_{Dol}$ 
the Hitchin-residue morphism.
\end{notn}

By base changing the action \eqref{eq: pl act on mdoll}, we have an action
    \begin{equation}
    \label{eq: wt P act on wt mdol}
        \wt{\sP}\times_{\Fc A}\wt{\sM}_{Dol}(C',D')\to \wt{\sM}_{Dol}(C',D').
    \end{equation}

\subsubsection{Refined picture in the vector bundle case}\label{subs: var res}$\;$

When $G=GL_r$, by taking into account the Fuchs relations on residues, we have a more refined picture.
Recall the notation  in diagram \eqref{diag: hdr fiber not conn}. 
The pre-image  via the morphism $\Fas$ of the 
tracelss part of $\Gamma(D', \Fc_{D'})^{\ov{0}}$ is the disjoint union of the affine hyperplanes
$\Gamma (D, \Fc_D)^{\delta}$,
with $\delta \in \bF_p$ ranging in the residual classes modulo $p$
(cf. \eqref{diag: hdr fiber not conn}).

We then have the following commutative diagram  with a Cartesian square
\begin{equation}
     \label{diag: hcdr bis}
         \xymatrix{
    \sM_{dR}(C,D,d) \ar@/_/[ddr]_-{\chi_D\circ \mathrm{res}_D} \ar@/^/[drr]^-{h_{dR}}\ar[dr]_(.65){h_{dR}^{\Fc A}}
    &
    &
    \\
    &
    \mathfrak{c}A^{-\overline{d}} \ar[r]_-{pr_1} \ar[d]^-{pr_2}
    &
    A(\omega_{C'}(D'))
    \ar[d]^-{ev_{D'}}
    \\
    &
\Gamma(D,\mathfrak{c}_D)^{-\overline{d}} \ar[r]_-{\Fas}
    &
    \Gamma(D',\mathfrak{c}_{D'})^{\ov{0}},
    }
 \end{equation}
where we name the fiber product $\Fc A^{-\ov{d}}$ the  variety of  residues for rank $r$  degree $d$ connections (it depends only on the residual class of $-d$ modulo $p$),
and where the  morphism  $h_{dR}^{\Fc A}=h_{dR}^{\Fc A}(d)$, which we call the de Rham-Hitchin-residue
morphism, is induced from the morphisms $h_{dR}$ and $\chi_D\circ \mathrm{res}_D$, in view of the square
being Cartesian.
Note that by \eqref{diag: hdr fiber not conn}, the morphism $\chi_D\circ \mathrm{res}_D$,
restricted to the degree $d$ component
$\sM_{dR}(C,D,d)$, surjects onto $\Gamma(D, \Fc_D)^{-\overline{d}}$.

 We have the de Rham-Hitchin-residue morphism on semistable  moduli spaces, which we also denote by 
\begin{equation}\label{eq: drres st}
h_{dR}^{\Fc A}  = h_{dR}^{\Fc A}(d): M_{dR}^{ss}(C,D,d)\to \Fc A^{-\overline{d}}.
\end{equation}
The morphism \eqref{eq: drres st}, being the first link leading to the proper morphism $h_{dR}$ (cf. \cite[Thm. 5.9]{langer2021moduli}) is proper.
 The morphism $h_{dR}^{\Fc A}(d)$ is surjective by \Cref{prop: surj of drhr}. The morphism $h_{dR}^{\Fc A}(pd)$ has geometrically connected and equidimensional fibers by \Cref{prop: surj and conn fib} and \Cref{cor: fiber dim of hdrc}. When $(pd,r)=1$, the morphism $h_{dR}^{\Fc A}$ is flat by \Cref{cor: hdr flat}.

\begin{remark}
    We have  $\Fc A = \coprod_{\delta \in \bF_p} \Fc A^{\delta}$. 
\end{remark}

The stack
$\widetilde{\mathscr M}_{Dol}(C',D')$ fits into the following base change Cartesian diagram 
\begin{equation}
    \label{diag: ca and wtmdol}
         \xymatrix{
\widetilde{\mathscr M}_{Dol}(C',D')
=  \coprod_{\delta \in\bF_p, d \in\bZ}
\widetilde{\mathscr M}_{Dol}(C',D',d)^{\delta}
\ar[r]^-{\xi}
\ar[d]_-{h^{\Fc A}_{Dol} =
\coprod_{\delta \in\bF_p,  d \in \bZ} h^{{\Fc A}}_{Dol} (d)^{\delta}} 
    &
    \mathscr{M}_{Dol}(C',D')\ar[d]^-{h_{Dol}}
    \\
    \Fc A =
 \coprod_{\delta \in\bF_p}
 \Fc A^{\delta} \ar[r]^-{ pr_1=\coprod_{\delta} pr_1^\delta}\ar[d]^-{pr_2}
    &
    A(\omega_{C'}(D')) \ar[d]^-{ev_{D'}}
    \\ 
\coprod_{\delta \in \bF_p} \Gamma (D, \Fc_D)^{\delta}
    \ar[r]^-{\Fas} & \Gamma(D',\Fc_{D'})^{\ov{0}},
    }
    \end{equation}
where all arrows are surjective, the projection morphism $pr_2$ preserves
the disjoint union decomposition labelled by $\delta$,
and the variety  $\Fc A$ is nonsingular
and a disjoint union of the  equidimensional affine spaces
$\Fc A^{\delta}$.
It is well-known that the Hitchin morphism is surjective in each degree. It follows that for every fixed degree $d$, and every residue class $\delta$, $\widetilde{\mathscr M}_{Dol}(C',D',d)^{\delta}$ surjects onto $\Fc A^{\delta}$.
The morphisms $pr_1^\delta$
are isomorphic, so that 
for $d$ fixed and varying $\delta$, there is no significant difference between
the 
$\widetilde{\mathscr M}_{Dol}(C',D',d)^{\delta}$,
as $\delta$ varies in $\bF_p$.

The Picard stack $\widetilde{\mathscr{P}}$ over $\mathfrak{c}A$ admits the decomposition
$\widetilde{\mathscr{P}}=
\coprod_{\delta,d} \widetilde{\mathscr{P}} (d)^\delta$
labelled by  $d\in \bZ$ and by $\delta \in \bF_p$.
The action \eqref{eq: wt P act on wt mdol} is compatible with the decomposition labelled by $\delta \in \bF_p$.

Let $M_{Dol}^{ss}(C',D',d)$ be the adequate moduli space of the semistable substack $\sM_{Dol}^{ss}(C',D',d)$. 
Let $\wt{M}_{Dol}^{ss}(C',D',d)$ be the fiber product 
\begin{equation}\label{eq: try1}
      \xymatrix{
    \widetilde{M}_{Dol}^{ss}(C',D',d) \ar[r]^-{\xi}\ar[d]_-{h^{\Fc A}_{Dol}}
    &
    M_{Dol}^{ss}(C',D',d)\ar[d]^-{h_{Dol}}
    \\
    \mathfrak{c}A \ar[r]^-{pr_1} 
    &
    A(\omega_{C'}(D')),
    }
\end{equation}
where, by a slight abuse of notation,
we have denoted the resulting  Hitchin residue morphism $h^{\Fc A}_{Dol}$ and the morphism $\xi$ above by the same notation as their stacky counterpart. 
 This Hitchin residue morphism $h^{\Fc A}_{Dol}$ is proper, surjective and with connected fibers, because so is the Hitchin morphism $h_{Dol}$.
 The l.h.s. column in \eqref{eq: try1} is compatible with the decomposition labelled by $\delta \in \bF_p$
(cf. \eqref{diag: ca and wtmdol}).

Since adequate moduli spaces are compatible with flat base change, we have that  the fiber product
$\widetilde{M}_{Dol}^{ss}(C',D',d)$ is also the adequate moduli space of 
\[
\widetilde{\sM}_{Dol}^{ss}(C',D',d):=\sM_{Dol}^{ss}(C',D',d)\times_{A(\omega_{C'}(D'))} \Fc A,
\]
compatibly with the decomposition labelled by $\delta \in \bF_p$.

\section{Log-$p$-Non-Abelian Hodge Theorem}\label{sec: true lognah}

\subsection{The pseudo-torsor $\sH$}\;

In this subsection, we first recall the construction of a stack $\sH$ over $A(\omega_{C'}(D'))$.
In the no pole case, the analogous stack is constructed in \cite[\S3.2]{chen-zhu} and used in \cite[\S5.3]{chzh17} and \cite[\S2.2]{herrero-zhang-semistable}.
A similar construction in the no pole case also appeared in the $p$-adic context in \cite[\S6.1]{heuer-xu2024padic}, where connections are replaced by $\mathrm{v}$-bundles.
The no pole case analogue of $\sH$ is a $\sP(\omega_{C'})$-torsor over $A(\omega_{C'})$.
In the log-case, the stack $\sH$ is constructed in \cite[\S5.4]{li2024tame}.
In the log-case, $\sH$ is not even a $\sP$-pseudo-torsor.
Our main contribution in this subsection is to show that $\sH$ actually sits over $\Fc A$ and is a pseudo-torsor under the action of $\wt{\sP}=\sP\times_{A(\omega_{C'}(D')}\Fc A$, see \Cref{lemma: unique res}.

\subsubsection{Lie-theoretic setup}\label{sssec: lie the set}$\;$

Let $I\subset \Fg\times G$ be the universal centralizer group scheme over $\Fg$, i.e. for each $x\in \Fg$, the fiber $I_x$ is the stabilizer of $x$ for  the $G$-conjugation-action on $\Fg$.
Recall the Chevalley quotient morphism $\chi:\Fg\to \Fc$. 
Let $\Fg^{reg}\subset \Fg$ be the regular locus. It is an open subscheme.
We have that $\Fg^{reg}$ is stable under the $G$-action, and that $\Fg^{reg}\git G=\Fc$ \cite[\S7.13]{jantzen2004nilpotent}.

By \cite[Lem. 2.1.1]{ngo-lemme-fondamental}, there is a canonical morphism of $\Fg$-group schemes $c: \chi^*J\to I$ which restricts to an isomorphism on $\Fg^{reg}$; the group scheme $J$ is the group scheme of regular centralizers.
Let $\mathrm{Lie}(I)\to \Fg$ be the Lie algebra of $I\to \Fg$ and let $\mathrm{Lie}(J)\to \Fc$ be the Lie algebra of $J\to \Fc$.
The differential  $dc: \chi^*\mathrm{Lie}(J)\to \mathrm{Lie}(I)$ of the morphism $c$ is a morphism of $p$-Lie algebras.

\subsubsection{The Chen-Zhu section $\tau$}\label{sssec: cz sec}$\;$

In \cite[\S2.3]{chen-zhu}, a section $\tau:\Fc\to \mathrm{Lie}(J)$ is constructed. 
It is obtained by descending the tautological section $x\mapsto x: \Fg^{reg}\to \mathrm{Lie}(I)|_{\Fg^{reg}}$ to $\Fc$.
By \cite[Lem. 2.2]{chen-zhu}, for each $x\in \Fg$, the differential $dc_x: \mathrm{Lie}(J_{\chi(x)})\to \mathrm{Lie}(I_x)$ sends $\tau(\chi(x))$ to $x$.

\subsubsection{$\bG_m$-Equivariancy}$\;$

The diagonal $\bG_m$-action on $\mathrm{Lie}(I)\subset \Fg\times\Fg$ induces a $\bG_m$-action on $\mathrm{Lie}(J)$ over $\Fc$.
The section $\tau$ is $\bG_m$-equivariant and so we can twist it by the $\bG_m$-torsor given by an invertible sheaf $L$ on $C$ and obtain a section $\tau_L: \Fc_L\to \mathrm{Lie}(J)_L$ of $C$-schemes. 
Recall that for $a: C\to \Fc_L$, we have $J_a:=a^*J_L$.
By \cite[Lem. 2.4]{herrero-zhang-semistable}, we have $a^*(\mathrm{Lie}(J))_L=\mathrm{Lie}(J_a)_L$ over $C$.

\subsubsection{$G$-equivariancy and a Lie-theoretic lemma}$\;$

Since $\Fc=\Fg\git G$, the $\Fg$-group scheme $\chi^*J$ is $G$-equivariant.
The conjugation action of $G$ on $G$ induces a $G$-action on $I\subset \Fg\times G$ such that both $I\to \Fg$ and $\chi^*J\to I$ are $G$-equivariant.
Therefore we have the descended morphism $c: \chi^*J/G\to I/G$ on $\Fg/G$, which enables us to define the $\sP$-action on $\sM_{Dol}$ as in \Cref{sssec: action of p on mdol}.
We record another consequence of the $G$-equivariancy, which is used in the proof of \Cref{thm: main thm general g}.
\begin{lemma}\label{lemma: g equiv}
Let $a\in \Fc$.
    Let $x,y\in \chi^{-1}(a)\subset \Fg$.
    Let $z\in \mathrm{Lie}(J_a)$.
    Then $dc_x(z), dc_y(z)\in \Fg$ have the same image under $\chi:\Fg\to\Fc$.
\end{lemma}
\begin{proof}
    The $G$-equivariance of $c: \chi^*J\to I$ entails that the assignment $x\mapsto dc_x(z)$ defines a $G$-equivariant morphism $\lambda_z:\chi^{-1}(a)\to \Fg$. 
    The morphism $\chi:\Fg\to \Fc$ is $G$-equivariant for the trivial $G$-action on $\Fc$.
    Since $\chi^{-1}(a)$ is connected \cite[Thm. 7.13]{jantzen2004nilpotent}, we have that the image of the composition $\chi^{-1}(a)\xrightarrow{\lambda_z}\Fg\xrightarrow{\chi}\Fc$ is just one point.
\end{proof}

\subsubsection{}

We record another Lie-theoretic lemma, to be used in the proof of \Cref{lemma: unique res}.
Let $x\in \Fg^{reg}$ be a regular element.
Let $I_x\subset G$ be the centralizer of $x$.
We have that $I_x$ is a smooth commutative subgroup of $G$ \cite[Cor. 3.3.6]{riche2017kostant}.
The Lie algebras $\mathrm{Lie}(I_x)\subset\Fg$ are $p$-Lie algebras with compatible $p$-operations $(-)^{[p]}$.
Let $AS: \Fg\to \Fg$ be the morphism given by $x\mapsto x^{[p]}-x$.
Let $\chi:\Fg\to \Fc=\Fg\git G$ be the quotient morphism.

Already for $G=GL_2$ and regular nilpotent $x$, the morphism $\chi|_{\mathrm{Lie}(I_x)}: \mathrm{Lie}(I_x)\to \Fc$ is not injective.
In contrast, let $\tau\in \mathrm{Lie}(I_x)$ be an element, and
let $(AS|_{\mathrm{Lie}(I_x)})^{-1}(\tau)\subset \mathrm{Lie}(I_x)$ be the preimage of $\tau$ under the morphism $AS|_{\mathrm{Lie}(I_x)}: \mathrm{Lie}(I_x)\to \mathrm{Lie}(I_x)$, we have

\begin{lemma}
\label{lemma: lie lemma}
    The morphism $\chi|_{AS^{-1}(\tau)}: (AS|_{\mathrm{Lie}(I_x)})^{-1}(\tau)\to \Fc$ is injective.
\end{lemma}
\begin{proof}
    Let $A,B\in (AS|_{\mathrm{Lie}(I_x)})^{-1}(\tau)\subset \mathrm{Lie}(I_x)$ with $\chi(A)=\chi(B)$.  We are done if we can show that $A=B$.

    Since $I_x$ is commutative and $k$ is perfect, there is a unique product decomposition of group schemes $I_x=(I_x)_u\times (I_x)_s$, where $(I_x)_u$ is unipotent and $(I_x)_s$ is of multiplicative type \cite[Thm. 16.13]{milne2017algebraic}.
    Therefore, we have a corresponding product decomposition of Lie algebras 
    \begin{equation}
    \label{eq: lie algebra decompose}
        \mathrm{Lie}(I_x)=\mathrm{Lie}(I_x)_n\oplus \mathrm{Lie}(I_x)_s,
    \end{equation} 
    corresponding to the nilpotent and semisimple parts, which coincides with the Jordan decomposition of each element in $\mathrm{Lie}(I_x)$, \cite[top of p.328]{milne2017algebraic}.
    Consider the Jordan decompositions $A=A_s+A_n$, $B=B_s+B_n$, $\tau=\tau_s+\tau_n$, where $(-)_s$ (resp. $(-)_n$) denotes the semisimple (resp. nilpotent) part.
    Since $A,B\in (AS|_{\mathrm{Lie}(I_x)})^{-1}(\tau)\subset \mathrm{Lie}(I_x)$, we have that $A^{[p]}-A=B^{[p]}-B=\tau$.
    By \eqref{eq: lie algebra decompose}, the Jordan decomposition in $\mathrm{Lie}(I_x)$ respects addition.
    By \cite[Ex. 4.4.21]{springer1998linear}, the Jordan decomposition also respects the $p$-operation.
    Therefore, we have the identities:
    \begin{equation}
    \label{eq: four eqs}
        A_s^{[p]}-A_s\stackrel{\mathrm{(i)}}{=}B_s^{[p]}-B_s\stackrel{\mathrm{(ii)}}{=}\tau_s, \quad
        A_n^{[p]}-A_n\stackrel{\mathrm{(iii)}}{=}B_n^{[p]}-B_n\stackrel{\mathrm{(iv)}}{=}\tau_n.
    \end{equation}

Since $(I_x)_u$ is unipotent, it is isomorphic to a subgroup of the group of upper triangular matrices with size $r$ with 1 on the diagonal.
Therefore, the $p$-Lie algebra $\mathrm{Lie}(I_x)_n$ is isomorphic to a sub-$p$-Lie algebra of $\mathfrak{n}_r$, the nilpotent upper triangular matrices \cite[Thm. 14.5]{milne2017algebraic}.
The $p$-operation on $\mathfrak{n}_r$ raises a matrix to its $p$-th power \cite[Ex. 4.4.10]{springer1998linear}.
A matrix calculation shows that the morphism of $k$-varieties $AS_{\mathfrak{n}_r}: \mathfrak{n}_r\to \mathfrak{n}_r$, given by $x\mapsto x^{[p]}-x$,
is injective.
Therefore the Artin-Schreier map on $\mathrm{Lie}(I_x)_n$ is also injective.
 \eqref{eq: four eqs}.(iii) then entails that $A_n=B_n$.

It remains to show that $A_s=B_s$.
Since $I_x$ is smooth, the group $(I_x)_s$ is a torus \cite[Cor. 16.15]{milne2017algebraic}.
Since $k$ is algebraically closed, the group $(I_x)_s$ is conjugated to a subgroup of $T$, the maximal torus fixed at the beginning \cite[Thm. 17.10]{milne2017algebraic}.
Without loss of generality, we can assume that $(I_x)_s$ is a subtorus of $T$, hence $A_s, B_s\in \Ft$.
By \cite[Prop. 2.11]{jantzen2004nilpotent}, we have that $\chi(A)=\chi(A_s)$ and $\chi(B)=\chi(B_s)$.
Since $\chi(A)=\chi(B)$ by assumption, we have that $\chi(A_s)=\chi(B_s)$.
Since $\Fc=\Ft\git W$ by assumption on $p$ and $G$, there is an element $\sigma\in W$ such that 
\begin{equation}
\label{eq: sigma a b}
    \sigma\cdot A_s=B_s.
\end{equation}

To derive a contradiction, assume that $A_s\neq B_s$. Thus $\sigma\neq 1$.

By \cite[Prop. 4.4.9]{springer1998linear}, the $W$-action on $\Ft$ is compatible with the $p$-operation.
Therefore, \eqref{eq: four eqs}.(i,ii) and \eqref{eq: sigma a b} entail that $\sigma\cdot \tau_s=\sigma\cdot (A_s^{[p]}-A_s)=B_s^{[p]}-B_s=\tau_s$.
By \cite[Ch. V, \S5, Ex. 8]{bourbaki-lie}, the stabilizer $W_{\tau_s}$ is generated by the reflections contained in it.
Therefore, there is a reduced decomposition of $\sigma$ into simple reflections $\sigma=s_{\alpha_1}...s_{\alpha_n}$ such that each $\alpha_i$ is a simple root and that $s_{\alpha_i}$ fixes $\tau_s$.

Let $C_{\Fg}(\tau_s)\subset \Fg$ be the centralizer of $\tau_s$ in $\Fg$. 
Let $\Fg=\Ft\oplus \bigoplus_{\alpha\in\Phi}\Fg_{\alpha}$ be the root space decomposition.
For any $x_{\alpha}\in \Fg_{\alpha}$, we have that $[\tau_s,x_{\alpha}]=d\alpha(\tau_s)x_{\alpha}$, where $d\alpha:\Ft\to \Fgl_1$ is the differential of the root morphism $\alpha: T\to \bG_m$.
Therefore, we have that $C_{\Fg}(\tau_s)=\Ft\oplus \bigoplus_{\alpha\in \Phi, d\alpha(\tau_s)=0}\Fg_{\alpha}$.
Since $s_{\alpha_i}(\tau_s)=\tau_s$, we have that $d\alpha_i(\tau_s)=0$ by \cite[Ch. VIII, \S2.2, Lem. 1.(1)]{bourbaki2004lie}, which gives the equality of self-maps  $s_{\alpha_i}=\mathrm{Id}-d\alpha_i(-)H_{\alpha_i}$ on $\Ft$, where $H_{\alpha_i}$ is a uniquely determined nonzero semisimple element in the $\Fsl_2$ associated with $\alpha_i$. 
Therefore, each $\Fg_{\alpha_i}$ is contained in $C_{\Fg}(\tau_s)$.
Since $\tau\in \Fg^{reg}$, we have that $\tau_n=A_n^{[p]}-A_n$ is nilpotent regular in $C_{\Fg}(\tau_s)$ \cite[Prop. 0.4]{kostant1963lie}.
It then follows from \cite[\S6.7, (1)]{jantzen2004nilpotent},  that for each $i=1,...,n$, $\tau_n$ has nonzero components in $\Fg_{\alpha_i}$ or $\Fg_{-\alpha_i}$ (it can't be both, since $\tau_n$ is nilpotent).
Without loss of generality, we can assume that $\tau_n$ has nonzero components in each $\Fg_{\alpha_i}$.
Using the formulas in \cite[p.279, Rmk. (3)]{mcninch2002abelian} and \eqref{eq: four eqs}.(iii,iv), we conclude that $A_n$ must also have nonzero components in each $\Fg_{\alpha_i}, i=1,...,n$.

On the other hand, since $\sigma\cdot A_s=B_s\neq A_s$, there is $j\in\{1,...,n\}$ such that $s_{\alpha_j}\cdot A_n\neq A_n$.
Applying \cite[Ch. VIII, \S2.2, Lem. 1.(1)]{bourbaki2004lie} again, we have that $d\alpha_j(A_s)\neq 0$, thus $\Fg_{\alpha_j}$ is not in $C_{\Fg}(A_s)$.
Therefore, $A_n$, being in $C_{\Fg}(A_s)$, cannot have a nonzero component in $\Fg_{\alpha_j}$, and so we obtain a contradiction with the conclusion of the last paragraph above.
\end{proof}

\subsubsection{The $J$-Hitchin system} $\;$

As defined in \cite[Def. 2.5]{herrero-zhang-semistable}, the $J$-Hitchin system is the following diagram of $A(L)$-stacks:
\begin{equation}
\label{eq: j-hitchin}
    \xymatrix{
    \sM_{Dol}(C,J,L) \ar[rr]^-{h^J}&& A(J,L) \ar@/^/[rr]^-{p} && A(L) \ar@/^/[ll]^-{\tau}.
} 
\end{equation}
Namely, for any $k$-scheme $S$,  over each $a\in A(L)(S)$, the fiber of \eqref{eq: j-hitchin} is the diagram of groupoids
\begin{equation}
    \xymatrix{
    \Gamma(C_S, (\mathrm{Lie}(J_a)/J_a)_L)\ar[rr]^-{\chi_{L,a}}&& \Gamma(C_S, \mathrm{Lie}(J_a)_L) \ar@/^/[rr]^-{p_a} && S \ar@/^/[ll]^-{\tau_L(a)},
} 
\end{equation}
where $\chi_{L,a}$ is induced by the quotient morphism $\mathrm{Lie}(J_a)/J_a\to \mathrm{Lie}(J_a)\git J_a=\mathrm{Lie}(J_a)$ (the equality holds because $J_a$ is commutative); $p_a$ is the structural morphism; $\tau_L(a)$ is induced by the section $\tau_L:\Fc_L\to \mathrm{Lie}(J)_L$.

\subsubsection{The $J$-de Rham-Hitchin morphism}\label{sssec: constr of h}$\;$

We freely use the formalism of log-connections on group schemes and torsors developed in \Cref{section: two views on residues}.
Let $a': C'\to \Fc_{\omega_{C'}(D')}$ be a section.
Let $a^p: C\to \Fc_{\omega_{C}^p(pD)}$ be the Frobenius pullback of $a'$.
The group scheme $J_{a^p}$ is isomorphic to $Fr^*J_{a'}$, and thus has a canonical connection $\nabla^{can}$ \cite[Ex. A.5]{chen-zhu}.

Let $\sM_{dR}(C,D,J_{a^p})$ be the groupoid of log-$J_{a^p}$-connections on $C$.
Taking the $p$-curvatures, we obtain a morphism of groupoids $\Psi_{a^p}:\sM_{dR}(C, D, J_{a^p})\to \Gamma(C, (\mathrm{Lie}(J_{a^p})/J_{a^p})_{\omega_{C}^p(pD)})$.
The composition $\chi_{\omega_{C}^p(pD),a^p}\circ \Psi_{a^p}: \sM_{dR}(C,D,J_{a^p})\to \Gamma(C, \mathrm{Lie}(J_{a^p})_{\omega_{C}^p(pD)})$ factors as the composition of the injective Frobenius pullback morphism $Fr^*: \Gamma(C', \mathrm{Lie}(J_{a'})_{\omega_{C'}(D')})\hookrightarrow \Gamma(C, \mathrm{Lie}(J_{a^p})_{\omega_{C}^p(pD)})$
and a morphism $h_{dR}^{J_{a^p}}: \sM_{dR}(C,D,J_{a^p})\to \Gamma(C', \mathrm{Lie}(J_{a'})_{\omega_{C'}(D')})$.
Indeed, the proof of the factorization in the no pole case in \cite[Lem. 2.9]{herrero-zhang-semistable} carries over  verbatim to the log-case.

The groupoid $\sM_{dR}(C,D,J_{a^p})$ gives rise to an $A(\omega_{C'}(D'))$-stack $\sM_{dR}(C,D,J^p)$.
The set $\Gamma(C',\mathrm{Lie}(J_{a'})_{\omega_{C'}(D')})$ gives rise to the $A(\omega_{C'}(D'))$-functor $A(J,\omega_{C'}(D'))$ as in \eqref{eq: j-hitchin}.
Finally, the morphism $h_{dR}^{J_{a^p}}$ gives rise to a morphism of $A(\omega_{C'}(D'))$-stacks
\begin{equation}\label{eq: hjdr}
h_{dR}^J: \sM_{dR}(C,D,J^p)\to A(J,\omega_{C'}(D')).
\end{equation}

\begin{defn}\label{defn: def H}
    We define the $A(\omega_{C'}(D'))$-stack $\sH$ so that the following diagram of $A(\omega_{C'}(D'))$-stacks is Cartesian:
    \begin{equation}\label{diag: def H}
        \xymatrix{
        \sH\ar[rr]^-{\pi_1}\ar[d]_-{\pi_2} && \sM_{dR}(C,D,J^p)\ar[d]^-{h_{dR}^J}\\
        A(\omega_{C'}(D'))\ar[rr]_-{\tau_{\omega_{C'}(D')}} && A(J,\omega_{C'}(D')).
        }
    \end{equation}
\end{defn}

\subsubsection{Description of objects in $\sH$}
\label{sssec: description of H}$\;$

Let $R$ be a $k$-scheme.
By construction, given $a'\in A(\omega_{C'}(D')(R))$, an object in the groupoid $\sH_{a'}$ is a pair $(E,\nabla)$ where $E$ is a $J_{a^p}$-torsor on $C_R$, $\nabla$ is a log-connection on $E$ such that the $p$-curvature $\Psi(\nabla)\in \Gamma(C_R, \mathrm{Lie}(J_{a^p})_{\omega_C^p(pD)})$ is the Frobenius pullback $Fr^*$ of 
the section $\tau_{\omega_{C'}(D')}(a')\in \Gamma(C'_R,\mathrm{Lie}(J_{a'})_{\omega_{C'}(D')})$.

Roughly speaking, the next lemma says that the action in question is free.
\begin{lemma}\label{lemma: faithful action}
    There is a $\sP$-action on $\sH$, $act:\sP\times_{A(\omega_{C'}(D'))}\sH\to \sH$.
    Furthermore,  for each object $X$ in $\sH$, the induced morphism $act(-,X):\sP\to \sH$ is fully faithful.
\end{lemma}
\begin{proof}
        It suffices to work over a point $a'\in A(\omega_{C'}(D'))(R)$.
    Let $\underline{E}:=(E,\nabla_E)$ be an object in $\sH_{a'}$.
    An object in $\sP_{a'}$ is a $J_{a'}$-torsor $F$ on $C'_R$.
    The pullback $Fr^*F$ is a $J_{a^p}$-torsor with a canonical connection $\nabla^{can}_{Fr^*F}$.
    The action of $\sP$ on $\sH$ sends the pair $((E,\nabla_E),F)$ to the log-connection $\underline{E}\cdot F:=(E\times^{J_{a^p}} Fr^*F,\nabla_E\otimes \nabla^{can}_{Fr^*F})$.
    Since $\Psi(\nabla^{can}_{Fr^*F})=0$, \eqref{eq: pcuvr twi prod} entails that the log-connection $\underline{E}\cdot F$ is still over $a'$, therefore the action is well-defined.

    Let $F,F'$ be  $J_{a'}$-torsors.
    Let $\underline{E}^{\vee}$ be the dual log-connection $(E^{\vee},\nabla_{E^{\vee}})$, see \Cref{sssec: dual log conn}.
    Then we have the following isomorphisms of sets:
    \begin{align*}
        & Hom_{\sH_{a'}}(\underline{E}\cdot F,\;\; \underline{E}\cdot F') \\
        =& Hom_{\sM_{dR}(J_{a^p})}(\underline{E}^{\vee}\cdot\underline{E}\cdot F, \;\;\underline{E}^{\vee}\cdot \underline{E}\cdot F')\\
        =& Hom_{\sM_{dR}(J_{a^p})}\Big((Fr^*F,\nabla^{can}_{Fr^*F}),\;\; (Fr^*F',\nabla^{can}_{Fr^*F'})\Big)\\
        =& Hom_{\sP_{a'}}(F,\;\; F'),
    \end{align*}
        where the second equality follows from \Cref{lemma: dual is dual} and the last equality follows from Cartier Descent for torsors \cite[Ex. A.5]{chen-zhu}.
We have shown the fully faithful claim.

\end{proof}

\begin{remark}\label{rmk: preserves residue}
Note that the $\sP$-action on $\sH$ preserves the residues.
    Indeed, $\nabla_{Fr^*F}^{can}$ has no poles, so the product formula \eqref{eq: res twi prod} entails that $\mathrm{res}_D(\nabla_{E}\otimes \nabla_{Fr^*F}^{can})=\mathrm{res}_D(\nabla_E)\in \Gamma(D, \mathrm{Lie}(J_{a^p})|_D)$.
\end{remark}

\subsubsection{The morphism $\pi^{\Fc A}$}$\;$

Recall the morphisms $\pi_2$ in \eqref{diag: def H} and $pr_1$ in \eqref{diag: base cA}.
\begin{prop}\label{prop: def pi ca}
    The structural morphism $\pi_2:\sH\to A(\omega_{C'}(D'))$ factors through
    \begin{equation}\label{eq: def pi ca}
        \sH\xrightarrow{\pi^{\Fc A}}\Fc A\xrightarrow{pr_1} A(\omega_{C'}(D')).
    \end{equation}
\end{prop}
\begin{proof} 
    Let $(E,\nabla)$ be an object in $\sH$ over $a'$ in $A(\omega_{C'}(D'))$.
    Its residue $\mathrm{res}_D(\nabla)$ is a section in $\Gamma(D, \mathrm{Lie}(J_{a^p})|_D)$.
    Under the canonical trivialization $\omega_C^p(pD)|_D=\cO_D$, the section $a^p|_D\in \Gamma(D, \Fc_{\omega_C^p(pD)})$ is identified with a section $\alpha_D\in \Gamma(D, \Fc_D)$.
    For each point $x\in D$, we have $\alpha_D(x)\in \Fc$.
    Consequently, $\mathrm{Lie}(J_{a^p})|_D$ is identified with the $D$-Lie algebra $\mathrm{Lie}(J_{\alpha_D})$ whose fiber over $x\in D$ is $\mathrm{Lie}(J_{\alpha_D(x)})$.
    Let $\kappa: \Fc\to \Fg$ be the Kostant section.
    The morphism $\mathrm{Lie}(J_{\alpha_D(x)})\xrightarrow{dc_{\kappa(\alpha_D(x))}} \mathrm{Lie}(I_{\kappa(\alpha_D(x))})\hookrightarrow\Fg$ as in \Cref{sssec: lie the set} assembles to a morphism of $D$-Lie algebras $dc_{\kappa(\alpha_D)}: \mathrm{Lie}(J_{\alpha_D})\to \Fg_D$.
    Let the morphism $\Xi_{a'}$ be the composition of following morphisms:
    \begin{equation}\label{eq: Xi a'}
       \Xi_{a'}: \sH_{a'}\xrightarrow{\mathrm{res}_D}\Gamma(D,\mathrm{Lie}(J_{a^p})|_D)=\Gamma(D, \mathrm{Lie}(J_{\alpha_D}))\xrightarrow{dc_{\kappa(\alpha_D)}} \Gamma(D,\Fg_D) \xrightarrow{\chi_D} \Gamma(D, \Fc_D).
    \end{equation}
    Letting $a'$ vary, we obtain a morphism $\Xi: \sH\to \Gamma(D,\Fc_D)$.
    To obtain the desired factorization \eqref{eq: def pi ca}, we need to show the commutativity of the following square:
    \begin{equation}
    \xymatrix{
        \sH\ar[d]_-{\Xi} \ar[r]^-{\pi_2} & A(\omega_{C'}(D')) \ar[d]^-{ev_{D'}} \\
        \Gamma(D,\Fc_D) \ar[r]_-{\Fas} & \Gamma(D',\Fc_{D'}).
        }
    \end{equation}
    
    For each $x\in D$, the morphism $dc_{\kappa(\alpha_D(x))}: \mathrm{Lie}(J_{\alpha_D(x)})\to \Fg$ is respects the $p$-Lie algebra structures.
    Furthermore, the morphism $AS_{\Fg}:\Fg\to \Fg,$ given by setting $ X\mapsto X^{[p]}-X$ descends to the morphism $AS_{\Fc}:\Fc\to \Fc$.
     By combining the discussion above with \Cref{sssec: description of H}, we have that the image of $\Fas\circ \Xi_{a'}$ is the image of $\tau_{\omega_{C'}(D')}(a')\in \Gamma(C', \mathrm{Lie}(J_{a'})_{\omega_{C'}(D')})$ under the composition of the following morphisms:
    \begin{equation}
        \Gamma(C', \mathrm{Lie}(J_{a'})_{\omega_{C'}(D')})\xrightarrow{f_1} \Gamma(D', \mathrm{Lie}(J_{a'})|_{D'})\xrightarrow{f_2} \Gamma(D,\Fg_D)\xrightarrow{\chi_D} \Gamma(D,\Fc_D),
    \end{equation}
    where $f_1$ is defined as the restriction to $D'$, followed by the canonical residue isomorphism $\omega_{C'}(D')|_{D'}=\cO_{D'}$, and $f_2$ is defined in a similar way as $dc_{\kappa(\alpha_D)}$ above.
    The desired commutativity then follows from \cite[Lem. 2.2]{chen-zhu}, which entails that the image of $\tau_{\omega_{C'}(D')}(a')$ under $f_2\circ f_1$ is given by the Kostant image of $a'|_{D'}\in \Gamma(D',\Fc_{D'})$.
\end{proof}

\subsubsection{$\pi^{\Fc A}$ and residues} $\;$

From now on we usually view $\sH$ as a $\Fc A$-stack via the morphism $\pi^{\Fc A}$ \eqref{eq: def pi ca}.
Let $(E,\nabla)$ be an object in $\sH_{a'}$, notation as in \Cref{sssec: description of H}.
The residue $\mathrm{res}_D(\nabla)$ is a section in $\Gamma(D,\mathrm{Lie}(J_{a^p})_D)$, see \Cref{sssec: residues}.

The following lemma is crucial in the upcoming proof of \Cref{prop: H pstorsor}. Recall \Cref{defn: base cA}.

\begin{lemma}
\label{lemma: unique res}
    For each point $(a',s)\in \Fc A=A(\omega_{C'}(D'))\times_{\Gamma(D',\Fc_{D'})}\Gamma(D,\Fc_D)$,
the log-connections in the fiber $\sH_{a',s}$ have the same residue in $\Gamma(D, \mathrm{Lie}(J_{a^p})_D)$.
\end{lemma}
\begin{proof}
We can assume that $D=x$ is a point.
By the construction of $\pi^{\Fc A}$ in the proof of \Cref{prop: def pi ca}, the objects in $\sH_{s,a'}$ are $J_{a^p}$-log-connections $(E,\nabla)$ whose residue $\mathrm{res}_D(\nabla)$ satisfies $\chi_D\circ dc_{\kappa(\alpha_D(x))}(\mathrm{res}_D(\nabla))=s$.
Since $\kappa(\alpha_D(x))\in \Fg^{reg}$, the morphism $dc_{\kappa(\alpha_D(x))}: \mathrm{Lie}(J_{\alpha_D(x)})\to \mathrm{Lie}(I_{\kappa(\alpha_D(x))})$ is an isomorphism.
We are reduced to showing that there is a unique $A\in \mathrm{Lie}(J_{\alpha_D(x)})=\mathrm{Lie}(I_{\kappa(\alpha_D(x))})$ such that $A^{[p]}-A=(Fr^*\tau_{\omega_{C'}(D')}(a'))|_{D}\in \Gamma(D, \mathrm{Lie}(J_{\alpha_D(x)}))$ and that $\chi_D(A)=s\in \Gamma(D,\Fc_D)$.
We can now conclude by the Lie-theoretic \Cref{lemma: lie lemma}.
\end{proof}

\subsubsection{Pseudo-torsorness}$\;$

The $\sP$-action on $\sH$ over $A(\omega_{C'}(D'))$ is not transitive since objects in $\sH_{a'}$ may have different residues, e.g. in the $GL_{r>p}$-case (see \Cref{lemma: compare cashen}), while the action of $\sP$ preserves residues (see \Cref{rmk: preserves residue}).
On the other hand, by \eqref{eq: Xi a'}, if two objects  in $\sH_{a'}$ have the same residue, then they are in the same fiber over $\Fc A$.
Therefore, the $\sP$-action on $\sH$ over $A(\omega_{C'}(D'))$ gives rise to a $\wt{\sP}(=\sP\times_{A(\omega_{C'}(D'))}\Fc A)$-action on $\sH$ over $\Fc A$.
\begin{prop}
    \label{prop: H pstorsor}
    The stack $\sH$ is a $\wt{\sP}$-pseudo-torsor over $\Fc A$. 
    In particular, $\sH$ is smooth over $\Fc A$.
\end{prop}
\begin{proof}
The smoothness of $\sH$ follows from the smoothness of $\sP$ over $A(\omega_{C'}(D'))$.

    Let $(a',s)\in \Fc A=A(\omega_{C'}(D'))\times_{\Gamma(D',\Fc_{D'})}\Gamma(D,\Fc_D)$ be a point such that the fiber $\sH_{a',s}$ contains an object $(E,\nabla)$.
    It suffices to show that the morphism $act(-,(E,\nabla)):\sP_{a'}\times \sH_{a',s}\to \sH_{a',s}$ is an equivalence.
    \Cref{lemma: faithful action} shows that the morphism $act(-,(E,\nabla))$ is fully faithful.
    It remains to show that it is also essentially surjective.
    
    Let $(F,\nabla_F)$ be another object in $\sH_{a',s}$.
    By \Cref{lemma: unique res}, we have that $\mathrm{res}_D(\nabla_E)=\mathrm{res}_D(\nabla_F)\in \Gamma(D,\mathrm{Lie}(J_{a^p})|_D)$.
    We also have that $\Psi(\nabla_E)=\Psi(\nabla_F)$ by \Cref{sssec: description of H}. 
    Let us form the dual $J_{a^p}$-torsor $E^{\vee}$ of $E$, which comes with the dual log-connection $\nabla_{E^{\vee}}$ by \Cref{lemma: dual is dual}.
Consider the log-$J_{a^p}$ connection $(X,\nabla_X):=(E^{\vee}\times^{J_{a^p}} F, \nabla_{E^{\vee}}\otimes\nabla_F)$.
Then \eqref{eq: res twi prod}, \eqref{eq: pcuvr twi prod}, \eqref{eq: res dual}, and \eqref{eq: pcurv dual} together entail that $\Psi(\nabla_X)=0$ and that $\mathrm{res}_D(\nabla_X)=0$.
Cartier descent \cite[Thm. 5.1]{katz1970} entails that $(X,\nabla_X)$  is isomorphic to 
$(Fr^*M,\nabla^{can})$
for a unique up to an isomorphism $J_{a'}$-torsor $M$ on $C'$.
We then have that $act(M, (E,\nabla))\cong (F,\nabla_F)$.
\end{proof}

Since $\sH$ is smooth over $\Fc A$, its image is open in $\Fc A$.
\begin{defn}\label{def: def him}
    We denote by $\Fc A_{im}$ the open subvariety of $\Fc A$ given by the image of $\pi^{\Fc A}:\sH\to \Fc A$.
\end{defn}
In general, $\sH$ does not surject onto $\Fc A$, see \Cref{cor: not surj}.

\subsubsection{Regular log-connections}$\;$

An $L$-twisted Higgs bundle $(P,\phi)$ on $C$ is regular if it defines a section $C\to (\Fg^{reg}/G)_L$.
The open substack of regular $L$-twisted Higgs bundles is denoted by $\sM_{Dol}^{reg}(C,L)$.
\begin{defn}
    A log-connection $(E,\nabla)$ is said to be regular if the $p$-curvature $(E,\Psi(\nabla))$ is a regular $\omega_C^p(pD)$-twisted Higgs bundle.
    The open substack of regular log-connections on $C$ is denoted by $\sM_{dR}^{reg}(C,D)$.
\end{defn}

If $\deg(D)$ is even, then a choice of a square root $\omega_C(D)^{1/2}$ of $\omega_C(D)$ gives rise to a Kostant section $\kappa: A(\omega_{C'}(D'))\to \sM_{Dol}^{reg}(C',\omega_{C'}(D'))$, see e.g. \cite[\S2.3]{chzh17}.

The no pole analog of the following \Cref{lemma: constr m} is observed in \cite[ top of p.1720,]{chen-zhu}.
The argument in \textit{loc.cit.} does not apply in the log-case. 
Below we provide a direct proof in the log-case, valid also in the no pole case.

\begin{prop}
\label{lemma: constr m}
Assume that one of the following two conditions are satisfied:
\begin{enumerate}
    \item[(i)] $G$ is a product of general linear groups.
    \item[(ii)] $\deg(D)$ is even.
\end{enumerate}
     Then there is a $\sP$-equivariant isomorphism of $\Fc A$-stacks:
    \begin{equation}
        \mathbf{m}: \sH\xrightarrow{\sim} \sM_{dR}^{reg}(C,D).
    \end{equation}
If (i) is satisfied, then the construction of $\mathbf{m}$ is canonical.    
If (ii) is satisfied, then the construction of $\mathbf{m}$ depends on  the choice of a square root $\omega_C(D)^{1/2}$ of $\omega_C(D)$.
\end{prop}

\begin{proof} 
It suffices to work over a point $(a',s)\in \Fc A(R)= (A(\omega_{C'}(D'))\times_{\Gamma(D',\Fc_{D'})}\Gamma(D,\Fc_D))(R)$, with $R$ an affine scheme.
Let $(E,\nabla_E)$ be an object in the fiber $\sH_{a'}$,  see \Cref{sssec: description of H}.

Suppose that $\deg(D)$ is even.
    Let $(X, \phi)$ be the Higgs bundle in the fiber $\sM_{Dol}(C',D')_{a'}$ defined by the Kostant section $\kappa(a')$.
    Recall that we have the morphisms of $C'_R$-group schemes $c_{X,\phi}: J_{a'}\to Aut(X,\phi)$.
    Applying Frobenius pull back and forgetting $\phi$, we obtain the morphism of $C_R$-group schemes $Fr^*c_{X,\phi}: J_{a^p}\to Aut(Fr^*X)$.
    Let $\nabla^{can}_{Fr^*X}$ be the canonical connection on $Fr^*X$.
    The morphism $\mathbf{m}$ sends the pair $(E,\nabla_E)$ to the log-connection 
    \begin{equation}
    \label{eq: ep and nabla}
        (E_X,\nabla_{E_X}):=(E\times^{J_{a^p}}_{Fr^*c_{X,\phi}} Fr^*X,\nabla_E\otimes \nabla^{can}_{Fr^*X}).
    \end{equation}
    
We now check that $(E_X,\nabla_{E_X})$ is regular and sits over $(a',s)$, which 
can be done \'etale locally over $C_R$.
Fix an \'etale cover $U$ of $C_R$ and trivializations of $E$, $X$ and $\omega_C(D)$ over $U$.
They induce an isomorphism $E_X\cong G_U$ and an identification of $Fr^*\phi$ with a section $U\to \Fg_U$.
The $p$-curvature $\Psi(\nabla_E)$ is identified with a section $U\to \mathrm{Lie}(J_{a^p})_U$.
The morphism $Fr^*c_{X,\phi}$ is identified with a morphism $c_{Fr^*\phi}: (J_{a^p})_U\to G_U$.
Then the product formula \eqref{eq: p cj} entails that $\Psi(\nabla_{E_X})=dc_{Fr^*\phi}(\Psi(\nabla_E)): U\to \Fg_U$.
Since $\Psi(\nabla_E)=\tau(a^p)$, \cite[Lem. 2.2]{chen-zhu} entails that $\Psi(\nabla_{E_X})=Fr^*\phi$, which is regular with Hitchin image $a^p$.
Similarly, the product formula \eqref{eq: res cj} entails that $\mathrm{res}_D(\nabla_{E_X})=dc_{Fr^*\phi}(\mathrm{res}_D(\nabla_E))\in \Gamma(D_R, \Fg_{D_R})$.
Since $\phi$ is defined by the Kostant section, the construction of the morphism $\pi^{\Fc A}$, especially \eqref{eq: Xi a'}, entails that $\chi_D(\mathrm{res}_D(\nabla_{E_X}))=s\in \Gamma(D_R,\Fc_{D_R})$.

    The $\sP$-equivariance follows from the construction of $\mathbf{m}$.

    To show that $\mathbf{m}$ is an isomorphism, we construct an inverse.
    Let $(F,\nabla_F)$ be an object in $\sM_{dR}^{reg}(C,D)_{a'}$.
    Then $(F,\Psi(\nabla_F))$ is a regular $\omega_C^p(pD)$-twisted Higgs bundle in $\sM_{Dol}^{reg}(\omega_C^p(pD))_{a^p}$.
    The Higgs bundle $(Fr^*X,Fr^*\phi)$ is the image of $a^p\in A(\omega_C^p(pD))$ under the Kostant section $A(\omega_C^p(pD))\to \sM^{reg}_{Dol}(\omega_C^p(pD))$. 
    Since $\sM_{Dol}^{reg}(\omega_C^p(pD))$ is a trivial $\sP(\omega_C^p(pD))$-torsor, up to isomorphisms, there is a unique $J_{a^p}$-torsor $E$ such that the $\sP(\omega_C^p(pD))$-action on $\sM_{Dol}(\omega_C^p(pD))$ sends $(E, (Fr^*X,Fr^*\phi))$ to $(F,\Psi(\nabla_F))$.
    We are done if we can show that there is a unique log-connection $\nabla_E$ on $E$ such that 
    \begin{enumerate}
        \item[(i)] $\Psi(\nabla_E)=\tau_{\omega_{C}^p(pD)}(a^p)\in \Gamma(C_R,\mathrm{Lie}(J_{a^p})_{\omega_C^p(pD)})$.
        \item[(ii)] $\nabla_E\otimes \nabla^{can}_{Fr^*X}=\nabla_F$.
    \end{enumerate}
    Note that (ii) implies (i).
    Indeed, we can check the implication locally over $C_R$ and 
    we fix the trivializations of $E, X$ and   $\omega_{C}(D)$ over $U$ as above.
    The trivialization induces an identity $Fr^*\phi=\Psi(\nabla_F): U\to \Fg_U^{reg}$.
    Assume that (ii) holds. Then \eqref{eq: p cj} entails that $\Psi(\nabla_F)=dc_{\Psi(\nabla_F)}(\Psi(\nabla_E))$.
    \cite[Lem. 2.2]{chen-zhu} and the regularity of $\Psi(\nabla_F)$ entails that $\Psi(\nabla_E)=\tau(a^p)$ is the only solution.

    We are reduced to show that there is a unique $\nabla_E$ such that (ii) holds. We can work locally over $C_R$ and fix the trivializations of $E, X, \omega_C(D)$ over $U$ above.
    We have $F=Fr^*X=G_U$ and $\Psi(\nabla_F)=Fr^*\phi=\kappa(a^p): U\to \Fg^{reg}_U$.
    The log-connection $\nabla_F$ is determined by the connection 1-form, which is a section $g_F: U\to \Fg_U$.
    Since log-connections commute with their $p$-curvatures \cite[(5.2.1)]{katz1970}, we have that $g_F$ factors through $U\to \mathrm{Lie}(I_{\kappa(a^p)})\subset \Fg_U$.
    Any log-connection $\nabla_E$ is determined by a section $g_E: U\to \mathrm{Lie}(J_{a^p})_U=\mathrm{Lie}(I_{\kappa(a^p)})_U$, which gives rise to the connection 1-form.
    Therefore, we can choose $\nabla_E$ so that $g_E=g_F$. Applying \eqref{eq: p cj} again, we see that this is the unique choice so that (ii) is satisfied.

    When $G$ is a product of general linear groups, we can replace the Kostant section with the canonical section $s: A(\omega_{C'}(D'))\to \sM_{Dol}(C',D')$ defined by the structural sheaf of spectral curves (see \Cref{sssec: two embeddings}), and the  proof given above  works in this context.
\end{proof}

\subsection{The Log-$p$-NAHT  for general $G$}\;

\subsubsection{Twisted product of stacks}$\;$

Given a Picard stack $\sG$ over a site $\sS$, and two stacks $\sT_1, \sT_2$  which $\sG$ acts on,  we can form the twisted product $\sT_1\times^{\sG}\sT_2:=(\sT_1\times \sT_2)/\sG$, which is a $2$-stack, as explained in \cite[p.419, Compl\'ement]{ngo2006fibration}. 
By \cite[Lem. 4.7]{ngo2006fibration}, if $\sT_1$ is a $\sG$-pseudo-torsor, then $\sT_1\times^{\sG}\sT_2$ is $2$-equivalent to a usual $1$-stack. 
Therefore, by using the anti-diagonal $\wt{\sP}$-action on $\sH\times_{\Fc A}\wt{\sM}_{Dol}(C,D)$ induced by \Cref{prop: H pstorsor} and \eqref{eq: wt P act on wt mdol}, we obtain the $\Fc A$-stack $\sH\times^{\wt{\sP}}\wt{\sM}_{Dol}(C',D')$.
Since $\wt{\sP}$ is commutative, there is still a natural $\wt{\sP}$-action on $\sH\times^{\wt{\sP}}\wt{\sM}_{Dol}(C',D')$.

\subsubsection{}
The next theorem is a log-analog of \cite[Thm. 1.2]{chen-zhu} and it complements 
some of the results in  \cite[\S4]{shen2018tamely} when $G=GL_r$. 

\begin{thm}[{\bf Log-$p$-NAHT}]\label{thm: main thm general g}
    There is a canonical $\wt{\sP}$-equivariant morphism of $\Fc A$-stacks
    \begin{equation}
        \label{eq: nah mor}
            \mathbf{n}: \mathscr{H}\times^{\widetilde{\mathscr{P}}} \widetilde{\mathscr M}_{Dol}(C',D')\to \mathscr{M}_{dR}(C,D), 
        \end{equation}
which is an isomorphism over the open subvariety $\Fc A_{im}\subset \Fc A$ as in \Cref{def: def him}.
\end{thm}
\begin{proof}
\underline{Construction of the morphism $\mathbf{n}$.}
    Let $(a', s)$ be a point in $\Fc A$ such that the fiber $\sH_{a',s}$ contains an object $(E,\nabla_E)$, see the description of $(E,\nabla_E)$ in \Cref{sssec: description of H}.
An object in $\wt{\sM}_{Dol}(C',D')_{a',s}$ is a Higgs field $(F,\phi)$ in $\sM_{Dol}(C',D')_{a'}$.
There is a morphism of group schemes 
$c_{F,\phi}:J_{a'}\to Aut(F)$.
The morphism $\mathbf{n}_{a',s}$ then sends $(E,\nabla)$ and $(F,\phi)$
to  the twist $(E_F,\nabla_{E_F}):=(E\times^{J_{a^p}}_{Fr^*c_{F,\phi}} Fr^*F, \nabla_E\otimes \nabla^{can}_{Fr^*F})$.

We now verify that $(E_F,\nabla_{E_F})$ is still over $(a',s)$ in $\Fc A$.

By \eqref{eq: p cj}, we have that $\Psi(\nabla_{E_F})$ is given by $dc_{F,\phi}(\Psi(\nabla_E))\in \Gamma(C, \Fg_{\omega_{C}^p(pD)})$ locally over $C$.
By \cite[Lem. 2.2]{chen-zhu}, we have that locally over $C$, $\Psi(\nabla_{E_F})$ is identified with $\phi\in \Gamma(C,\Fg_{\omega_{C}^p(pD)})$.
Therefore, we have that $(E_F,\nabla_{E_F})$ sits over $a'$ in $ A(\omega_{C'}(D'))$.

By \eqref{eq: res cj}, we have that $\mathrm{res}_D(\nabla_{E_F})$ is given by $dc_{\phi_D}(\mathrm{res}_D(\nabla_E))\in \Gamma(D, \Fg_D)$, where $\phi_D\in \Gamma(D,\Fg_D)$ and $c_{\phi_D}: \Gamma(D, (J_{a^p})_D)\to \Gamma(D,G_D)$ is the morphism induced by the morphism $c: \chi^*J\to I\to G\times\Fg$ as in \cite[Lem. 2.1.1]{ngo-lemme-fondamental}.
Let $(P,\psi)$ be the Kostant section $\kappa(a')$ in $\sM_{Dol}(C',D')_{a'}$.
By construction, $s\in \Gamma(D,\Fc_D)$ is given by the image of $\mathrm{res}_D(\nabla_E)\in \Gamma(D,\mathrm{Lie}(J_{a^p}))$ under the composition $\Gamma(D, \mathrm{Lie}(J_{a^p})_D)\xrightarrow{dc_{\psi_D}} \Gamma(D,\Fg_D)\to \Gamma(D,\Fc_D)$.
Since both $\phi_D$ and $\psi_D$ in $\Gamma(D,\Fg_D)$ have the same image in $\Gamma(D,\Fc_D)$, \Cref{lemma: g equiv} entails that $\mathrm{res}_D(\nabla_{E_F})$ still has invariants of residue $s$.

\underline{Isomorphy.} 
Let $(M,\nabla_M)$ be an object in $\sM_{dR}(C,D)_{a',s}$.
There is a morphism of group schemes $c_{M,\Psi(\nabla_M)}:J_{a^p}\to Aut(M, \Psi(\nabla_M))$.
Let $(E^{\vee},\nabla_{E^{\vee}})$ be the dual log-$J_{a^p}$-connection.
Consider the $G$-bundle $\wt{F}:=E^{\vee}\times^{J_{a^p}}_{c_{M,\Psi(\nabla_M)}} M$.
It comes with a log-connection $\nabla_{\wt{F}}$ and an $\omega_{C}^{\otimes p}(pD)$-Higgs field $\phi_{\wt{F}}$.
Indeed, the log-connection $\nabla_{\wt{F}}$ is given by the tensor product connection $\nabla_{E^{\vee}}\otimes \nabla_{M}$, and 
the Higgs field $\phi_{\wt{F}}$ is induced by $\Psi(\nabla_M)$.
Furthermore, we have that $\phi_{\wt{F}}$ is horizontal with respect to $\nabla_{\wt{F}}$ by the local argument in the last paragraph of the proof of \cite[Thm. 3.12]{chen-zhu}. 
Then \eqref{eq: p cj}, \eqref{eq: res cj}, \eqref{eq: pcurv dual}, and \eqref{eq: res dual} together entail that both the $p$-curvature and the residue of $\nabla_{\wt{F}}$ vanish.
Therefore, it follows from \cite[\S A.7]{chen-zhu} that the pair $(\wt{F},\phi_{\wt{F}})$
Frobenius descends to an $\omega_{C'}(D')$-Higgs bundle $(F,\phi_F)$ on $C'$.
The assignment that sends $(M,\nabla_M)$ to  the pair $((E,\nabla_E), (F,\phi_F))$ defines an inverse of $\mathbf{n}$.
\end{proof}

\subsection{Relation to earlier works}\label{subs: comparison}

\subsubsection{Comparison with Log-NAHT over $\bC$}\label{rmk: parabolic}\;

An important difference between the Log-$p$-NAHT \Cref{thm: main thm general g} above and the Log-$\bC$-NAHT as in \cite{simpson1990harmonic} is that no parabolic structure is needed in the Log-$p$-NAHT.
We recall the  Log-$\bC$-NAHT in the $G=GL_r$ case below, and we thank Pengfei Huang and Hao Sun for
helpful discussions.

 The Main Theorem in \cite[p.755]{simpson1990harmonic} states that, given a smooth non-compact complex algebraic curve $C^o$ with smooth completion $C$, the category of tame harmonic bundles on $C^o$ is canonically equivalent to the category of polystable \textit{filtered} logarithmic Higgs bundles of degree zero on $C$, as well as the category of polystable \textit{filtered} logarithmic connections of degree zero on $C$.

By inspection of the table relating residues and jumps in \cite[p.720]{simpson1990harmonic}, we see that the Main Theorem in \cite[p.755]{simpson1990harmonic}
cannot yield an equivalence of categories between polystable log-Higgs bundles and polystable log-connections,
at least if the first step of this putative equivalence is the embedding of polystable log-Higgs bundles
as polystable filtered logarithmic Higgs bundles with trivial filtration and a single jump at the real number zero;
indeed, in this case, the corresponding log-connections would have residues with purely imaginary eigenvalues.

In comparison, in the Log-$p$-NAHT, any section of the set-theoretic function $\sH(k)\to \Fc A_{im}(k)$ provides us with a morphism of groupoids $ (\sM_{dR}(C,D)|_{\Fc A_{im}})(k)\to\sM_{Dol}(C',D')(k)$, by considering the inverse of $\mathbf{n}$ as in \eqref{eq: nah mor} and the projection morphism $\wt{\sM_{Dol}}\to\sM_{Dol}$.

\subsubsection{Relation to \cite{schepler2005logarithmic}}$\;$

Let $W_2(k)$ be the ring of second truncated Witt vectors for $k$.
In the case with no poles, the Ogus-Vologodsky theory entails that each $W_2(k)$-lift $\wt{C'}$ of $C'$ gives rise to an equivalence of categories between rank $r\leq p$ nilpotent Higgs bundles on $C'$ and rank $r\leq p$ connections on $C$ with nilpotent $p$-curvature \cite[Thm. 2.8, Cor. 2.9]{ov07}.

Recall the number $h(G)$ as in \Cref{defn: cox num}.
In \cite[\S3.5]{chen-zhu}, it is explained that each $\wt{C'}$ defines an object in the nilpotent fiber, i.e. the fiber over the origin of the Hitchin base $A(\omega_{C'}(D'))$, of the no pole analogue of $\sH$ for each group $G$, when $p\geq h(G)$, so that the $p$-NAHTs considered in \cite{ov07} and \cite{chen-zhu} are compatible. Furthermore, when $G=SL_2$ or $PGL_2$, all objects in the nilpotent fiber of the no pole analogue of $\sH$ arise this way \cite[Lem. 3.23 and the paragraph after]{chen-zhu}.

\begin{remark}
    Although not stated, the assumption that $p\geq h(G)$ is needed in \cite[\S3.5]{chen-zhu}, since 
    a principal $SL_2$ for $G$ is used in the third paragraph of \cite[p. 1723]{chen-zhu}. 
    A principal $SL_2$ exists if and only if $p\geq h(G)$.
    Indeed, the if implication is entailed by \cite[Prop. 2]{serre1996exemples}, and the only if implication is due to the following: any nilpotent element $x$ in the $p$-Lie algebra $\Fsl_2$ satisfies the identity $x^{[p]}=0$, while given a regular nilpotent element $y$ in $\Fg$, we have that $y^{[p^m]}=0$ if and only if $p^m\geq h(G)$ by combining \cite[Order Formula 0.4]{testerman1995a_1}, which deals with regular unipotent elements in $G$, and \cite[Thm. 35]{mcninch2003sub}, which shows that the Springer isomorphism (which exists since $p\nmid |W|$) between the unipotent variety of $G$ and the nilpotent variety of $\Fg$ is compatible with the $p$-power operations, see also \cite[Rmk. 1.2]{sobaje2015springer}.
\end{remark}

The Ogus-Vologodsky theory has been extended  to the logarithmic case by Schepler in \cite{schepler2005logarithmic}.
Schepler considers the log-scheme $C'_{log}$ defined by the divisor $D'$, and shows that each $W_2(k)$-lifting $\wt{C_{log}'}$ of $C'_{log}$ gives rise to an equivalence of category between rank $r\leq p$ nilpotent log-Higgs bundles on $C'$ with nilpotent residues and rank $r\leq p$ log-connections on $C$ with nilpotent $p$-curvatures and residues \cite[Cor. 4.11]{schepler2005logarithmic}, see also \cite[Thm. 5.1]{langer2024bogomolov}.
(Note that our paper deals with curves, while both \cite{ov07} and \cite{schepler2005logarithmic} work with general (log-)varieties.)
By arguing as \cite[\S3.5]{chen-zhu}, we have the following result which is used in the proof of \Cref{prop: pstorsor ho}.
\begin{lemma}\label{lem: compare schepler}
Assume that $p\geq h(G)$.
    Let $\Fo:= (o_{A(\omega_{C'}(D'))}, o_{\Gamma(D,\Fc_D)})\in \Fc A=A(\omega_{C'}(D'))\times_{\Gamma(D',\Fc_{D'})}\Gamma(D,\Fc_D)$ be the origin.
    Let $Lift(C',D')$ be the groupoid of $W_2(k)$-liftings of the log-scheme $C'_{log}$.
    Then there is a canonical morphism $Lift(C',D')\to \sH_{\Fo}$, which is an isomorphism when $G=PGL_2$ or $SL_2$.
\end{lemma}
\begin{proof}
    By \cite[Prop. 3.6 and (5) on p.14]{schepler2005logarithmic}, each $W_2(k)$-lift $\wt{C'_{log}}$ of $C'_{log}$ gives rise to a short exact sequence of log-connections on $C$:
    \begin{equation}\label{eq: ext of conns}
        0\to (Fr^*T_{C'}(-D'),\nabla^{can})\xrightarrow{\alpha} (E,\nabla_E)\xrightarrow{\beta} (\cO_C,\nabla^{can})\to 0,
    \end{equation}
    such that the $p$-curvature $\Psi(\nabla_E): E\to E\otimes\omega_C^p(pD)$ is given by the composition $E\xrightarrow{\beta}\cO_C\cong Fr^*T_{C'}(-D')\otimes \omega_C^p(pD)\xrightarrow{\alpha\otimes 1} E\otimes \omega_C^p(pD)$.
    Let $o_{PGL_2}\in \Fp\Fgl_2\git PGL_2$ be the origin.
    Then $J_{o_{PGL_2}}\cong \bG_a$ and $\tau(o_{PGL_2})=1\in \mathrm{Lie}(J_{o_{PGL_2}})$, see \cite[Lem. 3.21, 3.22]{chen-zhu} and \cite[Proof of Lem. 3.24]{herrero-zhang-semistable}.
    Let $o_{A(PGL_2)}'$ be the origin of the Hitchin base $A(\omega_{C'}(D'),PGL_2)$, and let $o_{A(PGL_2)}^p$ the origin of the Hitchin base $A(\omega_{C}^p(pD))$.
    It follows that $J_{o_{A(PGL_2)}'}\cong T_{C'}(-D')$, the latter viewed as an additive group scheme over $C'$, and that $X:=(\beta^{-1}(1),\nabla_E|_{\beta^{-1}(1)})$ is a $J_{o_{A(PGL_2)}^p}$-log-connection in $\sH_{o'_{A(PGL_2)}}$.
    Since the two canonical connections in \eqref{eq: ext of conns} have zero residues, the residue of $X$ is nilpotent, thus $X$ lies in $\sH_{\Fo}$ for $PGL_2$.
    
    The assignment $\wt{C'_{log}}\mapsto X$ then defines a functor $f: Lift(C',D')\to \sH_{\Fo}$ when $G=PGL_2$.
    We now show that $f$ is an isomorphism.
    We can stackify both groupoids $Lift(C',D')$ and $\sH_{\Fo'}$ over the fppf site of $C'$ and obtain two $C'$-stacks $\mathfrak{L} ift$ and $\mathfrak{H}_{\Fo'}$, i.e.,  for each $C'$-scheme $U$, $\mathfrak{L} ift(U)$ is the groupoid of $W_2(k)$-lifts of $(U,D'\times_{C'} U)$, and $\mathfrak{H}_{\Fo'}(U)$ is given by $Fr^*T_C'(-D')$-connections on $Fr^*U:=U\times_{C'}C$ whose $p$-curvature is $1\in H^0(Fr^*U,\cO_{Fr^*U})\cong H^0(Fr^*U,\mathrm{Lie}(J_{o^p_{A(PGL_2)}})\otimes \omega_C^p(pD))$ and whose residue is nilpotent.
    By deformation theory of log-schemes as in \cite[Prop. 3.14]{kato1989logarithmic}, the stack $\mathfrak{L}ift$ is a gerbe banded by $T_{C'}(-D')$.
    The proof of \Cref{prop: H pstorsor} shows that $\mathfrak{H}_{\Fo'}$ is a gerbe banded by $J_{o'_{A(PGL_2)}}\cong T_{C'}(-D')$.
    Furthermore, the morphism $f$ can be promoted to a morphism $\Ff: \mathfrak{L}ift\to \mathfrak{H}_{\Fo'}$ of $C'$.
    By the construction of $E$ in \cite{schepler2005logarithmic}, different lifts give rise to non-identical $E$'s, so that $\Ff$ induces injective morphisms on automorphism groups.
    Since an injective endomorphism of $\bG_a$ is an isomorphism, we have that $\Ff$ is bijective on automorphism groups, thus an isomorphism of gerbe by \cite[Lem. 4.6]{ov07}.
Since we can recover $f$ from $\Ff$ by taking $U=C'$, we see that $f$ is an isomorphism of groupoids.
    
    The argument for $G=SL_2$ is similar because  $J_{o_{SL_2}}\cong \bG_a\times\mu_2$, so that we just need to replace \eqref{eq: ext of conns} by 
    \begin{equation}\label{eq: ext of conns sl2}
        0\to (Fr^*T_{C'}(-D')\times\mu_2,\nabla^{can})\xrightarrow{\alpha} (E\times \mu_2,\nabla_{E\times\mu_2})\xrightarrow{\beta} (\cO_C,\nabla^{can})\to 0.
    \end{equation}

    For a general $G$, 
    since $p\geq h(G)$, \cite[Prop. 2]{serre1996exemples} gives us a principal $SL_2\to G$, which induces a morphism $\varphi: J_{o_{SL_2}}\to J_{o_G}$. Therefore, a log-$J_{o^p_{A(SL_2)}}$-connection induces a log-$J_{o^p_{A(G)}}$-connection.
    By the construction of $\tau: \Fc\to \mathrm{Lie}(J)$, we see that $d\varphi: \mathrm{Lie}(J_{o_{SL_2}})\to \mathrm{Lie}(J_{o_G})$ sends $\tau(o_{SL_2})$ to $\tau(o_{G})$.
    Therefore, $\varphi$ induces a morphism $\sH_{\Fo,SL_2}\to \sH_{\Fo, G}$.
\end{proof}

\subsubsection{Relation to \cite{shen2018tamely}} $\;$

When $G=GL_{r< p}$, the base $\Fc A$  does not appear in \cite{shen2018tamely}, and the morphism $\mathbf{n}$ of \Cref{thm: main thm general g} is essentially described fiberwise over $\Gamma(D,\Fc_{D})$ in \cite[Thm. 4.3]{shen2018tamely}.
More precisely, in \cite{shen2018tamely}, Shiyu Shen works over some particular points in $\Fc A$.
Such points form an open $\Fc A_{Shen}\subseteq \Fc A$.
\cite[Thm. 4.3.(1)]{shen2018tamely} shows that $\Fc A_{Shen}\subset \Fc A_{im}$ as in \Cref{def: def him}.
Moreover, with notation as in \Cref{defn: base cA}, for every  $s\in \Gamma(D,\Fc_{D})$, the restriction of $\mathbf{n}$ to $pr_2^{-1}(s)\cap \Fc A_{Shen}$ is constructed in \cite[Thm. 4.3.(2)]{shen2018tamely}, and is shown to be an isomorphism.

Next, we show that in the case of $GL_{r<p}$ we have that   $\Fc A_{Shen}= \Fc A_{im}$. Assume that $D$ is a point.
In the $GL_r$ case, a point in $\Fc$ is given by an unordered tuple $\{b_1,...,b_r\}$ of $r$ elements in $k$, thought of as the eigenvalues of matrices.
Let $x:=(a, \{b_1,...,b_r\})\in \Fc A$.
If $b_1^p-b_1=b_2^p-b_2$ but $b_1\neq b_2$, then $x$ is not in $\Fc A_{Shen}$.

\begin{lemma}\label{lemma: compare cashen}
    When $G=GL_{r<p}$, we have $\Fc A_{Shen}=\Fc A_{im}$.
\end{lemma}
\begin{proof}
    That $\Fc A_{Shen}\subset \Fc A_{im}$ is \cite[Thm. 4.3.(1)]{shen2018tamely}.
    For $b_i\in k$ and  $n_i,i\in\bN_{>0}$, let $\mathbb{J}_{b_i,n_i}$ be the Jordan block with size $n_i$ and eigenvalue $b_i$.
    If the $b_i$ are pairwise distinct, then the direct sum matrix $R:=\bigoplus_i \mathbb{J}_{b_i,n_i}$ is regular. But if $b_1^p-b_1=b_2^p-b_2$, then the matrix $R^p-R$ is not regular.
    It follows that a log-connection of $\sM_{dR}(C,D)$ over $\Fc A\setminus\Fc A_{Shen}$ cannot have regular $p$-curvature, i.e. cannot lie in $\sM_{dR}^{reg}(C,D)$.
    By \Cref{lemma: constr m}, the $\Fc A$-stack $\sM_{dR}^{reg}(C,D)$ is isomorphic to $\sH$, so the lemma follows.
\end{proof}

\subsubsection{Relation to \cite{li2024tame}} $\;$

In \cite[Thm. 5.19]{li2024tame}, Mao Li and Hao Sun show that there is a canonical isomorphism of $A(\omega_{C'}(D'))$-stacks:
\begin{equation}
    LS: \sH\times^{\sA}\sX\xrightarrow{\sim} \sM_{dR}(C,D),
\end{equation}
where $\sA$ is the Picard stack over $A(\omega_{C'}(D'))$ whose fiber over $a'$ is given by the groupoid of $J_{a^p}$-log-connections with vanishing $p$-curvature, and $\sX$ is the moduli stack of triples $(E,\nabla,\Psi)$, where $E$ is a $G$-bundle, $\nabla$ is a log-connection with vanishing $p$-curvature, and $\Psi$ is a horizontal section of $ad(E)\otimes \omega_C^p(pD)$.
In \cite[Prop. 5.24]{li2024tame}, it is shown that $\sX=\coprod_{\tau}\sM_{Dol}(C', G_{\tau}')$, where $\tau$ runs through rational semisimple conjugacy classes in $\Gamma(D,\Fg_D)$,   $G_{\tau}'$ is the parahoric group scheme over $C'$ associated to $\tau$ as in \cite[\S4.3]{li2024tame}, and $\sM_{Dol}(C',G_{\tau}')$ is the moduli stack of $\omega_{C'}(D')$-twisted $G_{\tau}'$-Higgs bundles.

\subsubsection{Implications of \cite{li2024tame}  for $\Fc A_{im}$} $\;$

The base $\Fc A$ as in \Cref{defn: base cA} does not appear in \cite{li2024tame}.
Instead, a base denoted by $B_{\sL'}^{ext}$ is used in \textit{loc. cit.}
We recall its definition and show that it is identified with $\Fc A_{im}$ as in \Cref{def: def him}.

Recall the morphism $\tau_{\omega_{C'}(D')}: A(\omega_{C'}(D'))\to A(J,\omega_{C'}(D'))$ as in \Cref{defn: def H}, where $A(J,\omega_{C'}(D'))$ is a scheme over $A(\omega_{C'}(D'))$ whose fiber over $a'$ is $\Gamma(C, \mathrm{Lie}(J_{a'})\otimes\omega_{C'}(D'))$.
Restricting to $D'$ and using the canonical isomorphism $\omega_{C'}(D')|_{D'}=\cO_{D'}$, we get the $\Gamma(D',\Fc_{D'})$-scheme $\Gamma(D', \mathrm{Lie}(J)|_{D'})$ and the restriction morphism $res_{D'}: A(J,\omega_{C'}(D'))\to \Gamma(D', \mathrm{Lie}(J)|_{D'})$.
The morphism $AS_{J_{a'}}: \mathrm{Lie}(J_{a'})\to \mathrm{Lie}(J_{a'})$ given by $x\mapsto x^{[p]}-x$ give rises to a morphism $AS_J: \Gamma(D', \mathrm{Lie}(J)|_{D'})\to \Gamma(D', \mathrm{Lie}(J)|_{D'})$.
Composing with the inverse of the isomorphism $Fr^*_D: \Gamma(D', \mathrm{Lie}(J_{a'})|_{D'})\xrightarrow{\sim} \Gamma(D, \mathrm{Lie}(J_{a^p})|_D)$, we obtain the morphism $\Fas_J: \Gamma(D, \mathrm{Lie}(J^p)|_D)\to \Gamma(D', \mathrm{Lie}(J)|_{D'})$.
The base $B_{\sL'}^{ext}$ is defined so that the following diagram is Cartesian:
\begin{equation}
    \xymatrix{
    B_{\sL'}^{ext} \ar[rr]^-{res}\ar[d]_-{p_B} && \Gamma(D, \mathrm{Lie}(J^p)|_{D})\ar[d]^-{\Fas_J}\\
    A(\omega_{C'}(D'))\ar[rr]_-{res_{D'}\circ \tau_{\omega_{C'}(D')}} && \Gamma(D', \mathrm{Lie}(J)|_{D'}).
    }
\end{equation}
Given an object $(E,\nabla)$ in $\sH_{a'}$, its residue lies in $\Gamma(D, \mathrm{Lie}(J_{a^p})|_D)$, so that the structural morphism $\pi_2: \sH\to A(\omega_{C'}(D'))$ factors through $\sH\xrightarrow{h_B} B_{\sL'}^{ext}\xrightarrow{p_B} A(\omega_{C'}(D'))$.

\begin{lemma}\label{lemma: b injects to ca}
    There is a canonical injective $A(\omega_{C'}(D'))$-morphism $i_B: B_{\sL'}^{ext}\hookrightarrow \Fc A$.
    Furthermore, the morphism $\pi^{\Fc A}:\sH\to \Fc A$ factors through $i_B$.
\end{lemma}
\begin{proof}
    We can assume that $D$ is a single point.
    Let $a'\in A(\omega_{C'}(D'))$ with $a^p=Fr^*a'\in A(\omega_C^p(pD))$. We have $a^p|_{D}\in \Fc_{D}$ and $\tau(a^p|_{D})\in \mathrm{Lie}(J_{a^p})|_{D}$.
        Let $\kappa(a^p|_{D})\in \Fg$ be the Kostant section.
    A point in $p_B^{-1}(a')$ is given by a section of $ \mathrm{Lie}(J_{a^p})|_{D}$ which is in the Artin-Shreier preimage of $\tau(a^p|_{D})$.
    We have the morphisms $ \mathrm{Lie}(J_{a^p}|_D)\xrightarrow{dc_{\kappa(a^p|_{D})}} \Fg\xrightarrow{\chi}\Fc$.
    We define $i_B$ by applying $\chi\circ dc_{\kappa(a^p|_{D})}$ to $p_B^{-1}(a')$.
    The factorization claim follows from the construction of $\pi^{\Fc A}$, especially \eqref{eq: Xi a'}.
\end{proof}

\begin{cor}\label{cor: not surj}
    The injection $i_B$ in \Cref{lemma: b injects to ca} gives rise to an isomorphism $i_B: B_{\sL'}^{ext}\xrightarrow{\sim}\Fc A_{im}$.
    Furthermore, the morphisms $\sH\to \Fc A_{im}\to A(\omega_{C'}(D'))$ are both smooth and surjective.
    In particular, $\Fc A_{im}\neq \Fc A$ whenever $AS_{\Fc}:\Fc \to \Fc$ is not \'etale.
\end{cor}
\begin{proof}
    It is proved in \cite[Lem. 5.18 and the sentence above]{li2024tame} that $\sH$ surjects onto $B_{\sL'}^{ext}$ and that $B_{\sL'}^{ext}\to A(\omega_{C'}(D'))$ is smooth. Combining with \Cref{lemma: b injects to ca} and the construction of $\Fc A\to A(\omega_{C'}(D'))$ as in \Cref{defn: base cA}, the corollary follows.
\end{proof}

Consider the restrictions $\wt{\sP}|_{\Fc A_{im}}$ and $\wt{\sM_{Dol}}(C',D')|_{\Fc A_{im}}$, and view them as stacks over $A(\omega_{C'}(D'))$.

\begin{question}
    Are there $A(\omega_{C'}(D'))$-isomorphisms of the form $\sA\xrightarrow{\sim} \wt{\sP}|_{\Fc A_{im}}$ and $\sX\xrightarrow{\sim} \wt{\sM_{Dol}}(C',D')|_{\Fc A_{im}}$?
    A weaker question: is $\wt{\sP}|_{\Fc A_{im}}$ a Picard stack over $A(\omega_{C'}(D'))$?
\end{question}

\subsection{Semistability}\label{sec: lognah}

\subsubsection{Neutral components of Picard stacks} $\;$

Let $L$ be a line bundle over $C$.
By \cite[Lem. 3.1]{herrero-zhang-semistable}, the Picard stack $\sP(L)$ over $A(L)$ has an open substack $\sP^o(L)$ such that for every geometric point $a$ of $A(L)$, the fiber $\sP^o(L)_a$ is the connected component of $\sP(L)_a$ that contains the netural $J_{a}$-torsor.

Consider the Picard stacks $\sP^o(\omega_C^p(pD))\hookrightarrow \sP(\omega_C^p(pD))$ over $A(\omega_C^p(pD))$.
Let $Fr^*: A(\omega_{C'}(D'))\hookrightarrow A(\omega_C^p(pD))$ be the morphism (which is a closed immersion by \cite[Lem. 5.6]{herrero-zhang-mero}) of Hitchin bases induced by $Fr_C:C\to C'$.
\begin{defn}
    We define the $A(\omega_{C'}(D'))$-stacks 
    \begin{equation}
        \sP_p:=\sP(\omega_C^p(pD))\times_{A(\omega_C^p(pD)), Fr^*} A(\omega_{C'}(D')),
    \end{equation}
    \begin{equation}
        \sP_p^o:=\sP^o(\omega_C^p(pD))\times_{A(\omega_C^p(pD)), Fr^*} A(\omega_{C'}(D')).
    \end{equation}
\end{defn}

\subsubsection{Very good regular connections} $\;$
By the construction of $\sH$ as in \Cref{sssec: constr of h}, there is a morphism of $A(\omega_{C'}(D'))$-stacks $\sH\xrightarrow{forget} \sP_p$ given by forgetting the connections. 
\begin{defn}
    We define the $A(\omega_{C'}(D'))$-stack $\sH^o$ to be the pullback
    \begin{equation}
        \xymatrix{
        \sH^o \ar@{^{(}->}[r] \ar[d] & \sH \ar[d]^-{forget}\\
        \sP^o_p \ar@{^{(}->}[r] & \sP_p.
        }
    \end{equation}
\end{defn}

\begin{notn}
    In the no pole case, the stack analogous to $\sH^o$ is the stack of very good splittings as in \cite[\S3.3]{dCGZ} when $G=GL_r$, and the stack of very good $G$-splittings as in \cite[\S3.2]{herrero-zhang-semistable} for general $G$.
\end{notn}

\subsubsection{Non-emptyness of  $\sH^o$} $\;$

Let $\sP^o\hookrightarrow\sP$ be the open sub-Picard stack over $A(\omega_{C'}(D'))$ given by  neutral components.
Recall the notation $\wt{(-)}$ as in \Cref{sssec: wt not}.
We have the open sub-Picard stack $\wt{\sP^o}\hookrightarrow \wt{\sP}$.
We have shown in \Cref{prop: H pstorsor} that $\sH$ is a $\wt{\sP}$-pseudo-torsor over $\Fc A$.
\begin{prop}\label{prop: pstorsor ho}
The $\Fc A$-stack $\sH^o$ is a smooth $\wt{\sP^o}$-pseudo-torsor.
It is non-empty if $(p,G)$ is good enough as in \Cref{defn: good enough}.
\end{prop}
\begin{proof}
    The proof for the smooth-pseudo-torsorness is the same as in the no pole case in \cite[Prop. 3.22]{herrero-zhang-semistable}, see also \cite[Lem. 3.9]{dCGZ} for the $GL_r$-case.
    The idea is the following: 
    since $\sH$ is a $\wt{\sP}$-pseudo-torsor, any two objects $X,Y$ in $\sH^o$ in the same fiber over $(a',s)\in \Fc A$ are related by the action of an essentially unique object $E$ in $\sP_{a'}$.
    To prove the first assertion, we 
    are left with showing  that $E$ lies in $\sP_{a'}^o$.
    In the $GL_n$-case, $E$ is identified with a line bundle on a spectral curve $S_{a'}$ over $C'$, and $E\cdot X=Y$ entails that the pullback of $E$ to the spectral curve $S_{a^p}=S_{a'}\times_{C', Fr_{C}}C$ has degree 0 on each irreducible component. But the pullback changes the multidegrees of line bundles by multiplication by $p$, so that $E$ has to also lie in the neutral component of $Pic(S_{a'})$, which is $\sP_{a'}^o$.
    In the general $G$-case it is more subtle because $\sP$ may have torsion components and we need to prove the \textbf{claim} that the Frobenius pullback does not send torsion components in $\sP(\omega_{C'}(D'))$ to the neutral component in $\sP(\omega_C^p(D^p))$. 
    The proof of \cite[Prop. 3.22]{herrero-zhang-semistable} shows that, under our assumption that $p\nmid |W|$, the only component sent by $Fr_C^*$ to the component $\sP^o(\omega_C^p)$ is $\sP^o(\omega_{C'})$.
    The argument in the proof of \cite[Prop. 3.22]{herrero-zhang-semistable} carries verbatim to the log-case, establishing the \textbf{claim} above.
    
    To prove the second assertion, it remains to show that $\sH^o$ is nonempty.

    First assume that $p\geq h(G)$.
In the proof of \Cref{lem: compare schepler}, we have shown that $\sH$ is nonempty for $SL_2$ and that there is a morphism $\varphi: J_{o_{\Fc_{SL_2}}}\to J_{o_{\Fc_{G}}}$, inducing a morphism $\varphi':\sH_{\Fo, SL_2}\to \sH_{\Fo, G}$, where $\Fo\in \Fc A$ denotes the point giving nilpotent residues for connections and $p$-curvatures.
Indeed, more is true. 
Note that the morphism $\beta$ in \eqref{eq: ext of conns sl2} induces a deformation from the trivial $J_{o_{A(SL_2)}^p}$-torsor to $\beta^{-1}(1)$.
Therefore, the proof of \Cref{lem: compare schepler} shows that each $W_2(k)$-lift of $C'$ induces an object in $\sH_{o_{A(SL_2)}^p}^o$.
Since the morphism $Bun_C(J_{o_{A(SL_2)}^p})\to Bun_C(J_{o_{A(G)}^p})$ induced by $\varphi$ sends the trivial torsor to the trivial torsor, we have that $\varphi'$ restricts to a morphism $\sH_{\Fo,SL_2}^o\to \sH_{\Fo,G}^o$, thus the target is nonempty.
    
   We now assume that $G$ is a product of general linear groups. Since the formation of $\sH$ is compatible with Cartesian products of reductive groups, we may assume that $G=GL_r$.
    We first construct an object in $\sH$, i.e. a pair $(E,\nabla)$, where $E$ is a rank $r$ vector bundle and $\nabla$ is a log-connection whose $p$-curvature is regular, see the description of $\sH$ in the $GL_r$-case as in \Cref{lem: h and drreg glr}.
    Indeed, below we construct such a pair $(E,\nabla)$ which furthermore has nilpotent residue and $p$-curvature.
The proof of \cite[Prop. 4.6]{shen2018tamely} provides a way to construct regular log-connections with fixed $p$-curvatures and residues.
    By steps 1 and 3 in the proof of \cite[Prop. 4.6]{shen2018tamely}
     (these steps work also when $p\leq r$), we are reduced to the case where $C=\mathrm{Spec}(k[[t]])$ is a disk and $D$ is the closed point of $C$.
    Upon choosing a trivialization of $\omega_C^p(pD)$, the nilpotent spectral curve $S_{{\Fo}^p}$ is identified with $\mathrm{Spec}(k[[t]][y]/(y^r))$, where the coordinate $y$, when restricted to $D$, is given by the canonical section $(t\partial_t)^{\otimes p}$ of $\Big(T_C(-D)|_D\Big)^{\otimes p}$.
    We would like to construct a log-connection $\nabla$ on the vector bundle $\pi^p_*\cO_{S_{{\Fo}^p}}$ whose residue $\mathrm{res}_D(\nabla)$ is nilpotent, and whose $p$-curvature $\Psi(\nabla)$ saitsifies the identity  $\Psi(\nabla)(t\partial_t)=y$.

    Let $e, ye, ...,y^{r-1}e$ be a basis of $\pi^p_*\cO_{S_{ \Fo^p}}$.
    Let $A$ be the matrix of $\nabla(x\partial_x)$ wrt this basis.
    The requirement that $\mathrm{res}_D(\nabla)$   is nilpotent is satisfied if $A$ is a lower triangular matrix with 0 on the diagonal.
    The requirement that $\Psi(\nabla)(t\partial_t)=y$ is the same as requiring that the matrix $A^p-A$ has $1$ on the lower subdiagonal and $0$ everywhere else. 

We can take $A$ to be the matrix $(a_{i,j})$ where $a_{i,j}=-1$ if $j=i-1$ or $j=i-p$, and $a_{i,j}=0$ otherwise. 
   That is, $A$ is the matrix with $-1$'s on the 1st and $(p-1)$-th subdiagonals, and 0 everywhere else, e.g. 
   \[A=\begin{bmatrix}
       0 & 0 & 0 & 0 & 0 & 0\\
       -1 & 0 & 0 & 0 & 0 & 0\\
       0 & -1 & 0 & 0 & 0 & 0\\
       0 & 0 & -1 & 0 & 0 & 0\\
       -1 & 0 & 0 & -1 & 0 & 0\\
       0 & -1 & 0 & 0 & -1 & 0
   \end{bmatrix}, \; p=5, \; r=6.\]

   We now modify $(E,\nabla)$ to an element $(F,\nabla_F)$ in $\sH^o_{\Fo'}$.
Let $L_E$ be the $\ov{\sD}(-D)$-module on the nilpotent spectral curve $S_{o^p}$ corresponding to $(E,\nabla)$
(cf. \eqref{diag: four corners on sp curves}).
The restriction $L_E|_C$ corresponds to a quotient log-connection $(E_1,\nabla_1)$ of $(E,\nabla)$.
Since $E_1=L_E|_C$ is of rank 1, we have that $\nabla_1$ has vanishing $p$-curvature and residue.
By Cartier descent, $E_1=Fr^*M$ for some line bundle $M$ on $C'$.
By \cite[Lemma 5.11]{LiuQ}, we can extend the dual line bundle $M^{\vee}$ to a line bundle $\wt{M^{\vee}}$ on $S_{\Fo'}$, thus defining an object in the fiber $\sP_{o'}$.
Let $(F,\nabla_F)$ be the image of $(E,\nabla_E)$ under the action of $\wt{M^{\vee}}$.
Let $L_F$ be the corresponding $\ov{\sD}(-D)$-module on $S_{o^p}$.
Since $L_F|_C\cong L_E|_C\otimes Fr^*M^{\vee}\cong \cO_C$ (c.f. \eqref{eq: other def act P}), we have that $\deg_{S_{o^p}}(L_F)=r\deg_C(L_F|_C)=0$ by \cite[\S9, Prop. 5]{blr-neron}.
Therefore, we have that $(F,\nabla_F)\in \sH^o$. 
\end{proof}

By \Cref{prop: pstorsor ho}, the morphism $\sH^o\to \Fc A$ has open image.
\begin{defn}
    Let $\Fc A^o_{im}\subset \Fc A$ be the open subvariety given by the image of $\sH^o$.
\end{defn}

\subsubsection{Semistable Non Abelian Hodge Theorem} $\;$

The following theorem is the log-analogue of \cite[Thm. 4.10]{herrero-zhang-semistable} in the no pole case.

\begin{thm}\label{thm: main sst}
Let $d\in \pi_1(G)$.
   The $\Fc A$-morphism  \eqref{eq: nah mor} restricts to a $\Fc A$-morphism
    \begin{equation}
    \label{eq: nss stack g}
        \mathbf{n}^{ss}: \sH^o\times^{\wt{\sP}^{o}} \wt{\sM}_{Dol}^{ss}(C',D',d) \to \sM_{dR}^{ss}(C, D,pd),
    \end{equation}
    which is an isomorphism over the open subscheme $\Fc A_{im}^o$ of  ${\Fc A}$ defined by the image of $\sH^o$. 

    Furthermore, if $G=GL_r$ or if $G$ satisfies (LH) as in \Cref{def: low height}, then every term in \eqref{eq: nss stack g} has an adequate moduli space and we have a $\Fc A$-morphism of schemes:
\begin{equation}
    \label{eq: very good morphism schemes}
    n^{ss}: H^o\times^{\widetilde{P}^{o}} \widetilde{M}_{Dol}^{ss}(C',D',d)\to M_{dR}^{ss}(C,D,pd), 
\end{equation}
which is an isomorphism over $\Fc A_{im}^o$.    
\end{thm}
\begin{proof}
    The proof of \eqref{eq: nss stack g} is the same as in the no-pole case in \cite[Thm. 4.10]{herrero-zhang-semistable}, which is based on the theory of $\Theta$-semistability.
    
    The idea is the following: for each object in $x$ in $\sH^o_{a',s}$, the isomorphism $\mathbf{n}$ induces an isomorphism of $k$-stacks $\mathbf{n}_x: \sM_{Dol}(C',D')_{a'}\xrightarrow{\sim}\sM_{dR}(C,D)_{a',s}$, and it suffices to show that $\mathbf{n}_x$ preserves semistability and changes degree $d$ to $pd$.
    There are line bundles $L_{Dol}$ and $L_{dR}$ on $\sM_{Dol}(C',D')_{a'}$ and $\sM_{dR}(C',D')_{a'}$ which are the pullbacks of the determinant bundles $L_{C'}$ and $L_C$ on $Bun_G(C')$ and $Bun_G(C)$ as in \cite[\S1.F.a]{heinloth-hilbertmumford}.
    Let $\Theta$ be the quotient stack $\bA^1/\bG_m$.
    Given a line bundle $L$ on $\Theta$, the central fiber $L|_{B\bG_m}$ gives rise to an integer $wt(L)$, which is the weight of the $\bG_m$-action.
    By \cite[\S1.F.c]{heinloth-hilbertmumford}, an object $e$ in $\sM_{Dol}$ is semistable, if for any morphism $f:\Theta\to \sM_{Dol}$ with $f(1)=e$, we have $wt(f^*L_{Dol})\leq 0$. Similarly for objects in $\sM_{dR}$.
    For each morphism $f: \Theta\to \sM_{Dol}(C',D')_{a'}$, we have a composition $f_x:=forget\circ \mathbf{n}_x\circ f: \Theta\to \sM_{dR}(C,D)_{a',s}\to Bun_G(C)$.
     The morphism $\mathbf{n}_x$ preserves semistability if $wt(f_x^*L_{C})=p\cdot wt(f^*L_{Dol})$. 
    
    Since $x$ is in $\sH_{a',s}^{o}$, the morphism $forget\circ \mathbf{n}_x: \sM_{Dol}(C',D')\to Bun_G(C)$ is algebraically equivalent to the morphism $(E,\phi)\mapsto J_{a^p}\times^{J_{a^p}}_{Fr^*c_{E,\phi}}Fr^*E=Fr^*E$.
    
    Let $g:= forget\circ f:\Theta\to Bun_G(C')$.
    We have the morphism $Fr^*\circ g: \Theta\to Bun_G(C)$.
    It then follows that $wt(f_x^*L_C)=wt((Fr^*\circ g)^*L_{C'})=p\cdot wt(g^*L_{C'})=p\cdot wt(f^*L_{Dol})$, as desired, so that $\mathbf{n}_x$ preserves semistability.
    The morphism $\sM_{Dol}\to Bun_G(C)$ defined by $(E,\phi)\mapsto Fr^*E$ changes degree from $d$ to $pd$, so that we also get the degree change claim. 
    
When $G=GL_r$, the existence of the adequate moduli spaces is shown in \cite[Thm. 1.1]{langer2014semistable} (for $\wt{M}_{Dol}^{ss}$ and $M_{dR}^{ss})$, the same proof as in \cite[p.716, \S5]{chaudouard-laumon} (for $\wt{P}^{o,\ov{0}}$), and the same proof as in \cite[Thm. 4.13]{dCGZ} (for $H^o$).
When $G$ is a reductive group satisfying (LH), the existence of the adequate moduli spaces is shown in \cite[Thm. 2.26]{herrero-zhang-mero} (for $\wt{M}_{Dol}^{ss}$ and $M_{dR}^{ss}$), and the same proof as in \cite[Prop. 4.8]{herrero-zhang-semistable} (for $\wt{P}^{o,\ov{0}}$ and $H^o$).
\end{proof}

\section{The vector bundle case: $G=GL_r$}

We unravel the definition of $\sH$, $\sH^o$, the Non Abelian Hodge morphism $\mathbf{n}$, and its semistable analogue $\mathbf{n}^{ss}$ in the $G=GL_r$, and refine the picture according to degrees of various objects.

\subsection{Log-$p$-de Rham-BNR}\label{subs: logdrbnr}\;

We first list the relevant sheaves of non-commutative algebras (cf.  \eqref{eq: noncommalg} and
\eqref{diag: four corners on sp curves}) and then we prove 
the Log-$p$-de Rham-BNR \Cref{lemma: log dr bnr}.

    We have the following Cartesian diagram of $k$-schemes:
    \begin{equation}
    \xymatrix{
    T^{1,*}C(pD)\ar[r]^-{W} \ar[d]_-{\pi^p} & T^*C'(D') \ar[d]^-{\pi'} \\
    C\ar[r]_-{Fr_C} & C',
    }       
    \end{equation}
    where $T^*C'(D')$ is the total space of the line bundle $\omega_{C'}(D')$.
    Note that $T^{1,*}C(pD)$ is the total  space of the line bundle $\omega_C^p(pD)$.
Let $\mathscr{D}_C(-D)$ be the universal enveloping algebra of the Lie algebroid $T_X(-D)$.
Let $Z_{\mathscr{D}_C(-D)}(\mathcal{O}_C)$ be the centralizer of $\mathcal{O}_C$ inside $\mathscr{D}_C(-D)$.
We have a canonical identification:
\begin{equation}\label{eq: centralizer of o}
\pi^p_{*}\mathcal{O}_{T^{1,*}C(pD)}= Z_{\mathscr{D}_C(-D)}(\mathcal{O}_C).
\end{equation}
The identification is the log-version of \cite[Lem. 2.1.1]{BMR08}, and the proof therein, after natural 
and straightforward modifications, also works in the log-case.
Therefore, there is a left $\mathcal{O}_{T^{1,*}C(pD)}$ algebra $\overline{\mathscr{D}}(-D)$ such that $\pi^p_* \overline{\mathscr{D}}(-D)=\mathscr{D}_C(-D)$. 

By \cite[top of p.24]{hablicsek2020hodge}, there is a canonical identification 
\begin{equation}\pi'_*\mathcal{O}_{T^*C'(D')}=Fr_* Z_{\mathscr{D}_C(-D)}(\mathscr{D}_C(-D)).
\end{equation}
Because the morphisms $\pi'$ and $Fr_C$ are affine, the above identification gives rise to a central $\mathcal{O}_{T^*C'(D')}$-algebra $\mathscr{D}(-D)$ such that $\pi'_* \mathscr{D}(-D)=Fr_*\mathscr{D}_C(-D)$.
This is the log-version of the 
celebrated Azumaya algebra in \cite[Thm. 2.2.3]{BMR08}.
The algebra $\sD(-D)$ is studied in detail in \cite{hablicsek2020hodge}, where it is shown that $\sD(-D)$ is a parabolic Azumaya algebra over the log-space $(T^*C'(D'),(\pi')^*D')$ \cite[Lem. 3.2]{hablicsek2020hodge}.
Note that we also have the identity of $\mathcal{O}_{T^*C'(D')}$-algebras
$W_*\overline{\sD}(-D)=\mathscr{D}(-D)$.
This identification gives $\sD(-D)$ a left $W_*\cO_{T^{1,*}C(D^p)}$-module structure.
In summary, we draw the following diagram, marking which algebra is where:
\begin{equation}\label{eq: noncommalg}
    \xymatrix{
    \overline{\mathscr{D}}(-D)& T^{1,*}C(pD)\ar[r]^-{W} \ar[d]_-{\pi^p} & T^*C'(D') \ar[d]^-{\pi'} &\mathscr{D}(-D)=W_*\overline{\mathscr{D}}(-D) \\
    \mathscr{D}_C(-D)=\pi^p_*\overline{\mathscr{D}}(-D) &C\ar[r]_-{Fr_C} & C' & \pi'_* \mathscr{D}(-D)=Fr_*\mathscr{D}_C(-D).
    }
    \end{equation}
 Note that the algebras on the left are not central, the one on the top right corner is central and, for the  bottom-right corner, $\cO_{C'}$ is contained in the center.

Given a $k$-point  $a'$ of $A(\omega_{C'}(D'))$, let $S_{a'}$ be the corresponding spectral curve
for $\omega_{C'}(D')$ in $T^*C'(D')$, and let $S_{a^p}$ be the
degree $r$ spectral curve
in the total space of $\omega_C^p (pD)$ corresponding to $Fr_C^* a'$. 
The corresponding Cartesian diagram is:
\begin{equation}
    \label{diag: four corners on sp curves}
    \xymatrix{
    \overline{\mathscr{D}}(-D)& S_{a^p}\ar[r]^-{W} \ar[d]_-{\pi^p} & S_{a'} \ar[d]^-{\pi'} &\mathscr{D}(-D)=W_*\overline{\mathscr{D}}(-D) \\
    \mathscr{D}_C(-D)=\pi^p_*\overline{\mathscr{D}}(-D) &C\ar[r]_-{Fr_C} & C' & \pi'_*\sD(-D)=Fr_*\mathscr{D}_C(-D).
    }
\end{equation}

The identities in the diagram 
\eqref{diag: four corners on sp curves}
yield the following 
\begin{lemma}[\bf Log-$p$-de Rham-BNR]
    \label{lemma: log dr bnr}
    The fiber category $h^{-1}_{dR}(a')\subset \mathscr{M}_{dR}(C,D)$ is identified with the groupoid of torsion free rank $1$  coherent sheaves on $S_{a^p}$ together with a $\overline{\mathscr{D}}(-D)$-module structure.
    It is also identified with the groupoid of torsion free rank $p$  coherent sheaves on $S_{a'}$ together with a $\sD(-D)$-module structure.
\end{lemma}

Note that here we need to use the definition of the rank of a torsion free  coherent  sheaf given in \cite[Def. 1.1, 1.2]{schaub-courbes-spectrales}, see also \cite[\S2.4]{de2017support-sln}.
The proof of \Cref{lemma: log dr bnr} is the same as the one  in the no-pole case;
 see \cite[Prop. 3.15]{groechenig-moduli-flat-connections} and \cite[\S4.1]{dCGZ}.

 \begin{notn}
     Given an object $(E,\nabla)$ in $\sM_{dR}(C,D)_{a'}$, we call the corresponding torsion free rank 1 sheaf $L_E$ on $S_{a^p}$ as the de Rham-BNR sheaf corresponding to $(E,\nabla)$.
 \end{notn}

 \begin{remark}\label{rmk: bnr sheaf}
     By the same reasoning as in \cite[Rmk. 2.1.2]{BMR08} in the no pole case, the sheaf $L_E$ is also the BNR sheaf of the $\omega_C^p(pD)$-twisted Higgs bundle $(E,\Psi(\nabla))$.
 \end{remark}
The lemma below follows from \Cref{lemma: constr m} and \Cref{rmk: bnr sheaf}.
\begin{lemma}\label{lem: h and drreg glr}
    The $A(\omega_{C'}(D'))$-stack $\sH_{a'}$ is canonically identified with the substack $\sM_{dR}^{reg}(C,D)$ of $\sM_{dR}(C,D)$ given by objects $(E,\nabla)$ whose de Rham-BNR sheaf is invertible.
\end{lemma}

\subsubsection{}\label{sssec: action of p on mdr}
 In view of \Cref{lemma: log dr bnr}, we can
recast the action \eqref{eq: p action mdr a} of $\mathscr P$ on $\mathscr M_{dR}(C,D)$
as follows: let $L_E$ be a $\ov{\mathscr D}(-D)$ torsion free rank $1$ coherent sheaf on $S_{a^p}$; let $M$ be a line bundle on $S_{a'}$ and define the action by the assignment
\begin{equation}\label{eq: other def act P}
( M, L_E) \mapsto L_E \otimes W^* M.
\end{equation}
 The rhs in \eqref{eq: other def act P}  is indeed a
$\ov{\mathscr D}(-D)$-module on $S_{a^p}$
because, by the projection formula, its $W_*$ push-forward is a ${\mathscr D}(-D)$-module on $S_{a'}$
(cf. \eqref{diag: four corners on sp curves}).

\subsection{$\sH$ dominates $\Fc A$}\label{sub: heth0}$\;$

Recall that $\Fc A = \coprod_{\delta \in \bF_p} \Fc A^{\delta}$ (cf. \eqref{diag: ca and wtmdol}).

\begin{remark}\label{rmk:hd conn}
Let $\sH = \coprod_d \sH (d)$ be the disjoint union according to the degree of connections.
Then $\sH (d) \to \Fc A^{-\ov{d}}$. Each $\sH(d)$ is a union of connected components of $\sH$.  Since $\wt{\sP}(0)^{-\ov{d}}$  is connected (it is so over the dense open 
in $\Fc A^{-\ov{d}}$ of smooth spectral curves and it is smooth over the whole $\Fc A^{-\ov{d}}$), if $\sH(d)$
is non-empty, then it is a $\wt{\sP}(0)^{-\ov{d}}$-torsor by \Cref{prop: H pstorsor}, thus connected.
\end{remark}

By the forthcoming \Cref{prop: irreducible}, $\mathscr M^{ss}_{dR}(C,D,d)$ is integral, hence connected.
If
$\mathscr M_{dR}(C,D,d)$ is connected, then so is $\sH(d)$. We ignore if $\mathscr M_{dR}(C,D,d)$ is connected. In the case $p>r$, we have the following

\begin{lemma}\label{lm: hd}
 Let $p> r$.
The morphism $\sH\to A(\omega_{C'}(D'))$ is surjective.
  For every $d \in \bZ$ the stack $\sH (d)$ is non-empty  and connected  and it dominates $\Fc A^{-\ov{d}}$. In particular, $\sH$ dominates $\Fc A$.
\end{lemma}
\begin{proof}
The surjectivity assertion
follows from 
\cite[Prop. 4.6]{shen2018tamely},  whose proof uses the hypothesis $p>r$. To prove the second assertion, we first recall what is accomplished in \cite[Prop. 4.6]{shen2018tamely}. For ease of notation, we discuss
the case of one pole $D=\{x\}$; the general case is treated similarly. Fix a closed point $a'\in A(\omega_{C'}(D'))$; choose an $r$-tuple $\alpha=\{a_i\}$ of scalars  in $k$ subject to the constraint $\sum_i a_i=0$; choose a set-theoretic section $\sigma$ of the the Artin-Schreier function $AS: k \to k$, $t \mapsto t^p-t$;
take the $r$-tuple $\beta:=\{b_i:=\sigma (a_i)\}$.
By means of a clever local calculation, S. Shen constructs a regular log-connection $(E,\nabla)$ on $C$ with log-poles at $D$ only, with  characteristic polynomial of its $p$-curvature $h_{dR}((E,\nabla))= a'$, with eigenvalues of the residue of the $p$-curvature $\alpha$, with eigenvalues of the residue of the connection $\beta$, and with dR-BNR sheaf a line bundle $L$ on $S_{a^p}$. 
Note that by the Fuchs' relations, the residual class in $\bF_p$  of $\deg (E)$ equals $-\sum_i b_i$, and that, by Riemann-Roch, we have the following relation on the degrees:
\begin{equation}\label{eq: rel degs}
\deg (L) = p\frac{1}{2}r(r-1)(2g-2 + \deg(D)) + \deg (E).
\end{equation}

If $\alpha$ is chosen so that at least two of the scalars $a_i$ are distinct, then, by choosing the section $\sigma$, we can ensure that the value $\sum_i b_i$ equals any pre-fixed value $\delta \in \bF_p$. In summary, by choosing $\alpha$ generically, we can find elements in $\sH_{a'}$ that map to any pre-assigned connected component $\Fc A^{\delta}.$ By the smoothness of $\sH$ over $\Fc A$, we conclude that
every connected component $\Fc A^{\delta}$ is dominated by some connected component 
 $\sH (d(\delta))$, with $d(\delta) \equiv - \delta \mod (p)$. This proves the dominance assertion.

 In view of \Cref{rmk:hd conn},
it remains to show that $\sH (d)$ is non-empty for every $d \in \bZ$.
 Let $\delta := - \ov{d}$.
Then $d= d(\delta)+ mp$, for some $m \in \bZ$.
Let $a'$ be general; in particular $S_{a'}$ is smooth. Let $M$ be a degree $m$ line bundle on $S_a'$ and let $M_{a^p}$ be its pull-back to $S_{a^p}$.
If $(E,\nabla) \in \sH_{a'}$ is a regular log-conection as above of degree $d(\delta)$,  with associated dR-BNR sheaf $L$, then $L\otimes M_{a^p}$  is the dR-BNR sheaf of a regular connection in $\sH_{a'}$ of degree $d= d(\delta)+mp$ (cf.
\eqref{eq: rel degs}).
\end{proof}

\subsection{The refined theorems in the case $G=GL_r$.}

\begin{thm}\label{thm: main thm glr}
   Assume that $G=GL_r$. When taking into account the degrees, the morphism $\mathbf{n}$ in \Cref{thm: main thm general g} restricts in the following way:
For each $d,d'\in\bZ$, we have a morphism of $\Fc A^{-\ov{d}}$-stacks
\begin{equation}
\label{eq: nah mor -d}
    \mathbf{n}^{-\ov{d}}: \sH(d)\times^{\wt{\sP} (0)^{\ov{0}}}\wt{\sM}_{Dol}(C',D',d')^{-\ov{d}}\to \sM_{dR}\Big(C,D,pd'+d-p\frac{r(1-r)}{2}(2g-2+\mathrm{deg}(D)\Big),
\end{equation}
which restricts to an isomorphism over the open image $\Fc A_{im}^{-\ov{d}}$ of $\sH(d)$ in $\Fc A^{-\ov{d}}$.
\end{thm}
\begin{proof}
    Let $(F,\nabla_F)$ be an object in $\sH(d)$,  viewed as a regular log-connection on $C$, over $a'$ in $A(\omega_{C'}(D'))$.
Let $(E,\phi)$ be an object in $\sM_{Dol}(C',D',d')$ over $a'$.
Let $(Q,\nabla_Q)$ be the resulting object in $\sM_{dR}(C,D)$ over $a'$.
We need to show that $Q$ is of degree $pd'+d-p\frac{r(1-r)}{2}(2g-2+\mathrm{deg}(D))$. 

Let $L_F$ be the  dR-BNR line bundle on the spectral curve $S_{a^p}$  corresponding to the regular connection $(F, \nabla_F)$; in particular,
$L_F$ pushes-forward to $F$ via the projection $\pi^p: S_{a^p}\to C$.
    Let $L_E$ be the BNR sheaf of $(E,\phi)$ on the spectral curve $S_{a'}$.
    By  construction, we have an identification of   vector bundles $Q=\pi^p_*(L_F\otimes W^*L_E)$.
    In summary, we draw below which sheaf  is on which curve:
    \begin{equation}
    \xymatrix{
    L_F & S_{a^p}\ar[r]^-{W} \ar[d]_-{\pi^p} & S_{a'} \ar[d]^-{\pi'} & L_E \\
    Q=\pi^p_*(L_F\otimes W^*L_E), \;\; F=\pi^p_* L_F &C\ar[r]_-{Fr_C} & C' & E=\pi'_*L_E.
    }
\end{equation}

Let $\chi(-)$ denote the Euler characteristics.  We have the following chain of equality of numbers:
\begin{equation}
\label{eq: many chi calc}
    \chi(Q) \stackrel{1}{=}\chi(L_F\otimes W^*L_E)\stackrel{2}{=}\chi(W^*L_E) +\deg(L_F) \stackrel{3}{=}\chi(\pi^p_* W^*L_E)+\deg(L_F)\stackrel{4}{=}\chi(Fr_C^*\pi'_*L_E)+\deg(L_F)\stackrel{5}{=}\chi(Fr^*E)+\deg(L_F),
\end{equation}
where:
Equalities 1 and 5 are by definition;
    Equality 2: Riemann-Roch;
     Equality 3: $\pi^p$ is affine, thus has no effect on $\chi$;
   Equality 4: Flat base change.

Therefore, we have that $\deg(Q)=p\, \deg(E)+\deg(L_F)=pd'+\deg(L_F)$.
Finally, since $\chi(F)=\chi(L_F)$, Riemann-Roch gives us that $\deg(L_F)=\chi(F)-\chi(\cO_{S_{a^p}})=d-p\frac{r(1-r)}{2}(2g-2+\deg(D))$, and we are done.
\end{proof}

\subsubsection{The stack of very good regular log-connections}\label{subs: reg conn}\;

If $a'$ is a $k$-point, then $\mathscr{H}^o_{a'}$ is the groupoid of rank $r$ log-connections on $C$
with poles at $D$ and with 
log-dR-BNR sheaf a line bundle with class in $Pic^o(S_{a^p})$ (cf. \Cref{lemma: log dr bnr}).  
The same argument as in \cite[Thm. 4.11]{dCGZ} shows that $\sH^o$ is an open substack in the stable part of the de Rham moduli space $\sM_{dR}^{stable}(C,D)$.
A Riemann-Roch calculation,
 relating the degree of a log-connection to the degree of its dR-BNR sheaf, shows that 
    $\mathscr{H}^o$ is an open substack of
\begin{equation}
\label{eq: open in}
\mathscr{M}_{dR}^{stable}\Big(C,D,  d^o:=p\frac{r(1-r)}{2}(2g-2+\mathrm{deg}(D))\Big),
\end{equation}
 so that $\sH^o$
is also an open substack of $\sH(d^o)$. However, $\sH$
is not contained in the stable locus $\sM_{dR}^{stable}(C,D)$. 

The open substack $\sH^o$ of $\sH$ is contained in $\sH(d^o)$ (cf. \eqref{eq: open in}).
We have that $\mathscr P^o \subset 
 \mathscr P (0)$ is the open substack
over $A(\omega_{C'}(D'))$ of line bundles on the spectral curves $S_{a'}$ over $C'$ that are of degree zero on every geometric irreducible component
of $S_{a'}$. Let $\widetilde{\mathscr P}^o\subset \widetilde{\mathscr P} (0)$ be the 
analogous open substack after the base change $\Fc A \to A(\omega_{C'}(D')).$
Let  $\widetilde{\mathscr P}^{o,\ov{0}}\subseteq \widetilde{\mathscr P}^o$ be the 
open substack sitting over  $\Fc A^{\ov 0}$.
\Cref{prop: pstorsor ho} entails that
     $\sH^o$ is a $\wt{\sP}^{o,\ov{0}}$-torsor over a dense open subset of $\Fc A^{\ov 0}$. In particular, $\sH^o$ is integral.
     
We can remember that each term in \Cref{thm: main sst} lies over $\Fc A^{\ov{0}}$ and obtain:

\begin{thm}[\bf Semistable Log-NAHT]
\label{thm: sst dr}
    The $\Fc A$-morphism  \eqref{eq: nah mor} restricts to a $\Fc A^{\ov 0}$-morphism
    \begin{equation}
    \label{eq: nss stack}
        \mathbf{n}^{ss}: \sH^o\times^{\wt{\sP}^{o,\ov{0}}} \wt{\sM}_{Dol}^{ss}(C',D',d)^{\ov{0}} \to \sM_{dR}^{ss}(C, D,pd),
    \end{equation}
    which is an isomorphism over the open and dense subscheme $\Fc A_{im}^{0}$ of  ${\Fc A}^{0}$ defined by the image of $\sH^o$. 
    Furthermore, every term in \eqref{eq: nss stack} has an adequate moduli space and we have a $\Fc A^{0}$-morphism of schemes:
\begin{equation}
    \label{eq: very good morphism schemes}
    N^{ss}: H^o\times^{\widetilde{P}^{o,\ov{0}}} \widetilde{M}_{Dol}^{ss}(C',D',d)^{\ov{0}}\to M_{dR}^{ss}(C,D,pd), 
\end{equation}
which is an isomorphism over $\Fc A_{im}^{0}$.
\end{thm}

\section{Consequences}\label{sect conseq}

In this section, we record some consequences of \Cref{thm: sst dr} about the geometry of the de Rham-Hitchin morphism.
Throughout this section, we assume $G=GL_r$.

\subsection{Irreducibility, surjectivity, connected fibers}
\label{subs: geom conseq}\;

\subsubsection{The Hodge-Hitchin morphism} $\;$

In some of the proofs below, we use the Hodge-Hitchin morphism $h_{Hod}:\sM_{Hod}(C,D)\to A(\omega_{C'}(D'))\times \bA^1$.
The moduli stack $\sM_{Hod}(C,D)$ parametrize vector bundles with $t$-log-connections, where $t\in k$.
The morphism $h_{Hod}$ interpolates the de Rham-Hitchin morphism (when $t=1$) and a morphism that is universally homeomorphic to the Hitchin morphism (when $t=0$). We refer the reader to \cite[\S2]{deC-H-2022geometry} for more details.

Recall the de Rham-Hitchin morphism $h_{dR}$ and the de Rham-Hitchin-residue morphism $h_{dR}^{\Fc A}$  in \eqref{diag: hcdr bis}.
The following two propositions are used in the proof of \Cref{prop: surj and conn fib}.
In the case $r$ and $d$ are coprime, the proposition below is a special case of \cite[Thm. 1.1]{deC-H-2022geometry}.
\begin{prop}[\bf Irreducibility of de Rham moduli stack/space]
\label{prop: irreducible}
    The semistable stack $\sM_{dR}^{ss}(C,D,d)$ and its moduli space $M_{dR}^{ss}(C,D,d)$ are integral and normal.
    The same is true if we replace ``dR" with ``Dol."
\end{prop}
\begin{proof}
The ``Dol" case is already known
by  \cite[Prop. 2.9.(2)]{maulik-shen-chi}. We prove the ``dR" case; the same proof works in the ``Dol" case.

By \cite[Prop. 5.6]{deC-H-2022geometry}, we have that $\sM_{dR}^{ss}(C,D,d)$  is smooth and non-empty.
By \cite[Prop. 5.4.1]{alper_adequate}, the adequate moduli space $M_{dR}^{ss}(C,D,d)$ is normal.
    It suffices to prove that $\sM_{dR}^{ss}(C,D,d)$ is irreducible.
By \cite[Prop. 4.13]{laumon2018champs}, a normal connected Noetherian algebraic stack is integral. 
We are thus reduced to showing that $\sM_{dR}^{ss}$ is connected.

By contradiction, assume that $\sM_{dR}^{ss}$ has two nonempty connected components $X$ and $Y$.
View $\sM_{dR}^{ss}$ as a closed substack of the $\bA^1$-stack $\sM^{ss}_{Hod}$.
Consider the $\bG_m$-orbits $\bG_m\cdot X$ and $\bG_m\cdot Y$ of $X$ and $Y$ in $\sM_{Hod}^{ss}$. 
By taking their closure $\ov{\bG_m\cdot X}$ and $\ov{\bG_m\cdot Y}$ in $\sM_{Hod}^{ss}$, we get two irreducible components of $\sM_{Hod}^{ss}$, thus making $\sM_{Hod}^{ss}$ reducible.
Recall that $\sM_{Hod}^{ss}$ is also smooth \cite[Prop. 5.6]{deC-H-2022geometry} 
and note that it is connected:  indeed, $\sM_{Dol}^{ss}$ is connected \cite[Prop. 2.9.(2)]{maulik-shen-chi}
and we can flow objects in $\sM_{Hod}^{ss}$ into $\sM_{Dol}^{ss}$ using the $\bG_m$-action and  semistable reduction \cite[Thm. 5.1]{langer2014semistable}.
We thus obtain a contradiction with \cite[Prop. 4.13]{laumon2018champs}.
\end{proof}

\begin{prop}[\bf Surjectivity of $h_{dR}$ and of $h_{dR}^{\Fc A}$]\label{prop: surj of drhr}
    The  de Rham-Hitchin morphism $h_{dR}:M_{dR}^{ss}(C,D,d)\to A(\omega_{C'}(D'))$ and the  de Rham-Hitchin-residue morphism $h_{dR}^{\Fc A}: M_{dR}^{ss}(C,D,d)\to \Fc A^{- \ov{d}}$ are surjective.
    The same is true if we replace ``dR" with ``Dol."
\end{prop}
\begin{proof}
The ``Dol" case is well-known, due to the BNR correspondence and the properness of the Hitchin morphism, and it is stated here for completeness.
We prove the ``dR" case.
  The morphism $h_{dR}: M_{dR}^{ss}(C,D,d)\to A(\omega_{C'}(D'))$ is proper by \cite[Thm. 5.9]{langer2021moduli}.
    Therefore, the morphism $h_{dR}^{\Fc A}:M_{dR}^{ss}(C,D,d)\to \Fc A^{-\ov{d}}$,
    through which $h_{dR}$ factors, is also proper.
    Since the projection morphism  $\Fc A^{\delta} \to A(\omega_{C'}(D')) $ is finite surjective,
    to reach both surjectivity 
    conclusions, it suffices to show that $h_{dR}: M_{dR}^{ss}(C,D,d)\to A(\omega_{C'}(D'))$ is dominant.

It suffices to show that the restriction of $h_{dR}$ on the stable locus $M_{dR}^s(C,D,d)\to A(\omega_{C'}(D'))$ is dominant.
Since the stable locus $M_{Dol}^s(C,D,d)$, containing the moduli of stable vector bundles of degree $d$, is nonempty,
we have that the stable locus $M_{Hod}^s(C,D,d)$, being open, is necessarily non-empty.
By virtue of the $\mathbb G_m$-action, the de Rham stable locus $M_{dR}^s(C,D,d)$ is also non-empty.
By the same argument as in the proof of \cite[Lem. 5.3]{deC-H-2022geometry}, the stable moduli space $M_{Hod}^s(C,D,d)$ is the $\bG_m$-rigidification of the stack $\sM_{Hod}^s(C,D,d)$, which is smooth over $\bA^1$ by \cite[Prop. 5.6]{deC-H-2022geometry}.
Therefore, the morphism $M_{Hod}^s(C,D,d)\to \bA^1$ is also smooth.
It follows that the stable moduli spaces $M_{dR}^s(C,D,d)$ and $M_{Dol}^s(C,D,d)$,  both being irreducible by \Cref{prop: irreducible} and \cite[Prop. 2.9.(2)]{maulik-shen-chi},  have the same dimension.

By \cite[Cor. 8.2]{chaudouard-laumon}, all fibers of $h_{Dol}: M_{Dol}^{s}(C,D,d)\to A(\omega_{C}(D))$ have the same dimension $d_h$.
Therefore, all fibers of the 0-Hodge-Hitchin $h_{Hod,0}: M_{Dol}^{s}(C,D,d)\to A(\omega_{C'}(D'))$ also have dimension $d_h$, because  $h_{Hod,0}$   differs from $h_{Dol}$ by a universal homeomorphism of the targets, see \cite[\S 7]{deC-H-2022geometry}.

Assume $h_{dR}:M_{dR}^s(C,D,d)\to A(\omega_{C'}(D'))$ is not dominant.
Then there must be a fiber $F_{dR}$ of $h_{dR}$ whose dimension is larger than $d_h$.
Now the 0-limit of $F_{dR}$ under the natural $\bG_m$-action on $M_{Hod}$ lies in the nilpotent fiber $h_{Hod,0}^{-1}({\Fo}')$, which has dimension $d_h$.
We thus obtain a contradiction with the upper-semicontinuity \cite[\href{https://stacks.math.columbia.edu/tag/0D4I}{Tag 0D4I}]{stacks-project} of the fiber dimension of the proper morphism $h_{Hod}:M_{Hod}^{s}(C,D,d)\to A(\omega_{C'}(D'))\times\bA^1$.
\end{proof}

\subsubsection{Connected fibers:
degree $pd$} $\;$

\label{subs: geom conseq bis}\;

\begin{prop}[\bf Degree $pd$: Stein factorization of $h_{dR}(dp)$]\label{prop: surj and conn fib}
The composition of morphisms $M_{dR}^{ss}(C,D,pd) \to cA^{\ov{0}} \to A(\omega_{C'}(D'))$ is the Stein factorization of $h_{dR}:M_{dR}^{ss}(C,D,pd)  \to A(\omega_{C'}(D'))$. In particular, the de Rham-Hitchin-residue morphism $h_{dR}^{\Fc A}: M_{dR}^{ss}(C,D,pd)\to \Fc A^{\ov 0}$ is surjective with geometrically connected fibers.
\end{prop}

\begin{proof}
The surjectivity of $h^{\Fc A}_{dR}$ follows from \Cref{prop: surj of drhr}.
    Since the base field $k$ is algebraically closed, by \cite[\href{https://stacks.math.columbia.edu/tag/0387}{Tag 0387}, \href{https://stacks.math.columbia.edu/tag/055H}{Tag 055H}]{stacks-project}, to show that $h_{dR}^{\Fc A}$ has geometrically connected fibers,
    it suffices to show that $h_{dR}^{\Fc A}$ has connected fibers.
    Consider the following diagram of $k$-schemes:
    \begin{equation}
    \xymatrix{
        B\ar[r]^-{f} \ar@/^3pc/[rrd]^-{x} \ar[dr]^-{y} & Z\ar[dr] \ar[d]^-{g}  &\\
        M_{dR}^{ss}(C,D,pd) \ar[r]_-{h_{dR}^{\Fc A}} \ar[u]^-{r} 
        \ar@/_2pc/[rr]^-{h_{dR}}& \Fc A^{\ov 0}\ar[r]_-{pr_1} & A(\omega_{C'}(D')),
        }
    \end{equation}
    where 
    \begin{enumerate}
        \item the morphisms $M_{dR}^{ss}(C,D,pd)\xrightarrow{r} B\xrightarrow{x} A(\omega_{C'}(D'))$ is the Stein factorization \cite[I. (9.1.21.1)]{ega1-springer} of $pr_1\circ h_{dR}^{\Fc A}=h_{dR}$; 
        \item since $pr_1: \Fc A^{\ov 0}\to A(\omega_{C'}(D'))$ is affine, the universal property of the Stein factorization \cite[I. (9.1.21.4)]{ega1-springer} induces a morphism $y: B\to \Fc A^{\ov 0}$;
        \item since $x: B\to A(\omega_{C'}(D'))$ is 
        finite \cite[\href{https://stacks.math.columbia.edu/tag/03H0}{Tag 03H0}]{stacks-project},   $h_{dR}^{\Fc A}$ is 
        surjective  and $r$ is surjective,  we have 
        that $y$ is finite surjective. 
        By  \Cref{prop: irreducible},  $M_{dR}^{ss}(C,D,pd)$ is integral, so that  $B$ is an integral variety. Therefore, we can apply \cite[Lemma 4.4.2]{deCHaiLi} and factor $y$ as $B\xrightarrow{f} Z\xrightarrow{g} \Fc A^{\ov 0}$,
        where $f$ is radicial, $g$ is separable and generically \'etale over $\Fc A^{\ov 0}$, and both $f$ and $g$ are finite surjective.
    \end{enumerate}
    
    By \Cref{thm: sst dr}, the general fiber of $h_{dR}^{\Fc A}$ is connected, thus $g:Z\to \Fc A^{\ov 0}$ has general connected fibers.
    Therefore, the morphism $g$ is finite birational, thus an isomorphism \cite[\href{https://stacks.math.columbia.edu/tag/0AB1}{Tag 0AB1}]{stacks-project}.
     Since $f$ is  radicial, finite, and surjective, it is a homeomorphism \cite[\href{https://stacks.math.columbia.edu/tag/01S4}{Tag 01S4}]{stacks-project}
    and \cite[\href{https://stacks.math.columbia.edu/tag/04DF}{Tag 04DF}]{stacks-project}, so that   $r$ has connected fibers.
    Therefore, the composition $h_{dR}^{\Fc A}=g\circ f\circ r$ has connected fibers.

    To prove the remaining Stein factorization assertion, since $g$ is an isomorphism,  it suffices to show that $f$ is an isomorphism. By \Cref{thm: sst dr}, the morphism $h^{\Fc A}_{dR}$ is smooth over a dense open subset of $\Fc A^{\ov{0}}$ because the same is true for the morphism $h^{\Fc A}_{Dol}$ in view of the BNR correspondence. It follows that the finite radicial  surjective morphism $f$ is birational, hence an isomorphism by \cite[\href{https://stacks.math.columbia.edu/tag/0AB1}{Tag 0AB1}]{stacks-project}.
\end{proof}

\subsection{Fiber dimensions and flatness: degree $pd$}
\label{subs: fltns}\;

 We thank Roberto Fringuelli
for explaining to us the proof of the following lemma.

\begin{lemma}\label{lm: dim fib hitch}
Every irreducible component of every fiber of the Hitchin morphism $h_{Dol}: M_{Dol}^{ss}(C,D,d)\to A(\omega_{C}(D))$ has the same dimension
\begin{equation}
    d_h(r):=r(g-1)+\frac{r(r-1)}{2}\mathrm{deg}(\omega_C(D))+1.
\end{equation}
\end{lemma}
\begin{proof}
    If we replace $M_{Dol}^{ss}(C,D,d)$ with its stable part $M_{Dol}^s(C,D,d)$, then this Lemma is \cite[Cor. 8.2]{chaudouard-laumon}.
     To prove the lemma, it suffices to prove that no Hitchin fiber contains an irreducible component made of strictly semistable Higgs bundles.
     We now prove this statement by induction on the rank $r$.

     When $r=1$, we have that $M_{Dol}^{ss}=M_{Dol}^s$, so that the base case is true.
     
     Suppose that the case $<r$ is proved.
     To derive a contradiction let $a$ be a geometric point of $A(\omega_{C}(D))$ whose Hitchin fiber $h_{Dol}^{-1}(a)$ contains an irreducible component $Y$ made of strictly semistable Higgs bundles.
     The restrictions of $M_{Dol}^{ss}(C,D,d)$ and $M_{Dol}^s(C,D,d)$ coincides over the open $A^{int}(\omega_C(D))\subset A(\omega_C(D))$, parametrizing integral spectral curves.
     Therefore, the spectral curve $S_a$ corresponding to $a$ can be written as $\sum_i n_i S_{m_i}$, where each $S_{m_i}$ is an integral spectral curve of degree $m_i<r$, and $n_i\in\bN_{>0}$ indicating the mulitiplicity of $S_{m_i}$ inside $S_a$.
     The generic point $\eta$ of $Y$ is represented by a polystable Higgs bundle $(E,\phi)$ on $C_{k(\eta)}$.
     Write $(E,\phi)$ as $(E_1,\phi_1)\oplus (E_2,\phi_2)$, where $(E_i,\phi_i)$ is a polystable Higgs bundle of rank $r_i<r$.
     Let $a_i$ be the Hitchin image of $(E_i,\phi_i)$.
     Let $S_{a_i}$ be the spectral curve corresponding to $a_i$.
     We have an identification of spectral curves $S_a=S_{a_1}+S_{a_2}$.
     Therefore, the direct sum defines a morphism $s: h_{Dol}^{-1}(a_1)\times h_{Dol}^{-1}(a_2)\to h_{Dol}^{-1}(a)$.
     The restriction of $s$ over the component $Y$ maps dominantly to the reduction $Y_{red}$ since its image contains $\eta$.
     Therefore, we have that 
     \begin{equation}
         \mathrm{dim}(Y)=d_{h}(r_1)+d_h(r_2)<d_h(r),
     \end{equation}
     where the last inequality follows from an elementary calculation.
     However, because that every fiber of $h_{Dol}$ over the open $A^{int}(\omega_{C'}(D'))$ has fiber dimension $d_h(r)$, and that $M_{Dol}^{ss}(C,D,r)$ is irreducible \cite[Prop. 2.9.(3)]{maulik-shen-chi}, the existence of $Y$ contradicts the upper semicontinuity \cite[\href{https://stacks.math.columbia.edu/tag/02FZ}{Tag 02FZ}]{stacks-project}.
     Therefore, $Y$ does not exist.
\end{proof}

\begin{prop}[\bf Equidimensionality of the fibers
of $h_{dR}$]
\label{prop: fiber dim of hdr}
    Every  irreducible component of every  fiber of the de Rham-Hitchin morphism $h_{dR}: M_{dR}^{ss}(C,D,pd)\to A(\omega_{C'}(D'))$ has  the same dimension $\dim (h_{dR})$.
\end{prop}
\begin{proof}
 By \cite[Cor. 8.2]{chaudouard-laumon} and \Cref{thm: sst dr},  there is a nonempty open  subset $V \subset A(\omega_{C'}(D'))$ over which $h_{dR}$ has equidimensional fibers of the same dimension $=\dim (h_{dR})$.
    
   By the upper semicontinuity \cite[\href{https://stacks.math.columbia.edu/tag/02FZ}{02FZ}]{stacks-project}, we have that the dimension of every irreducible component of every  fiber of $h_{dR}$ over a point outside $V$ must be larger than or equal to $dim(h_{dR})$.
   If it were larger, just like in the proof of \Cref{prop: irreducible}, we can take any such component $F_{dR}$
   over a closed point $a' \in A(\omega_{C'}(D'))$, form the closure
   $\ov{\mathbb{G}_m \cdot F_{dR}}$. By the properness of $h_{Hod}$, the image 
   $h_{Hod}(\ov{\mathbb{G}_m \cdot F_{dR}})$
   meets $(0_{\mathbb A^1}, 0_{A(\omega_{C'}(D'))})$, so that
   $\ov{\mathbb{G}_m \cdot F_{dR}}$
   meets non-trivially  the nilpotent cone $N_{Dol}^{ss}$ in $M_{Dol}^{ss}$. We obtain a contradiction by the upper-semicontinuity of the dimension fiber of the proper morphism $h_{Hod}$.
\end{proof}

\begin{cor}[\bf Equidimensionality of the fibers
of $h_{dR}^{\Fc A}$]
\label{cor: fiber dim of hdrc}
    The de Rham-Hitchin-residue morphism $h_{dR}^{\Fc A}: M_{dR}^{ss}(C,D,pd)\to \Fc A^{\ov 0}$ has equidimensional fibers of the same dimension.
\end{cor}

\begin{cor}[\bf Flatness of $h_{dR}$ and $h_{dR}^{\Fc A}$]
\label{cor: hdr flat}
    Assume $(pd,r)=1$.
    Then the de Rham-Hitchin morphism $h_{dR}:M_{dR}^{ss}(C,D,pd)\to A(\omega_{C'}(D'))$ and the de Rham-Hitchin-residue morphism $h_{dR}^{\Fc A}: M_{dR}^{ss}(C,D,pd)\to \Fc A^{\ov 0}$ are flat.
\end{cor}
\begin{proof}
    This follows from the smoothness of $M_{dR}^{ss}(C,D,pd)$ \cite[Corollary 5.5]{deC-H-2022geometry}, \Cref{prop: fiber dim of hdr}, \Cref{cor: fiber dim of hdrc}, and miracle flatness \cite[\href{https://stacks.math.columbia.edu/tag/00R4}{Tag 00R4}]{stacks-project}.
\end{proof}

\subsection{Weak Abelian Fibrations}
\label{subs: waf}\;

Recall B.C. Ng\^o's notion
\cite[\S7.1.4]{ngo-lemme-fondamental} of a weak Abelian fibration; see also \cite[\S1.1]{maulik-shen-chi}. 

 Let $S$ be a $k$-scheme. Let $f:M\to S$ be a proper  morphism of  $k$-schemes with $M$ quasi-projective.
Let $g: P\to S$ be a smooth commutative group scheme.
Suppose that there is an action
$S$-morphism $act: P\times_S M\to M$.
Let $P^o$ be the group scheme of neutral components of $P$.
    Let $g^o: P^o\to S$ be the restricted morphism.
    Let $T(P^o)=\cH^{2d-1}(g^0_!\oql)(d)$ be the Tate module.
\begin{defn}[\bf Weak Abelian fibration]
\label{def: weak Abelian fib}
    The data $(S,P,M)$ is called a weak Abelian fibration if the following conditions  are satisfied:
\begin{enumerate}
    \item 
    Every fiber of $g$ is of pure dimension $d$. 
    \item 
    $M$ is of pure dimension $d+dim(S)$.
    \item The action of $P$ on $M$ has affine stabilizers.
    \item 
   The  Tate module  $T(P^o)$ is polarizable as in \cite[\S7.1.4]{ngo-lemme-fondamental}.
\end{enumerate}
\end{defn}

It follows from \cite[\S3.1-3.3]{de2017support-sln} that we have the following weak Abelian fibration: $$(S,P,M)=(A(\omega_{C'}(D')), P^o, M_{Dol}^{ss}(C',D',d)).$$

\begin{prop}[\bf Weak Abelian fibrations]
    \label{prop: weak ab fib}
We have the following weak Abelian fibrations:
    \begin{enumerate}
    \item[(A)] $(S,P,M)=(\Fc A^{\ov{0}},  \wt{P}^{o,\ov{0}}, M_{dR}^{ss}(C,D,pd))$.
    \item[(B)] $(S,P,M)=(A(\omega_{C'}(D')),  P^o, M_{dR}^{ss}(C,D,pd))$.
\end{enumerate}
\end{prop}
\begin{proof} 
The criteria (1) and (4) are about the fibers of $P^o$.
(1) is standard, see, e.g. \cite[Proposition 4.3.5]{ngo-lemme-fondamental}.
(4) is given by \cite[Thm. 3.3.1]{de2017support-sln}.

For (3), we use the BNR correspondence in case (A),
the dR-BNR \Cref{lemma: log dr bnr} in case (B), to relate log-Higgs or log-connections to certain torsion free sheaves on spectral curves.
By pulling-back to the normalization of $S_{a'}$, and by then taking determinants, we see that the stabilizer of the action is an extension of a finite group scheme by an affine group scheme, so that the stabilizer is affine. See also \cite[Prop. 3.2.1]{de2017support-sln}.

Item (2)  follows from the irreducibility (hence pure-dimensionality) of the de Rham moduli space \Cref{prop: irreducible}, and from \Cref{thm: sst dr}, which ensures that the relative dimension is as predicated. 
\end{proof}

\subsection{Cohomological Consequences}\;

\subsubsection{Full support in prime characteristic} $\;$

Given a proper morphism of $k$-varieties $f:X\to Y$,
The Decomposition Theorem of \cite{bbdg} entails that the derived direct image  complex $f_*\IC_X$ decomposes into a direct sum of shifted simple perverse sheaves, each having  support an integral closed subvariety of $Y$.
We say that the complex $f_*\IC_X$ has full support if
the simple perverse sheaves appearing in the Decomposition Theorem  are supported on $Y$.

For any degree $d$, let $h_{Dol}:M_{Dol}(C,D,d,r)\to A(\omega_{C}(D))$ be the Hitchin morphism on the moduli space.
When the field $k =\bC$, it is proved in \cite[Thm. 0.4]{maulik-shen-chi} that $h_{Dol,*}\IC$ has full support.
We note in \Cref{tanto} that the same result holds in characteristic $p>r$. Whereas no essential new idea is needed, we include a proof for completeness. We use \Cref{tanto}
in the proof of \Cref{thm: summary}.(4).

\begin{thm}[\bf Full support when $p>r$]\label{tanto}
    Assume that $p>r$.
    Then $h_{Dol,*}\IC$ has full support.
\end{thm}
\begin{proof}
Let $A^{ell}(\omega_C(D))$ be the elliptic locus of $A(\omega_{C}(D))$ parametrizing integral spectral curves. It is an open and dense subscheme of $A(\omega_{C}(D))$.
The theorem follows from:
\begin{enumerate}
    \item All the supports of the complex $h_{Dol,*}\IC$ are contained in the elliptic locus. $A^{ell}(\omega_{C}(D))$.
    
    \item  The restriction of the complex  $h_{Dol,*}\IC$
    to the elliptic locus has full support.
\end{enumerate}

The proof of (1) is identical to the one in  \cite[\S4.5]{maulik-shen-chi}:
loc. cit. works over the complex numbers, but the argument is characteristic free
(cf. \cite[\S9]{chaudouard-laumon}).

We show that (2) follows from Ng\^o's Support theorem when $p>r$.
Over $A^{ell}(\omega_{C}(D))$, the Higgs bundles in the Hitchin fibers do not have proper sub-Higgs-bundles, since the spectral curves are integral.
Thus, $h_{Dol}^{-1}(A^{ell}(\omega_{C}(D))$ is contained in the stable locus of $M_{Dol}^{ss}(\omega_{C}(D))$, which is smooth \cite[Thm. 4.3]{chaudouard-laumon}.
Therefore, over $A^{ell}(\omega_{C}(D))$, we have that $\IC=\oql$, i.e. the intersection complex is the constant sheaf.

Let $h^{ell}_{Dol}$ be the restriction of $h_{Dol}$ to $A^{ell}(\omega_{C}(D))$.
We now apply Ng\^o's Support Theorem \cite[Thm. 7.3.1]{ngo-lemme-fondamental} to $h^{ell}_{Dol,*}\IC=h^{ell}_{Dol,*}\oql$:
Given any closed point $a\in A(\omega_{C}(D))$, let $\delta(a)$ be the dimension of the affine part of the Jacobian of the spectral curve $S_a$.
Given any closed subvariety $Z\subset A^{ell} (\omega_{C}(D))$, let $\delta(Z)$ to be the minimal 
(= general) value of $\delta$ on $Z$
\cite[Thm. 7.3.1]{ngo-lemme-fondamental} shows that, for any support $Z$ of $h^{ell}_{Dol,*}\oql$, we have the inequality $\mathrm{codim}(Z)\leq \delta_Z$.
Furthermore, the equality holds if and only if $Z$ is the  support of a non-trivial direct summand of the top degree cohomology sheaf
$R^{top}h_{Dol,*}^{ell}\oql$ (here, $top$ is twice the dimension of the fibers of the Hitchin morphism).

For any closed subvariety $Z$ of $A^{ell}(\omega_{C}(D))$, the equality does hold in characteristic $p>r$, because of the Severi inequality \cite[Thm. 3.3]{MRV2019fourier}. It is here that we need the lower bound $p>r$.
Therefore, any support of $h_{Dol,*}^{ell}\oql$ is the support of a non-trivial direct summand of  $R^{top}h_{Dol,*}^{ell}\oql$.
The geometric fibers of $h_{Dol}^{ell}$ are compactified Jacobians of integral curves lying on a smooth surface, thus they are integral by \cite[Thm. 9]{altman-iarrobino-kleiman-irreducibility}.
Therefore, we have $R^{top}h_{Dol,*}^{ell}\oql=\oql$,
which is full support, and (2) follows.
\end{proof}

\subsubsection{Known cohomological results}$\;$

The following \Cref{thm: summary} summarizes
the various cohomological results
we are aware of that relate the various moduli spaces we deal with in this paper.
In the second cohomological NAHT \Cref{thm: coh pnaht2} below, we establish yet another additive isomorphism of the same form as
\Cref{thm: summary}.\ref{m5}, in the case $d'=dp$, $C_2=C$, and $C_1=C'$, the Frobenius twist of $C$. The proof hinges on 
the canonical semistable NAHT morphism \eqref{eq: very good morphism schemes}.

\begin{thm}\label{thm: summary}$\;$

\begin{enumerate}
\item\label{m1}
\cite[\bf  Thm. 2.1, No-poles Cohomological NAHT]{de-zhang-naht}
    Let the degree $d=d'p$ be a multiple of $p$ and that $\gcd (r,d)=1$. Let $C$ be a curve.
There is a natural isomorphism of  cohomology rings 
\[
H^*(M_{Dol}^{ss}(C, d'p))\cong H^*(M_{dR}^{ss}(C,d'p)),
\]
that is a filtered isomorphism for  the perverse Leray filtrations induced by the Hitchin and by the de Rham-Hitchin morphisms.

\item\label{m2}
\cite[\bf Thm. 3.6, Cohomological Log-NAHT]{deC-H-2022geometry}
    Let $\gcd (r,d)=1$. Let $C$ be a curve.
    There is a canonical isomorphism of cohomology rings:
\[
        H^*(M_{Dol}^{ss}(C,D, d),\oql)\cong H^*(M_{dR}^{ss}(C,D,d),\oql)
\]
    that is a filtered isomorphism for  the perverse Leray filtrations induced by the Hitchin and by the de Rham-Hitchin morphisms.

\item\label{m3} \cite[\bf 
Thm. 3.8, Changing curves]{deC-H-2022geometry}
    Let $\gcd(r,d)=1$.
    Let $C_1$ and $C_2$ be two  curves of the same genus.
    Let $D_i$ be a collection of mutually distinct points of $C_i, i=1,2$.
    Assume that $\deg(D_1)=\deg(D_2)$.
    Then there are isomorphisms of cohomology rings:
\[
        H^*(M_{Dol}^{ss}(C_1,D_1,d),\oql)\cong H^*(M_{Dol}^{ss}(C_2,D_2,d),\oql),
\]
\[
        H^*(M_{dR}^{ss}(C_1,D_1,d),\oql)\cong H^*
        (M_{dR}^{ss}(C_2,D_2,d),\oql),
\]
    that are  filtered isomorphisms for  the perverse Leray filtrations induced by the Hitchin and  by the de Rham-Hitchin morphisms.

\item\label{m4}
\cite[\bf  Thm. 0.1, Changing degrees]{maulik-shen-chi}
Assume that $p>r$.
    For any integer $d,d'$, we have an isomorphism of intersection cohomology groups:
\[
        I\!H^*(M_{Dol}^{ss}(C,D,d),\oql)\cong I\!H^*(M_{Dol}^{ss}(C,D,d'),\oql).
\]
    Furthermore, the isomorphism preserves the perverse Leray filtrations induced by the Hitchin morphisms.

    \item\label{m5}
    {\bf [Comparing Dolbeault and de Rham: different curves, different poles, different degrees]}
    Let  $\gcd(r,d)=\gcd (r,d')=1$, and that $p>r$. Let $C_1$ and $C_2$ be two  curves of the same.
    Let $D_i$ be a collection of mutually distinct points of $C_i,i=1,2$. Assume that $\deg(D_1)=\deg(D_2)$.
    Then we have an isomorphism of cohomology groups:
\[
    \label{eqn: dol c1 d dr c2 d'}
        H^*(M_{A}^{ss}(C_1, D_1, d),\oql)\cong H^*(M_{B}^{ss}(C_2, D_2, d'),\oql),
\]
    where $A$ (resp. B) can be either $Dol$, or $dR$.
    Furthermore, the isomorphism preserves the perverse Leray filtrations induced by the Hitchin and by the  de Rham-Hitchin morphisms.
    \end{enumerate}
\end{thm}

\begin{proof}
Parts \ref{m1}, \ref{m2} and \ref{m3} are proved in the given references.
Part \ref{m5} is obtained by combining parts \ref{m1}, \ref{m2}, \ref{m3} and \ref{m4}, and is listed as a convenient summary.
We are left with proving \ref{m4}.

In view of \Cref{tanto} on full supports, the proof of this theorem is identical to the one of the characteristic zero case in  \cite[Thm. 0.1]{maulik-shen-chi}. We sketch the proof for the convenience of the reader.

Let $h(d)$ and $h(d')$ be the Hitchin morphisms for the degree $d$ and $d'$ Dolbeault moduli spaces.
The theorem follows from the existence of an isomorphism  in $D^b_c(A(\omega_{C}(D)),\oql):$
\begin{equation}
\label{eq: chi ind on ic new}
    h(d)_*\IC_{M_{Dol}^{ss}(C,D,d)}\cong h(d')_*\IC_{M_{Dol}^{ss}(C,D,d')}.
\end{equation}
In turn, the isomorphism \eqref{eq: chi ind on ic new} follows from:
\begin{enumerate}
    \item \eqref{eq: chi ind on ic new} holds over a dense open of $A(\omega_{C}(D))$.
    \item Both sides of \eqref{eq: chi ind on ic new} have full supports when $p>r$ (cf. \Cref{tanto}).
\end{enumerate}
Item (1)  holds in characteristic $p>0$: 
over the open and dense subset 
$A(\omega_{C}(D))^{sm} \subset A(\omega_{C}(D))$ of smooth spectral curves, the fibers of $h(d)$ and $h(d')$ are both torsors under the relative Jacobian of the smooth spectral curves, in which case (1) is well-known: for example, it follows from the same argument as in the proof of \cite[Thm. 5.1]{dCGZ}, which uses the Decomposition Theorem of \cite{bbdg} to break both sides of \eqref{eq: chi ind on ic new} into direct sums of shifted perverse cohomology sheaves, and then uses the Homotopy Lemma \cite[Lem. 3.2.3]{laumon-ngo2008lemme} to show that the perverse cohomology sheaves are isomorphic.
\end{proof}

\subsubsection{Cohomological Embedding} $\;$

In this subsection, by building on 
\Cref{prop: over caim}, we prove
\Cref{thm: injection} and its \Cref{thm: coh pnaht2}.

\begin{thm}[\bf Split monomorphism]
\label{thm: injection}
    Assume that $p>r$.
    The semistable NAHT morphism \eqref{eq: very good morphism schemes} induces a split monomorphism in $D^b_c(A(\omega_{C'}(D')),\oql)$
    \begin{equation}
    \label{eq: inclusion ic whole}
    h_{Dol,*}\IC_{M_{Dol}^{ss}(C',D',d)}\hookrightarrow h_{dR,*}\IC_{M_{dR}^{ss}(C,D,pd).}
    \end{equation}
\end{thm}

The proof of \Cref{thm: injection} is postponed after the proof of \Cref{prop: over caim}.

The following corollary  is an immediate consequence of \Cref{thm: injection} and  \Cref{thm: summary}.(\ref{m4}). The perhaps surprising aspect is that, while the split monomorphism \eqref{eq: inclusion ic whole} is certainly not an isomorphism, it  induces an isomorphism on the global intersection cohomology groups.

\begin{cor}[\bf The Second Cohomological Log-NAHT]
\label{thm: coh pnaht2}
Assume that $(pd,r)=1$ and that $p>r$. 
Then the semistable NAHT morphism \eqref{eq: very good morphism schemes} induces a filtered  isomorphism of  cohomology groups
    \begin{equation}
    \label{eq: dol d to dr pd coh}
        \Big(H^*(M_{Dol}^{ss}(C',D',d)),P^h\Big)\cong 
        \Big(H^*(M_{dR}^{ss}(C,D,pd)), P^{h_{dR}}\Big),
    \end{equation}
    where the filtrations are the perverse Leray filtrations induced by the Hitchin and by the  de Rham-Hitchin morphisms.
\end{cor}
\begin{proof}
The split monomorphism  \eqref{eq: inclusion ic whole}  induces a split inclusion of intersection cohomology groups
    \begin{equation}
        \label{eq: the desired inc}
        I\!H^*(M_{Dol}^{ss}(C',D',d),\oql)\hookrightarrow I\!H^*(M_{dR}^{ss}(C,D,pd),\oql),
    \end{equation}
    which is filtered strict  for  the perverse Leray filtrations induced by the Hitchin and by the  de Rham-Hitchin morphisms. \Cref{thm: summary}.(\ref{m4})
    implies that this split filtered strict inclusion is a filtered isomorphism.
\end{proof}

Let us start with one consequence of the semistable Log-NAH
\Cref{thm: sst dr}.
Recall that $\Fc A_{im}$ is the open subscheme of $\Fc A$ which is the image of $\sH^o$.
Recall the morphisms
$h^{\Fc A}_{Dol}:\wt{M}_{Dol}^{ss}(C',D',d)\to \Fc A$, and $ h_{dR}^{\Fc A}: M_{dR}^{ss}(C, D, pd)\to \Fc A$.
\begin{prop}
\label{prop: over caim}
    We have the following isomorphism in $D^b_c(\Fc A_{im},\oql):$
    \begin{equation}
\label{eq: over caim}
    (h^{\Fc A}_{Dol,*}\IC_{\wt{M}_{Dol}^{ss}})_{\Fc A_{im}}\cong (h^{\Fc A}_{dR,*}\IC_{M_{dR}^{ss}})_{\Fc A_{im}}.
\end{equation}
\end{prop}
\begin{proof}
    The proof is the same as the proof of \cite[Thm. 5.1]{dCGZ}. Namely, we first use the Decomposition Theorem to break both sides into the direct sum of shifted respective perverse cohomology sheaves over $\Fc A_{im}$.
The morphism \eqref{eq: very good morphism schemes} entails that \'etale locally over $\Fc A_{im}$ we have the desired isomorphism \eqref{eq: over caim}.
    Over the fiber products (``intersection") of the \'etale neighborhoods, the local isomorphisms differ from each other by an action of the connected group scheme $H^o/\Fc A_{im}$.
    However, the homotopy lemma \cite[Lem. 3.2.3]{laumon-ngo2008lemme} entails that connected group schemes must act trivially on the perverse cohomology sheaves.
    Therefore, we can glue the local isomorphisms together and get a global one, indeed the category perverse sheaves satisfies \'etale descent \cite[Prop. 3.2.2]{bbdg}.
\end{proof}

\begin{proof}[Proof of \Cref{thm: injection}]
Recall the projection morphism $p_1:\Fc A\to A(\omega_{C'}(D'))$ as in \eqref{diag: ca and wtmdol}.
Let $A^U(\omega_{C'}(D'))\subset A(\omega_{C'}(D'))$ be a nonempty open whose inverse image 
$$p_1^{-1}(A^U(\omega_{C'}(D')))=:\Fc A^U$$ 
lies in $\Fc A_{im}$.

Let $\wt{M}_{Dol}(C')^U, M_{Dol}^U, \Fc A^U, M_{dR}^U$
be the inverse images of $A^U(\omega_{C'}(D'))$ under the corresponding maps to $A(\omega_{C'}(D'))$.
We have the following commutative diagram where the top square is Cartesian:
\begin{equation}
\label{diag: xi hc hdrc}
\xymatrix{
M_{dR}^U \ar[r]^-{h_{dR}^{\Fc A,U}} \ar[dr]_-{h_{dR}^U} &
\Fc A^U \ar[d]_-{p_1^U} &
\wt{M}_{Dol}^U(C') \ar[d]_-{\xi^U} \ar[l]_-{h^{\Fc A,U}_{Dol}} 
    \\
     &    A^U(C')&
             M_{Dol}^U. \ar[l]_-{h^U_{Dol}}
}
\end{equation}

Apply $p_{1,*}$ to the restriction of \eqref{eq: over caim} over $\Fc A^U$, we obtain $h^U_{Dol,*}\xi^U_*\IC_{\wt{M}_{Dol}^U}\cong h_{dR,*}^U\IC_{M_{dR}^U}$.

Note that $\xi^U$ is a quotient by an action of an abstract finite group.
Indeed, the Artin-Schreier morphism on $\bA^1$ is the quotient by the finite abstract group $\bZ/p\bZ$, acting by translation.
    Therefore, the morphisms $AS_{\Ft}$ and $AS_{\Fc}$ are both quotient morphisms by $(\bZ/p\bZ)^{\oplus r}$.
By the construction of $\Fc A$ and $\wt{M}_{Dol}^{ss}$ as in \eqref{diag: ca and wtmdol}, we have that $\xi^U$ is a quotient by $(\bZ/p\bZ)^{\oplus r\deg(D)}$. 

\cite[Lem. 4.1.4.(ii)]{decataldo-cambridge} entails that $\IC_{M_{Dol}^U}$ is a direct summand of $\xi_*^U\IC_{\wt{M}_{Dol}^U}$.
Therefore, we have the following monomorphism  and isomorphism in $D^b_c(A^U(\omega_{C'}(D'))):$
\begin{equation}
\label{eq: inclusions on u}
    h_{Dol,*}^U\IC_{M_{Dol}^U}\hookrightarrow h_{Dol,*}^U\xi_*^U\IC_{\wt{M}_{Dol}^U}\cong h_{dR,*}^U\IC_{M_{dR}^U}.
\end{equation}

By \Cref{tanto}, the complex $h_{Dol,*}\IC_{M_{Dol}^{ss}}$ has full support.
Therefore the Decomposition Theorems for $h_{Dol,*}\IC_{M_{Dol}^{ss}}$ and $h_{dR,*}\IC_{M_{dR}^{ss}}$ are of the following form:
\[h_{Dol,*}\IC_{M_{Dol}^{ss}}=\bigoplus_i \IC(L_{i,Dol})[i],\quad h_{dR,*}\IC_{M_{dR}^{ss}}=\bigoplus_i\IC(L_{i,dR})[i]\oplus\bigoplus_{Z\in\cZ} F_Z,\]
 where $L_{i,Dol}, L_{i,dR}$'s are lisse sheaves on some open of $A(\omega_{C'}(D'))$, $\cZ$ is a finite set of proper closed subschemes of $A(\omega_{C'}(D'))$, and each $F_Z$ is a shifted perverse sheaf supported on $Z$.
 The inclusion \eqref{eq: inclusions on u} then entails that each $L_{i,Dol}$ is a direct summand of $L_{i,dR}$.
 Therefore, the inclusion \eqref{eq: inclusions on u} extends to the desired inclusion \eqref{eq: inclusion ic whole} over $A(\omega_{C'}(D'))$.
\end{proof}

\appendix
\section{Log-connections on torsors}
\label{section: two views on residues}

In this appendix, we develop the formalism of log-connections on affine schemes, affine group schemes, and their torsors. 
We confirm that the residues and $p$-curvatures of them behave as one may expct when taking duals and twisted products \eqref{eq: pcuvr twi prod}-\eqref{eq: res dual}.
The no pole case, where of course there are no residues,  is developed in \cite[\S A.1, A.5-6]{chen-zhu}.

In this appendix, we let $C$  be a smooth $k$-variety and $D$  be a simple normal crossing divisor on $C$.
By log-connections, we always mean log-connections with log-poles at $D$.
We denote by $\Omega^1_C({\rm log} D)$
the locally free $\mathcal O_C$-module
of log-$1$-forms on $C$  with simple poles along $D$ and by $Der_D(C)$ its dual $\mathcal O_C$-module; cf. \cite[\S4.1, 4.2]{katz1970}.
In the body of the paper,   we only use the case when $C$ is a curve, and we use the notation $\omega_C(D)$ for $\Omega^1_C({\rm log} D)$
and $T_C(-D)$ for $Der_D(C)$.

\subsubsection{Log-connections on affine schemes} $\;$
Let $f:Y\to C$ be an affine morphism of $k$-schemes.
A log-connection $\nabla$ on $Y$ relative to $C$ is a log-connection $\nabla_Y$ on $f_*\cO_Y$ which, in addition, satisfies the Leibnitz rule for the multiplication on $f_*\cO_Y$, i.e., for each local section $\partial$ of $Der_D(C)$, and local sections $y_1$ and $y_2$ of $f_*\cO_Y$, we have

\begin{equation}
\label{eq: derivation}
    \nabla_Y(\partial)(y_1y_2)=y_1\nabla_Y(\partial)(y_2)+y_2\nabla_Y(\partial)(y_1)
    \quad {\rm in} \;\;
    f_* \cO_Y.
\end{equation}

Given two affine $C$-schemes $f_Y: Y\to C$ and $f_Z: Z\to C$ with log-connections $\nabla_Y$ and $\nabla_Z$ and a $C$-morphism $g:Y\to Z$, we say that $g$ is compatible with the log connections if for each local section $\partial$ of $Der_D(C)$ and local section $z$ of $f_{Z,*}\cO_Z$, we have $g^{\#} (\nabla_Z(\partial)(z)  )=\nabla_Y(\partial)(g^{\#}z)$.

\subsubsection{Product log-connections} $\;$
Given two affine morphisms $f: Y\to C$ and $h: Z\to C$,  with log-connections $\nabla_{Y}$ and $\nabla_{Z},$ the product connection $\nabla_{Y}\times \nabla_{Z}$ on the fiber product $Y\times_C Z$ is defined as follows:
for local sections $\partial, y$ and $z$ of $Der_D(C), f_*\cO_Y$
and $h_*\cO_Z$ respectively, we set
\begin{equation}
\label{eq: prod conn}
    (\nabla_Y\times\nabla_Z)(\partial)(y\otimes z) :=
\Big(\nabla_Y(\partial)(y)\Big)\otimes z+y\otimes \nabla_Z(\partial)(z).
\end{equation}

\subsubsection{Log-connections on affine group schemes} $\;$

Let $\cG$ be an affine group scheme over $C$.
A log-connection on $\cG$ is a log-connection on the  affine $C$-scheme underlying $\cG$ which is furthermore compatible with the unit
$u: g_*{\cO}_{\cG} \to {\cO}_C$
(i.e. $\partial \circ u = u 
 \circ \nabla(\partial)$)
 and with the multiplication $\cG\times_C \cG\to \cG$ where we put the product log-connection on the source.

\subsubsection{Actions compatible with log-connections} $\;$

Let $\nabla_{\cG}$ be a log-connection  $\cG$.
Given an affine $C$-scheme with a log-connection $(f:Y\to C,\nabla_Y)$,
we say that $(\cG,\nabla_{\cG})$ acts on $(Y,\nabla_Y)$ if $\cG$ acts on $Y$ and if the following diagram is commutative:
\begin{equation}
\label{diag: conn on tors}
    \xymatrix{
f_*\cO_Y \ar[r]^-{act^{\#}} \ar[d]_-{\nabla_Y} &
f_*\cO_Y\otimes g_*\cO_{\cG} \ar[d]^-{\nabla_Y\otimes 1+1\otimes \nabla_{\cG}}\\
f_*\cO_Y\otimes \Omega^1_{C}(\mathrm{log}D) \ar[r]_-{act^{\#}\otimes 1}  & f_*\cO_Y\otimes g_*\cO_{\cG}\otimes \Omega^1_C(\mathrm{log}D).
}
\end{equation}

\subsubsection{Twisted products} $\;$
Let $f:Y\to C$ and $h: Z\to C$ be two affine $C$-schemes with $\cG$-actions.
We define the twisted product $Y\times^{\cG}Z$ to be the GIT quotient $(Y\times Z)\git \cG$ over $C$, i.e. we have
\begin{equation}
\label{eq: equalizer}
    Y\times^{\cG} Z:=\mathrm{Spec}_{\cO_C}(\mathrm{Eq}(f_*\cO_Y\otimes h_*\cO_Z\rightrightarrows f_*\cO_Y\otimes h_*\cO_Z\otimes g_*\cO_{\cG})),
\end{equation}
where Eq is the equalizer and the two arrows are the action and projection comorphisms.
The twisted product $Y\times^{\cG} Z$ is well-defined because the formation of the equalizer respects Zariski localizations over $C$.
Indeed, more is true: \cite[p.243, Lemma 2]{seshadri1977geometric} shows that the formation of the equalizer (thus of $Y\times^{\cG} Z$) commutes with flat base change to $C$.

\subsubsection{Log-connections on twisted products} $\;$
Suppose that there are log-connections $\nabla_Y, \nabla_Z$ and  $\nabla_{\cG}$ on $Y,Z$ and $G$ respectively, and that the actions respect the log-connections.
We thus obtain the log-connection $\nabla_Y \times \nabla_Z$ on $Y\times Z$. By using the definitions above, one can check that for every log-derivation $\partial$, $(\nabla_Y\times\nabla_Z) (\partial)$
preserves  the equalizer in \eqref{eq: equalizer}, thus inducing a connection 
on the twisted product
\begin{equation}
\label{eq: twi prod conn}
    \nabla_Y\otimes \nabla_Z \;\text{ on } Y\times^{\cG} Z.
\end{equation}

\subsubsection{Log-connections on torsors} $\;$

Let $s: E\to C$ be a $\cG$-torsor over $C$. Since affineness satisfies (fpqc, hence) fppf descent \cite[\href{https://stacks.math.columbia.edu/tag/0244}{Tag 0244}]{stacks-project}, it follows that $E$ is affine over $C$.
A log-connection on the $\cG$-torsor $E$ is a log-connection $\nabla_E$ such that the $\cG$-action on $E$ gives rise to a $(\cG,\nabla_{\cG})$-action on $(E,\nabla_E)$.

\subsubsection{Adjoint bundle} $\;$

The $\cO_C$-linear derivations on $s_*\cO_E$ form a quasi-coherent sheaf $Der_{\cO_C}(s_*\cO_E),$ which is a subsheaf of the quasi-coherent sheaf $End_{\cO_C}(s_*\cO_E)$.
Given a derivation $v: s_*\cO_E\to s_*\cO_E$ in $Der_{\cO_C}(s_*\cO_E)$, we say that $v$ is $\cG$-invariant if for any local section $e$ of $s_*\cO_E$, we have that $act^{\#}(v(e))=(v\otimes 1)(act^{\#}(e)),$ where $act^{\#}: s_*\cO_E\to s_*\cO_E\otimes g_*\cO_{\cG}.$
The $\cG$-invariant derivations form a subsheaf $\mathrm{ad}(E)$ of $Der_{\cO_C}(s_*\cO_E)$.
When $E=\cG$, we have $\mathrm{Lie}(\cG)=\mathrm{ad}(\cG)$ by the definition of the former in \cite[\S A.7.1-2]{conrad2015pseudo}.
 There is a canonical isomorphism of $C$-schemes $\alpha: E\times^{\cG} \mathrm{Lie}(\cG)\xrightarrow{\sim} \mathrm{ad}(E)$:
 given local sections $e$ and $\partial$ of $E$ and $\mathrm{Lie}(\cG)$, we have an isomorphism $f_e:G\xrightarrow{\sim} E$, then $\alpha$ sends $(e,\partial)$ to the pushforward derivation $f_{e,*}\partial$.
 In particular, if $\mathcal G$ is commutative, then the twisted product is $Lie (\mathcal G)$ and we have the
 \begin{lemma}
 \label{lemma: comm ad}
     If $\cG$ is commutative, then there is a canonical isomorphism $\mathrm{Lie}(\cG)\xrightarrow{\sim} \mathrm{ad}(E)$. 
 \end{lemma}

 \subsubsection{$L$-twisted $\cG$-Higgs bundles} $\;$
 
Let $\cG$ be an affine group scheme over $C$.
Let $L$ be an invertible sheaf on $C$.
An $L$-twisted $\cG$-Higgs bundle is a pair $(E,\phi_E)$, where $E$ is a $\cG$-torsor over $C$ and $\phi_E$ is a section in $\Gamma(C, \mathrm{ad}(E)\otimes L)$.

\subsubsection{Residues}\label{sssec: residues} $\;$

Let $\wt{D}$ be the normalization of $D$.
Let $(s:E\to C,\nabla_E)$ be an affine $C$-scheme with a log-conneciton.
Let $\partial^{can}$ be the canonical section of $Der_D(C)|{\wt{D}}$.
We have that  $\nabla_E(\partial^{can})$ lands in $End_{\cO_{\wt{D}}}(s_*\cO_{E_D}|\wt{D})$.
The condition \eqref{eq: derivation} implies that $\nabla_E(\partial^{can})$ lands in the subsheaf of derivations $Der_{\cO_{\wt{D}}}(s_*\cO_{E_D}|\wt{D})$.
If furthermore $(E,\nabla_E)$ is a torsor under $(\cG,\nabla_{\cG})$, then \eqref{diag: conn on tors} implies that $\nabla_E(\partial^{can})$ is a $\cG_D$-invariant derivation.
Since $\mathrm{ad}(E)$ is given by the $\cG$-invariant derivations, we have a well-defined section $\mathrm{res}_D(\nabla_E)\in \Gamma(\wt{D}, \mathrm{ad}(E)|{\wt{D}})$
(cf. \cite[(3.8.4)]{deligne1970equations}).

\subsubsection{$p$-Curvature}\label{sssec: def p curv} $\;$

 Let $(E,\nabla_E)$ be a torsor under $(\cG,\nabla_{\cG})$, where both $\nabla$s are log-connections.
Consider the assignment $\wt{\Psi}(\nabla_E): Der_D(C) \to End_k(\cO_E)$ given by $\partial\mapsto \nabla(\partial)^p-\nabla(\partial^{[p]})$.
The same proof as \cite[Lem. 5.4]{herrero-zhang-mero} shows that $\wt{\Psi}(\nabla_E)$  is a section of the sheaf $\mathrm{ad}(E)$ of $\cG$-invariant derivations, that $\wt{\Psi}$ is $p$-linear, so that
 $\wt{\Psi}(\nabla_E)$ induces a
 section   $\Psi(\nabla_E)\in \Gamma(C, \mathrm{ad}(E)\otimes \mathrm{fr}_C^*(\Omega_C^1(\mathrm{log}D)))$ (recall that $\mathrm{fr}_C:C\to C$ is the absolute Frobenius morphism),
which we name the $p$-curvature
of $(E, \nabla_E)$ as a $(\cG, \nabla_{\cG})$-torsor.
The following lemma is an enhancement of Frobenius descent.

\subsubsection{Residue and $p$-curvature of twisted products} $\;$

Let $(\cG,\nabla_{\cG})$ be an affine commutative group scheme over $C$ with a log-connection.
Let $(E,\nabla_E)$ and $(F,\nabla_F)$ be torsors under $(\cG,\nabla_{\cG})$.
The product formula \eqref{eq: prod conn} and \Cref{lemma: comm ad} entail that we have the following identities on $p$-curvatures and residues:
\begin{equation}
    \label{eq: pcuvr twi prod}
    \Psi(\nabla_E\otimes\nabla_F) =\Psi(\nabla_E)+\Psi(\nabla_F)\in \Gamma(C, \mathrm{Lie}(\cG)\otimes \mathrm{fr}_C^*(\Omega_C^1(\mathrm{log}D))),
\end{equation}
\begin{equation}
\label{eq: res twi prod}
    \mathrm{res}_D(\nabla_{E}\otimes \nabla_F)=\mathrm{res}_D(\nabla_E)+\mathrm{res}_D(\nabla_F) \in \Gamma(\wt{D}, \mathrm{Lie}(\cG)|\wt{D}).
\end{equation}

\subsubsection{Residues and $p$-curvature for morphisms of group schemes} $\;$

Let $(\cJ,\nabla_{\cJ})$ and $(\cG,\nabla_{\cG})$ be two affine group schemes over $C$ with log-connections.
Let $c:\cJ\to \cG$ be a morphism of group schemes compatible with the connections.
Consider the left action of $\cJ$ on $\cJ\times_C \cG$ by $j_1\cdot (j_2,g)=(j_2j_1^{-1},c(j_1)g)$.
The morphism $m:\cJ\times_C\cG\to \cG$ defined by $(j,g)\mapsto c(j)g$ is $\cJ$-invariant, thus inducing a morphism $\cJ\times^{\cJ}\cG\to \cG$, which is indeed an isomorphism.
It is clear that the differential $dm:\mathrm{Lie}(\cJ)\times_C \mathrm{Lie}(\cG)\to \mathrm{Lie}(\cG)$ maps $(\partial_j, \partial_g)$ to $dc(\partial_j)+\partial_g$.
Combined with product formula \eqref{eq: prod conn}, we have the following identities on $p$-curvatures and residues:
\begin{equation}
    \label{eq: p cj}\Psi(\nabla_{\cJ}\otimes\nabla_{\cG})=dc(\Psi(\nabla_{\cJ}))+\Psi(\nabla_{\cG})\in \Gamma(C, \mathrm{Lie}(\cG)\otimes\mathrm{fr}_C^*(\Omega_C^1(D)))
\end{equation}
\begin{equation}
    \label{eq: res cj}
    \mathrm{res}_D(\nabla_{\cJ}\otimes\nabla_{\cG})=dc(\mathrm{res}_D(\nabla_{\cJ}))+\mathrm{res}_D(\nabla_{\cG})\in \Gamma(\wt{D}, \mathrm{Lie}(\cG)|\wt{D}).
\end{equation}

\subsubsection{Dual log-connections}\label{sssec: dual log conn} $\;$

Let $g:\cG\to C$ be an affine commutative group scheme on $C$ with a log-connection $\nabla_{\cG}$.
Let $(s:E\to C,\nabla_E)$ be a torsor under $(\cG,\nabla_{\cG})$.
The dual torsor $E^{\vee}$ is defined as the presheaf that sends a $C$-scheme $U$ to the set of $\cG(U)$-equivariant isomorphisms $E(U)\to \cG(U)$, where $\cG(U)$ and $E(U)$ are the $\cG(U)$-sets of $U$-points of $\cG$ and $E$.
Since $\cG$ is commutative, we have that $E^{\vee}$ is a  torsor under $\cG$.
Since $\cG$ is affine, we have that $E^{\vee}$ is representable by an affine scheme over $C$ by \cite[\href{https://stacks.math.columbia.edu/tag/049C}{Tag 049C}]{stacks-project}.
Furthermore, there is a canonical $C$-isomorphism of $\cG$-torsors $\epsilon: E^{\vee}\times^{\cG}E\xrightarrow{\sim}\cG$.

\begin{lemma}
\label{lemma: dual is dual}
    There is a unique log-connection $\nabla_{E^{\vee}}$ on $E^{\vee}$ such that $\epsilon$ induces an isomorphism of $(\cG,\nabla_{\cG})$-torsors:
    \begin{equation}
    \label{eq: dual is dual}
        (E^{\vee}\times^{\cG}E, \nabla_{E^{\vee}}\otimes \nabla_E)\xrightarrow{\sim} (\cG,\nabla_{\cG}).
    \end{equation}
\end{lemma}
\begin{proof}
  Given any $\cG$-torsor $F$, let $Conn(F)$ be the set of log-connections on $F$ compatible with $\nabla_{\cG}$.
   Since $\cG$ is commutative, by \Cref{lemma: comm ad}, we have that $Conn(F)$ is an affine space over the vector space $Higgs(\cG):=\Gamma(C, \mathrm{Lie}(\cG)\otimes\Omega_C^1(D))$.
   Given a connection $\nabla$ on $E^{\vee}\times^{\cG}E$, we obtain a connection $\epsilon_*\nabla$ on $\cG$ given by the formula $\epsilon^{\#,-1}\nabla(\epsilon^{\#}-)$, where $\epsilon^{\#}: \cO_{\cG}\xrightarrow{\sim}\cO_{E^{\vee}\times^{\cG}E}$ is the comorphism.
   There is a map of $Higgs(\cG)$-affine spaces $f: Conn(E^{\vee})\to Conn(\cG)$ given by $\nabla'\mapsto \epsilon_*(\nabla'\otimes\nabla_E)$; in particular,
    $f$ is automatically a bijection. It follows that there is a unique $\nabla_{E^{\vee}}$ such that $f(\nabla_{E^{\vee}})=\nabla_{\cG}$, i.e. \eqref{eq: dual is dual} is satisfied. 
\end{proof}

Equations \eqref{eq: dual is dual}, \eqref{eq: pcuvr twi prod} and \eqref{eq: res twi prod}, taken  together, tell us that we have the following identity of $p$-curvatures and residues:
\begin{equation}
\label{eq: pcurv dual}
    \Psi(\nabla_{E^{\vee}})=\Psi(\nabla_{\cG})-\Psi(\nabla_E),
\end{equation}
\begin{equation}
\label{eq: res dual}
    \mathrm{res}_D(\nabla_{E^{\vee}})=\mathrm{res}_D(\nabla_{\cG})-\mathrm{res}_D(\nabla_E).
\end{equation}

\bibliographystyle{alphaCap}
\footnotesize{\bibliography{logp.bib}}

\begin{thebibliography}{dCMSZ23}

\bibitem[AGV73]{SGA4}
M~Artin, A~Grothendieck, and JL~Verdier.
\newblock Th{\'e}orie des topos et cohomologie {\'e}tale des sch{\'e}mas. Tomes
  1--3.
\newblock {\em Lecture Notes in Mathematics}, 269:270, 1973.

\bibitem[AIK77]{altman-iarrobino-kleiman-irreducibility}
Allen~B Altman, Anthony Iarrobino, and Steven~L Kleiman.
\newblock Irreducibility of the compactified Jacobian.
\newblock {\em Real and complex singularities, Oslo 1976}, pages 1--12, 1977.

\bibitem[Alp14]{alper_adequate}
Jarod Alper.
\newblock Adequate moduli spaces and geometrically reductive group schemes.
\newblock {\em Algebr. Geom.}, 1(4):489--531, 2014.

\bibitem[BB07]{bezrukavnikov2007geometric}
Roman Bezrukavnikov and Alexander Braverman.
\newblock Geometric Langlands Correspondence for D-modules in Prime
  Characteristic: the GL (n) Case.
\newblock {\em Pure and Applied Mathematics Quarterly}, 3(1):153--179, 2007.

\bibitem[BBDG18]{bbdg}
Alexander Beilinson, Joseph Bernstein, Pierre Deligne, and Ofer Gabber.
\newblock {\em Faisceaux pervers}, volume~4.
\newblock Soci{\'e}t{\'e} math{\'e}matique de France Paris, 2018.

\bibitem[BD91]{beilinson-drinfeldquantization}
Alexander Beilinson and Vladimir Drinfeld.
\newblock Quantization of Hitchin’s integrable system and Hecke eigensheaves,
  1991.

\bibitem[BDP17]{balaji2017complete}
Vikraman Balaji, Pierre Deligne, and AJ~Parameswaran.
\newblock On complete reducibility in characteristic $ p$.
\newblock {\em {\'E}pijournal de G{\'e}om{\'e}trie Alg{\'e}brique}, 1, 2017.

\bibitem[BL22]{bhatt2022prismatization}
Bhargav Bhatt and Jacob Lurie.
\newblock The prismatization of $ p $-adic formal schemes.
\newblock {\em arXiv preprint arXiv:2201.06124}, 2022.

\bibitem[BLR90]{blr-neron}
Siegfried Bosch, Werner L\"{u}tkebohmert, and Michel Raynaud.
\newblock {\em N\'{e}ron models}, volume~21 of {\em Ergebnisse der Mathematik
  und ihrer Grenzgebiete (3) [Results in Mathematics and Related Areas (3)]}.
\newblock Springer-Verlag, Berlin, 1990.

\bibitem[BMRR08]{BMR08}
Roman Bezrukavnikov, Ivan Mirkovi{\'c}, Dmitriy Rumynin, and Simon Riche.
\newblock Localization of modules for a semisimple Lie algebra in prime
  characteristic.
\newblock {\em Annals of Mathematics}, pages 945--991, 2008.

\bibitem[BNR89]{bnr-spectral-curves}
Arnaud Beauville, M.~S. Narasimhan, and S.~Ramanan.
\newblock Spectral curves and the generalised theta divisor.
\newblock {\em J. Reine Angew. Math.}, 398:169--179, 1989.

\bibitem[Boa18]{boalch2018wild}
Philip Boalch.
\newblock Wild character varieties, meromorphic Hitchin systems and Dynkin
  diagrams.
\newblock {\em Geometry and physics}, 2:433--454, 2018.

\bibitem[Bou02]{bourbaki-lie}
Nicolas Bourbaki.
\newblock {\em Lie groups and {L}ie algebras. {C}hapters 4--6}.
\newblock Elements of Mathematics (Berlin). Springer-Verlag, Berlin, 2002.
\newblock Translated from the 1968 French original by Andrew Pressley.

\bibitem[Bou04]{bourbaki2004lie}
N.~Bourbaki.
\newblock {\em Lie Groups and Lie Algebras: Chapters 7-9}.
\newblock Number v. 3; v. 7 in Actualit{\'e}s scientifiques et industrielles.
  Springer Berlin Heidelberg, 2004.

\bibitem[CGP15]{conrad2015pseudo}
Brian Conrad, Ofer Gabber, and Gopal Prasad.
\newblock {\em Pseudo-reductive groups}, volume~26.
\newblock Cambridge University Press, 2015.

\bibitem[CL16]{chaudouard-laumon}
Pierre-Henri Chaudouard and G\'{e}rard Laumon.
\newblock Un th\'{e}or\`eme du support pour la fibration de {H}itchin.
\newblock {\em Ann. Inst. Fourier (Grenoble)}, 66(2):711--727, 2016.

\bibitem[CZ15]{chen-zhu}
Tsao-Hsien Chen and Xinwen Zhu.
\newblock Non-abelian {H}odge theory for algebraic curves in characteristic
  {$p$}.
\newblock {\em Geom. Funct. Anal.}, 25(6):1706--1733, 2015.

\bibitem[CZ17]{chzh17}
Tsao-Hsien Chen and Xinwen Zhu.
\newblock Geometric Langlands in prime characteristic.
\newblock {\em Compositio Mathematica}, 153(2):395--452, 2017.

\bibitem[Dal17]{dalakov2017lectures}
Peter Dalakov.
\newblock Lectures on {H}iggs moduli and abelianisation.
\newblock {\em Journal of Geometry and Physics}, 118:94--125, 2017.

\bibitem[dC17]{de2017support-sln}
Mark~Andrea de~Cataldo.
\newblock A support theorem for the Hitchin fibration: the case of $SL_n$.
\newblock {\em Compositio Mathematica}, 153(6):1316--1347, 2017.

\bibitem[dC22]{decataldo-cambridge}
Mark Andrea~A. de~Cataldo.
\newblock Perverse {L}eray filtration and specialisation with applications to
  the {H}itchin morphism.
\newblock {\em Math. Proc. Cambridge Philos. Soc.}, 172(2):443--487, 2022.

\bibitem[dCGZ22]{dCGZ}
Mark Andrea~A de~Cataldo, Michael Groechenig, and Siqing Zhang.
\newblock The de Rham stack and the variety of very good splittings of a curve.
\newblock {\em arXiv preprint arXiv:2211.09630, to appear in Math. Res. Lett.},
  2022.

\bibitem[dCH22]{deC-H-2022geometry}
Mark Andrea~A de~Cataldo and Andres~Fernandez Herrero.
\newblock Geometry of the logarithmic Hodge moduli space (with an Appendix
  joint with Siqing Zhang).
\newblock {\em arXiv preprint arXiv:2211.06754}, 2022.

\bibitem[dCHL18]{deCHaiLi}
Mark~Andrea de~Cataldo, Thomas~J Haines, and Li~Li.
\newblock Frobenius semisimplicity for convolution morphisms.
\newblock {\em Mathematische Zeitschrift}, 289(1):119--169, 2018.

\bibitem[dCHM12]{de2012topology}
Mark Andrea~A de~Cataldo, Tam{\'a}s Hausel, and Luca Migliorini.
\newblock Topology of Hitchin systems and Hodge theory of character varieties:
  the case A 1.
\newblock {\em Annals of Mathematics}, pages 1329--1407, 2012.

\bibitem[dCMS22]{de2022abeliansurface}
Mark de~Cataldo, Davesh Maulik, and Junliang Shen.
\newblock Hitchin fibrations, abelian surfaces, and the P= W conjecture.
\newblock {\em Journal of the American Mathematical Society}, 35(3):911--953,
  2022.

\bibitem[dCMSZ23]{dCMSZ}
Mark~Andrea de~Cataldo, Davesh Maulik, Junliang Shen, and Siqing Zhang.
\newblock Cohomology of the moduli of Higgs bundles on a curve via positive
  characteristic.
\newblock {\em Journal of the European Mathematical Society}, pages 1--21,
  2023.

\bibitem[dCZ22]{de-zhang-naht}
Mark~Andrea de~Cataldo and Siqing Zhang.
\newblock A cohomological nonabelian Hodge theorem in positive characteristic.
\newblock {\em Algebraic Geometry}, 9(5), 2022.

\bibitem[Del70]{deligne1970equations}
Pierre Deligne.
\newblock Equations diff{\'e}rentielles {\`a} points singuliers r{\'e}guliers.
\newblock {\em Lecture notes in mathematics}, 163, 1970.

\bibitem[DG71]{ega1-springer}
Jean Dieudonn{\'e} and Alexandre Grothendieck.
\newblock {\em {\'E}l{\'e}ments de g{\'e}om{\'e}trie alg{\'e}brique}, volume
  166.
\newblock Springer Berlin Heidelberg New York, 1971.

\bibitem[DP06]{donagi2006langlands}
Ron Donagi and Tony Pantev.
\newblock Langlands duality for Hitchin systems.
\newblock {\em arXiv preprint math/0604617}, 2006.

\bibitem[DPS24]{donagi2024twistor}
Ron Donagi, Tony Pantev, and Carlos Simpson.
\newblock Twistor Hecke eigensheaves in genus 2.
\newblock {\em arXiv preprint arXiv:2403.17045}, 2024.

\bibitem[EG20]{esnault2020rigid}
H{\'e}lene Esnault and Michael Groechenig.
\newblock Rigid connections and F-isocrystals.
\newblock {\em Acta Math}, 225:103--158, 2020.

\bibitem[Gro16]{groechenig-moduli-flat-connections}
Michael Groechenig.
\newblock Moduli of flat connections in positive characteristic.
\newblock {\em Math. Res. Lett.}, 23(4):989--1047, 2016.

\bibitem[Hab20]{hablicsek2020hodge}
M{\'a}rton Hablicsek.
\newblock Hodge theorem for the logarithmic de Rham complex via derived
  intersections.
\newblock {\em Research in the Mathematical Sciences}, 7:1--21, 2020.

\bibitem[Hau05]{hausel2005mirror}
Tam{\'a}s Hausel.
\newblock Mirror symmetry and Langlands duality in the non-Abelian Hodge theory
  of a curve.
\newblock In {\em Geometric methods in algebra and number theory}, pages
  193--217. Springer, 2005.

\bibitem[Hei17]{heinloth-hilbertmumford}
Jochen Heinloth.
\newblock Hilbert-{M}umford stability on algebraic stacks and applications to
  {$\mathcal G$}-bundles on curves.
\newblock {\em \'{E}pijournal Geom. Alg\'{e}brique}, 1:Art. 11, 37, 2017.

\bibitem[Hit87]{hitchin1987self}
Nigel~J Hitchin.
\newblock The self-duality equations on a Riemann surface.
\newblock {\em Proceedings of the London Mathematical Society}, 3(1):59--126,
  1987.

\bibitem[HKSZ22]{huang2022tame}
Pengfei Huang, Georgios Kydonakis, Hao Sun, and Lutian Zhao.
\newblock Tame parahoric nonabelian Hodge correspondence on curves.
\newblock {\em arXiv preprint arXiv:2205.15475}, 2022.

\bibitem[HMMS22]{hausel2022p}
Tam{\'a}s Hausel, Anton Mellit, Alexandre Minets, and Olivier Schiffmann.
\newblock $ P= W $ via $ H\_2$.
\newblock {\em arXiv preprint arXiv:2209.05429}, 2022.

\bibitem[Hof10]{hoffmann-connected-components}
Norbert Hoffmann.
\newblock On moduli stacks of {$G$}-bundles over a curve.
\newblock In {\em Affine flag manifolds and principal bundles}, Trends Math.,
  pages 155--163. Birkh\"{a}user/Springer Basel AG, Basel, 2010.

\bibitem[Hos23]{hoskins2023two}
Victoria Hoskins.
\newblock TWO PROOFS OF THE P= W CONJECTURE.
\newblock {\em S{\'e}minaire BOURBAKI}, page 76e, 2023.

\bibitem[HT04]{hausel-thaddeus-04}
Tam\'{a}s Hausel and Michael Thaddeus.
\newblock Generators for the cohomology ring of the moduli space of rank 2
  {H}iggs bundles.
\newblock {\em Proc. London Math. Soc. (3)}, 88(3):632--658, 2004.

\bibitem[HX24]{heuer-xu2024padic}
Ben Heuer and Daxin Xu.
\newblock $p$-adic non-abelian Hodge theory for curves via moduli stacks.
\newblock {\em arXiv preprint arXiv:2402.01365}, 2024.

\bibitem[HZ23a]{herrero-zhang-mero}
Andres~Fernandez Herrero and Siqing Zhang.
\newblock Meromorphic Hodge moduli spaces for reductive groups in arbitrary
  characteristic.
\newblock {\em arXiv preprint arXiv:2307.16755}, 2023.

\bibitem[HZ23b]{herrero-zhang-semistable}
Andres~Fernandez Herrero and Siqing Zhang.
\newblock Semistable Non Abelian Hodge theorem in positive characteristic.
\newblock {\em arXiv preprint arXiv:2310.09923}, 2023.

\bibitem[Jan04]{jantzen2004nilpotent}
Jens~Carsten Jantzen.
\newblock Nilpotent orbits in representation theory.
\newblock {\em Lie theory: Lie algebras and representations}, pages 1--211,
  2004.

\bibitem[Kat70]{katz1970}
Nicholas~M Katz.
\newblock Nilpotent connections and the monodromy theorem: Applications of a
  result of Turrittin.
\newblock {\em Publications Math{\'e}matiques de l'IH{\'E}S}, 39:175--232,
  1970.

\bibitem[Kat89]{kato1989logarithmic}
Kazuya Kato.
\newblock Logarithmic structures of Fontaine-Illusie, Algebraic analysis,
  geometry, and number theory (Baltimore, MD, 1988), 1989.

\bibitem[Kos63]{kostant1963lie}
Bertram Kostant.
\newblock Lie group representations on polynomial rings.
\newblock {\em American Journal of Mathematics}, 85(3):327--404, 1963.

\bibitem[Lan14]{langer2014semistable}
Adrian Langer.
\newblock Semistable modules over {L}ie algebroids in positive characteristic.
\newblock {\em Doc. Math.}, 19:509--540, 2014.

\bibitem[Lan15]{langer2015bogomolov}
Adrian Langer.
\newblock Bogomolov’s inequality for Higgs sheaves in positive
  characteristic.
\newblock {\em Inventiones mathematicae}, 199(3):889--920, 2015.

\bibitem[Lan21]{langer2021moduli}
Adrian Langer.
\newblock Moduli spaces of semistable modules over Lie algebroids.
\newblock {\em arXiv preprint arXiv:2107.03128}, 2021.

\bibitem[Lan24]{langer2024bogomolov}
Adrian Langer.
\newblock Bogomolov’s inequality and Higgs sheaves on normal varieties in
  positive characteristic.
\newblock {\em Journal f{\"u}r die reine und angewandte Mathematik (Crelles
  Journal)}, 2024(810):1--48, 2024.

\bibitem[Liu02]{LiuQ}
Qing Liu.
\newblock {\em Algebraic geometry and arithmetic curves}, volume~6.
\newblock Oxford University Press on Demand, 2002.

\bibitem[LMB18]{laumon2018champs}
G{\'e}rard Laumon and Laurent Moret-Bailly.
\newblock {\em Champs alg{\'e}briques}, volume~39.
\newblock Springer, 2018.

\bibitem[LN08]{laumon-ngo2008lemme}
G{\'e}rard Laumon and Bao~Ch{\^a}u Ng{\^o}.
\newblock Le lemme fondamental pour les groupes unitaires.
\newblock {\em Annals of mathematics}, pages 477--573, 2008.

\bibitem[LS24]{li2024tame}
Mao Li and Hao Sun.
\newblock Tame Parahoric Nonabelian Hodge Correspondence in Positive
  Characteristic over Algebraic Curves, 2024.

\bibitem[McN02]{mcninch2002abelian}
George~J McNinch.
\newblock Abelian unipotent subgroups of reductive groups.
\newblock {\em Journal of Pure and Applied Algebra}, 167(2-3):269--300, 2002.

\bibitem[McN03]{mcninch2003sub}
George~J McNinch.
\newblock Sub-principal homomorphisms in positive characteristic.
\newblock {\em Mathematische Zeitschrift}, 244:433--455, 2003.

\bibitem[Mil17]{milne2017algebraic}
James~S Milne.
\newblock {\em Algebraic groups: the theory of group schemes of finite type
  over a field}, volume 170.
\newblock Cambridge University Press, 2017.

\bibitem[Moc04]{mochizuki2004kobayashi}
Takuro Mochizuki.
\newblock Kobayashi-Hitchin correspondence for tame harmonic bundles and an
  application.
\newblock {\em arXiv preprint math/0411300}, 2004.

\bibitem[MRV19]{MRV2019fourier}
Margarida Melo, Antonio Rapagnetta, and Filippo Viviani.
\newblock Fourier--Mukai and autoduality for compactified Jacobians. I.
\newblock {\em Journal f{\"u}r die reine und angewandte Mathematik (Crelles
  Journal)}, 2019(755):1--65, 2019.

\bibitem[MS22]{manikandan2022criterion}
S~Manikandan and Anoop Singh.
\newblock A criterion for the existence of logarithmic connections on curves
  over a perfect field.
\newblock {\em Indian Journal of Pure and Applied Mathematics}, 53(2):330--339,
  2022.

\bibitem[MS23]{maulik-shen-chi}
Davesh Maulik and Junliang Shen.
\newblock Cohomological $\chi$--independence for moduli of one-dimensional
  sheaves and moduli of Higgs bundles.
\newblock {\em Geometry \& Topology}, 27(4):1539--1586, 2023.

\bibitem[MS24]{maulik2024p}
Davesh Maulik and Junliang Shen.
\newblock The P=W conjecture for GL\_n.
\newblock {\em Annals of Mathematics}, 200(2):529--556, 2024.

\bibitem[Ng{\^o}06]{ngo2006fibration}
Bao~Ch{\^a}u Ng{\^o}.
\newblock Fibration de Hitchin et endoscopie.
\newblock {\em Inventiones mathematicae}, 164:399--453, 2006.

\bibitem[Ng{\^o}10]{ngo-lemme-fondamental}
Bao~Ch{\^a}u Ng{\^o}.
\newblock Le lemme fondamental pour les alg\`ebres de {L}ie.
\newblock {\em Publ. Math. Inst. Hautes \'{E}tudes Sci.}, (111):1--169, 2010.

\bibitem[Ogu22]{ogus2022crystalline}
Arthur Ogus.
\newblock Crystalline prisms: Reflections on the present and past.
\newblock {\em arXiv preprint arXiv:2204.06621}, 2022.

\bibitem[OV07]{ov07}
Arthur Ogus and Vadim Vologodsky.
\newblock Nonabelian Hodge theory in characteristic p.
\newblock {\em Publications math{\'e}matiques}, 106(1):1--138, 2007.

\bibitem[Ric17]{riche2017kostant}
Simon Riche.
\newblock Kostant section, universal centralizer, and a modular derived Satake
  equivalence.
\newblock {\em Mathematische Zeitschrift}, 286(1):223--261, 2017.

\bibitem[Sch98]{schaub-courbes-spectrales}
Daniel Schaub.
\newblock Courbes spectrales et compactifications de jacobiennes.
\newblock {\em Math. Z.}, 227(2):295--312, 1998.

\bibitem[Sch05]{schepler2005logarithmic}
Daniel~Kenneth Schepler.
\newblock {\em Logarithmic nonabelian Hodge theory in characteristic p}.
\newblock University of California, Berkeley, 2005.

\bibitem[Ser96]{serre1996exemples}
Jean-Pierre Serre.
\newblock Exemples de plongements des groupes PSL2 (Fp) dans des groupes de Lie
  simples.
\newblock {\em Inventiones mathematicae}, 124:525--562, 1996.

\bibitem[Ser05]{serre2005complete}
J~Serre.
\newblock Compl{\`e}te r{\'e}ductibilit{\'e}.
\newblock {\em ASTERISQUE-SOCIETE MATHEMATIQUE DE FRANCE}, 299:195, 2005.

\bibitem[Ses77]{seshadri1977geometric}
Conjeevaram~S Seshadri.
\newblock Geometric reductivity over arbitrary base.
\newblock {\em Advances in Mathematics}, 26(3):225--274, 1977.

\bibitem[She24]{shen2018tamely}
Shiyu Shen.
\newblock {Tamely Ramified Geometric Langlands Correspondence in Positive
  Characteristic}.
\newblock {\em International Mathematics Research Notices}, page rnae005, 02
  2024.

\bibitem[Sim90]{simpson1990harmonic}
Carlos~T Simpson.
\newblock Harmonic bundles on noncompact curves.
\newblock {\em Journal of the American Mathematical Society}, 3(3):713--770,
  1990.

\bibitem[Sim94a]{Simpson-repnI}
Carlos~T. Simpson.
\newblock Moduli of representations of the fundamental group of a smooth
  projective variety. {I}.
\newblock {\em Inst. Hautes \'{E}tudes Sci. Publ. Math.}, (79):47--129, 1994.

\bibitem[Sim94b]{simpson-repnII}
Carlos~T. Simpson.
\newblock Moduli of representations of the fundamental group of a smooth
  projective variety. {II}.
\newblock {\em Inst. Hautes \'{E}tudes Sci. Publ. Math.}, (80):5--79 (1995),
  1994.

\bibitem[Sob15]{sobaje2015springer}
Paul Sobaje.
\newblock Springer isomorphisms in characteristic p.
\newblock {\em Transformation Groups}, 20:1141--1153, 2015.

\bibitem[Spr98]{springer1998linear}
Tonny~Albert Springer.
\newblock {\em Linear algebraic groups}, volume~9.
\newblock Springer, 1998.

\bibitem[{Sta}]{stacks-project}
The {Stacks Project Authors}.
\newblock \textit{Stacks Project}.
\newblock \url{https://stacks.math.columbia.edu}.

\bibitem[Tes95]{testerman1995a_1}
Donna Testerman.
\newblock $ A\_1 $-type overgroups of elements of order $ p $ in semisimple
  algebraic groups and the associated finite groups.
\newblock {\em Journal of Algebra}, 177(1):34--76, 1995.

\end{thebibliography}

 \textsc{Mark Andrea de Cataldo, Department of Mathematics, Stony Brook University}\par\nopagebreak
  \texttt{mark.decataldo@stonybrook.edu}

\textsc{Siqing Zhang, School of Mathematics, Institute for Advanced Study}\par\nopagebreak
  \texttt{szhang@ias.edu}

\end{document}